\documentclass{birkjour}
\usepackage{amssymb}
 \newtheorem{thm}{Theorem}[section]
 \newtheorem{thmm}{Theorem}
 \newtheorem{thmmm}{Theorem} 
 \newtheorem{cor}[thm]{Corollary}
 \newtheorem{lem}[thm]{Lemma}
 \newtheorem{prop}[thm]{Proposition}
 \theoremstyle{definition}
 \newtheorem{defn}[thm]{Definition}
 \theoremstyle{remark}

 \numberwithin{equation}{section}

\newcommand{\ee}{\varepsilon}
\newcommand{\R}{\mathbb{R}}

\newcommand{\Ee}{\mathbb{E}}
\newcommand{\N}{\mathbb{N}}
\newcommand{\Z}{\mathbb{Z}}
\newcommand{\M}{\mathbb{M}}
\newcommand{\T}{\mathbb{T}}
\newcommand{\X}{\mathbb{X}}
\newcommand{\Hm}{\mathbb{H}}
\newcommand{\mS}{\mathbb{S}}
\newcommand{\PP}{\mathbb{P}}
\newcommand{\Rn}{\mathbb{R}^{2n}}
\newcommand{\Tn}{\mathbb{T}^{2n}}
\newcommand{\Zn}{\mathbb{Z}^{2n}}

\newcommand{\Hc}{\mathcal{H}}
\newcommand{\J}{\mathcal{J}}
\newcommand{\Lc}{\mathcal{L}}
\newcommand{\Kc}{\mathcal{K}}
\newcommand{\MM}{\mathcal{M}}
\newcommand{\Pc}{\mathcal{P}}
\newcommand{\V}{\mathcal{V}}
\newcommand{\WW}{\mathcal{W}}

\newcommand{\Ah}{\nabla_{_{1/2}} A_H} 
\newcommand{\AH}{\nabla^2_{_{1/2}} A_H}
\newcommand{\dd}{\overline \p_{J_0}} 
\newcommand{\dHH}{\overline \p_{J,H}}  
\newcommand{\dH}{\overline \p_{J_0,H}}
\newcommand{\ffi}{\varphi_{X}} 
\newcommand{\HH}{H^{1/2}(\mS^1)}
\newcommand{\HR}{H^{1/2}(\mS^1,\R^{2n})}
\newcommand{\HT}{H^{1/2}(\mS^1,\T^{2n})}
\newcommand{\HS}{H^{1/2}(\mS^1)}
\newcommand{\LS}{L^{2}(\mS^1)} 
\newcommand{\LLR}{L^2(\mS^1,\R^{2n})}

\newcommand{\LRR}{L^{2}(\R\times \mS^1,\Rn)} 
\newcommand{\LZR}{L^2(Z,\R^{2n})}
\newcommand{\WRR}{H^{1}(\R\times \mS^1,\Rn)}
\newcommand{\WRRl}{H^{1}_{\loc}(\R\times \mS^1,\Rn)}
\newcommand{\WRTl}{H^{1}_{\loc}(\R\times \mS^1,\T^{2n})}
\newcommand{\WZR}{H^{1}(Z,\R^{2n})}
\newcommand{\WZTl}{H^{1}_{\loc}(Z,\T^{2n})}
\newcommand{\CoR}{C^{\infty}_0([0,\infty)\times \mS^1,\R^{2n})} 
\newcommand{\CC}{C^{\infty}(\R\times \mS^1,\T^{2n})}
\newcommand{\Txx}{\Theta_{x^-,c_{x^+}}} 
\newcommand{\ffiXK}{\varphi_{\X_K}}

\newcommand{\lo}{\longrightarrow}
\newcommand{\rlo}{\rightarrow}
\newcommand{\p}{\partial}
\newcommand{\li}{\left}
\newcommand{\re}{\right}
\newcommand{\nl}{||\,}
\newcommand{\nr}{\,||}
\newcommand{\mi}{\,\,\big|\,\,}

\DeclareMathOperator{\coker}{\mathrm{coker}}
\DeclareMathOperator{\ran}{\mathrm{ran}}
\DeclareMathOperator{\graph}{\mathrm{graph}}
\DeclareMathOperator{\rest}{\mathrm{sing}}
\DeclareMathOperator{\crit}{\mathrm{crit}}
\DeclareMathOperator{\loc}{\mathrm{loc}}
\DeclareMathOperator{\hyb}{\mathrm{hyb}}
\DeclareMathOperator{\spec}{\mathrm{spec}}
\DeclareMathOperator{\ind}{\mathrm{ind}}
\DeclareMathOperator{\reg}{\mathrm{reg}}
\DeclareMathOperator{\Hr}{\mathcal{H}_{\mathrm{reg}}}
\DeclareMathOperator{\Image}{\mathrm{Im}}
\DeclareMathOperator{\Gr}{\mathrm{Gr}}
\DeclareMathOperator{\dist}{\mathrm{dist}}
\DeclareMathOperator{\codim}{\mathrm{codim}}
\DeclareMathOperator{\ev}{\mathrm{ev}}
\DeclareMathOperator{\Sp}{\mathrm{Sp}}
\DeclareMathOperator{\Id}{\mathrm{Id}}

\DeclareMathOperator{\Hrr}{\mathcal{H}_{\mathrm{reg}}}
\DeclareMathOperator{\supp}{\mathrm{supp}}

\renewcommand{\labelenumi}{$(\roman{enumi})$} 

\begin{document}

%-------------------------------------------------------------------------
% editorial commands: to be inserted by the editorial office
%
%\firstpage{1} \volume{228} \Copyrightyear{2004} \DOI{003-0001}
%\seriesextraline{This is the Concrete Title of this Book\br H.E. R and S.T.C. W, Eds.}
%\firstpage{1}

\title{Isomorphic chain complexes of Hamiltonian dynamics on tori}

%----------Author 1
\author[Michael Hecht]{Michael Hecht}

\address{%
Karl-Heine-Str.61 \\
04229 Leipzig \\
Germany}

\email{hecht@math.uni-leipzig.de}
\email{michael@bioinf.uni-leipzig.de}
%\thanks{This work was completed with the support of our
%----------Author 2
%\author{A Second Author}
%\address{The address of\br
%the second author\br
%sitting somewhere\br
%in the world}
%\email{dont@know.who.knows}
% \subjclass{Floer homology 57R58; Morse-Smale systems 37D15; Chain complexes 18G35}


\date{April 18, 2012}
\dedicatory{I dedicate this work to Yvonne Choquet-Bruhat.}
%%% ----------------------------------------------------------------------
 
\begin{abstract} In this work we construct for a given smooth,  generic Hamiltonian $H : \mS^1 \times \Tn \lo \R$ on the torus $\Tn = \Rn/\Z^{2n}$ a chain isomorphism 
$ \Phi_* : \big(C_*(H),\p^M_*\big) \lo \big(C_*(H),\p^F_*\big)$ between the Morse complex of the Hamiltonian action $A_H$ on the free loop space of the torus 
$\Lambda_0(\Tn)$ and the Floer complex. 
Though both complexes are generated by the critical points of $A_H$, their boundary operators differ. Therefore the construction of $\Phi$ is based on counting the 
moduli spaces of hybrid type solutions which involves stating a new non-Lagrangian boundary value problem for Cauchy-Riemann type operators not yet studied in 
Floer theory.  We finally want to note that the problem is completely symmetric. So we also could construct an isomorphism 
$\Psi_* :  \big(C_*(H),\p^F_*\big) \lo \big(C_*(H),\p^M_*\big)\,.$
% The essential reason for the non-triviality of the result is that one has to develop  new methods to prove that this new coupling condition leads to a 
% well posed Fredholm problem and to the fact that the moduli spaces are precompact. It is crucial for the statement that the torus is compact, possesses trivial tangent 
% bundle and an additive structure.
\end{abstract}

%%% ----------------------------------------------------------------------
\maketitle
%%% ----------------------------------------------------------------------
%\tableofcontents
\section{Introduction}
For a generic smooth Hamiltonian $H :\mS^1 \times \Tn \lo \R $ with only contractible periodic orbits and standard complex structure 
\begin{equation*}
J_0 =\li ( \begin{array}{rc}
 0 & 1 \\
-1 & 0
\end{array}\re)
\end{equation*}
we have proven the following statement. 
\begin{thmm}[Main Theorem]
 The Morse complex of the Hamiltonian action $A_H$ is chain isomorphic to the Floer complex of $(H,J_0)$. 
\end{thmm}
To give a short summary observe that the component of contractible loops of the free loop space of the $2n$ - dimensional torus $(\Tn = \Rn/\Z^{2n},\omega_0)$ 
equipped with the standard symplectic structure $\omega_0 = \sum_{i=1}^ndy_i\wedge dx_i$ can be completed to a Hilbert manifold $\M$ with respect to the 
$H^{1/2}$- inner product. In this case $\M$ splits into $\M = \Tn \times \Hm$, where $\Hm$ is a Hilbert space and each $x \in \M$ can be written as 
\begin{equation*}
 x  = [x_0] + \sum_{k \not = 0}e^{2\pi J_0kt}x_k \,, \quad [x_0] \in \Tn \,, x_k \in \Rn \quad \forall \, k \in \Z\setminus\{0\}\,.
\end{equation*}
Note that $J_0$ is the complex structure  on $\Tn$ such that $(\Tn,\omega_0,J_0)$ becomes a K\"ahler manifold with 
\begin{equation*}
\omega_0(\cdot,J_0\cdot) = \li<\cdot,\cdot\re>\,. 
\end{equation*}
For given 1-periodic Hamiltonian $H : \mS^1 \times \Tn \lo \R$ the Hamiltonian action $A_H : \M \lo \R $ is defined by 
\begin{equation*}
A_H(x) =  -\frac{1}{2}\li(\nl \PP^+x \nr^2_{\HH} - \nl \PP^-x \nr^2_{\HH}\re) +\int_0^1 H\big(t,x(t)\big)dt 
\end{equation*}
where $\PP^{\pm}: \M \lo \Hm^{\pm}\,,  \sum \limits_{k\in \Z}e^{2\pi J_0kt}x_k \mapsto  \sum \limits_{\pm k>0}e^{2\pi J_0kt}x_k $ denote the projections 
onto the loops all of whose frequencies are positive or negative respectively. The critical points of $A_H$ are precisely the 1-periodic contractible solutions $\Pc_0(H)$ of the Hamiltonian flow 
$\varphi_{X_H} : \R\times \Tn \lo \Tn$ determined by the equation 
\begin{equation*}
 \frac{d}{dt} \varphi_{X_H} = X_H \circ \varphi_{X_H} \,, \quad  \varphi_{X_H}(0,\cdot) = \mathrm{id}_{\Tn} \,,\quad \text{with} \quad i_{X_H}\omega_0 = -dH \,. 
\end{equation*}
Now there are several possibilities to develop a Morse theory for this action functional, see for instance \cite{ConZeh}, \cite{amann}, \cite{rabino}, \cite{rabino2}.
We want to focus on two other ideas. The first one goes back to the results of A. Abbondandolo and P. Majer published in  \cite{Al5} and  \cite{Al1} and uses 
the fact that $A_H$ possesses a smooth gradient $\X:= -\Ah$ on $\M$ to build up a theory  directly on the loop space. The second one is known as A. Floer's approach 
\cite{floer3} where solutions of a type of Cauchy Riemann PDE's are considered, i.e., the solutions 
$$ u : \R\times \mS^1 \lo \Tn $$ 
of the  Floer equation
\begin{equation*}
\dHH \big(u(s,t)\big) = \p_su(s,t) +J_t(u(s,t))\big(\p_t u(s,t) -X_H(t,u(s,t))\big) = 0 \,, \quad J_t \in \J
\end{equation*}
where $\J$ denotes the set of all smooth and 1-periodic almost complex structures on $(\Tn,\omega_0)$ which are compatible with $\omega_0$. In our case we have to 
model these equations in the $H^1$-setup, therefore we find it convenient to restrict ourselves to the case where $J =J_0$ is the constant standard complex structure.
To achieve transversality we asssume that the Hamiltonian $H$ was generically chosen. Hence we can count the solutions with finite energy which therefore converge to 
critical points of $A_H$ at the asymptotics and define the Floer boundary operator   
\begin{equation*}
 \p^F_k  : C_k(H) \lo C_{k-1}(H)\,,  
\end{equation*}
between the free abelian groups generated by the critical points of $A_H$ with Conley-Zehnder index $k \in \Z$. According  to Floer's fundamental theorem the above 
setup becomes a complex, i.e.,  $\p^F_{k-1} \circ \p^F_k = 0$. The homology of this complex is called the {\bf Floer homology } and it is known due to Floer's continuation 
theorem that this homology is independent of the chosen Hamiltonian $H$. In particular this asserts why we can assume that there are only contractible 1-periodic 
Hamiltonian orbits and can therefore restrict ourselves to the component of contractible loops of the free loop space $\Lambda_0(\Tn)$ as introduced. Furthermore 
the continuation property was used by A. Floer to prove that the Floer homology is chain isomorphic to the singular homology of the symplectic manifold, i.e., in our 
case the torus $\Tn$. This fact is the main ingredient in Floer's proof of the {\bf Arnold conjecture} in the generic situation \cite{floer3}.\\

Coming back to the first approach the fact that $A_H$ is strongly indefinite, i.e., possesses infinite Morse indices and co-indices for all singular points 
$x,y \in \rest(\X)$ leads to the problem that the unstable and stable manifolds $\WW^u(x),\WW^s(x)$ are infinite dimensional submanifolds of $\M$. 
Nevertheless, one can hope, that the intersections $\WW^u(x)\cap\WW^s(y) $ are finite dimensional. To prove such a result A. Abbondandolo and P. Majer 
require the existence of a subbundle $\V \subset T\M = \M \times \Ee$, $\Ee = \Rn \times \Hm$ which in our situation can be chosen as 
$\V = \M \times \big(\R^n\times\Hm^+\big)$ such that the Morse vector field $\X$ satisfies two conditions, namely : \\

{\bf (C1)} for every singular point $x \in \rest (\X)$, the unstable eigenspace $E^u\li(D\X(x)\re)$ of the Jacobian  of $\X$ at $x$ is a compact perturbation of 
$\V(x)$; meaning that the corresponding orthogonal projections $P_{E^u}:=P_{E^u\li(D\X(x)\re)}$ and $P_V$ with $V = \R^n\times \Hm^+$ have compact difference, i.e.,
$P_{E^u} - P_V$ is a compact operator. \\

{\bf (C2)} for all $p\in \M$ the operator $\big[ D\X(p),P_V \big]$ is compact, where $[S,T]=ST-TS$, $\forall$ $S,T \in \Lc(\Ee)$ denotes the commutator. \\  

Under condition {\bf C1} and a strengthened, global version of {\bf C2}, A. Abbondandolo and P. Majer could prove in \cite{Al5} that if the intersections 
$\WW^u(x)\cap\WW^s(y)$, $x,y \in \rest(\X)$ are transverse, then these are finite dimensional manifolds which are compact up to broken trajectories. Furthermore
A. Abbondandolo and P. Majer defined a relative Morse index with respect to $\V$, which in our situation due to our special choice of $\V$ coincides with the Conley-Zehnder index. 
Therefore 
we are almost in the same situation as in the finite dimensional case treated in \cite{schwarz-2} and can built up a Morse complex for the action functional $A_H$. 
That is to consider again the free abelian groups $C_k(H)$, generated by the critical points $x\in \crit(A_H) $ of $A_H$ with Conley-Zehnder index $\mu(x) = k \in \Z$ and to 
define the boundary operator 
\begin{equation*}
 \p^M_k : C_k(H) \lo C_{k-1}(H) 
\end{equation*}
in the usual way by counting the connected components of $\WW^u(x)\cap\WW^s(y)$ with $\mu(x) = k$, $\mu(y) = k-1$. The fact that $\p^M_{k-1}\circ \p^M_k  = 0$ is 
again deep and proven in \cite{Al5}. So we obtain a homology called the {\bf Morse homology} of $A_H$. Though we have to perturb the gradient of $A_H$ 
by a compact vector field $K : \M \lo \Ee$ to achieve transversality, the above homology is independent of the perturbation as well as of the particular choice of $H$. 
So again the theory can be treated in the situation of contractible critical points of $A_H$. \\

After introducing both complexes we now want to give a sketch of the construction of the isomorphism. The main ingredients are the {\bf moduli spaces of hybrid 
type curves}. For given critical points $x^-,x^+ \in \crit(A_H)$, a $C^3$- vector field $X$  on $\M$ with globally defined flow  
that are the spaces  
\begin{align*}
 \MM_{\hyb}(x^-,x^+,H,J_0,X):= \Big \{ & u \in  H^1_{\loc}(Z,\Tn) \mi   \dH (u) = 0 \,, \\ 
                                       & u(0,\cdot)\in \WW^u_X(x^-)\,, u(+\infty,\cdot) = x^+\Big\}\,,
\end{align*}
where $Z=[0,+\infty)\times \mS^1$ denotes the half cylinder.  
Though the solutions are smooth on the interior $(0,+\infty)\times\mS^1$ the equation $\dH (u) = 0$ is understood in the weak sense. The fact that the 
evaluation at zero is a well defined smooth submersion implies that the boundary condition $u(0,\cdot)\in \WW^u_X(x^-)$ which says that the loop 
$u(0,\cdot) \in \M$ shall sit on the unstable manifold of $x^-$ with respect to $X$ is well posed. The new outcome of this thesis is that this non-Lagrangian 
boundary condition leads to a well posed Fredholm problem and to the fact that the moduli spaces are compact up to broken trajectories in the 
$H^1_{\loc}$-sense as long as $X$ is of the form $-\Ah + K$ where $K : \M \lo \Ee$ plays the role of a compact perturbation. \\

Both results in this robust theory use new estimates for the linearized operator and the connecting curves and again the existence of the subbundle $\V$ which 
is admissible for $X$ in the sense of conditions {\bf C1} and {\bf C2}. As an outlook one hopefully can extract these methods for the analogous problem on more general 
manifolds as the torus $\Tn$. \\
   
To finish the construction of the isomorphism we define the following map on the abelian groups 
$$C_k(H) = \bigoplus_{\genfrac{}{}{0pt}{}{x \in \Pc_0(H)\,,}{\mu(x)=k }}\Z_2 x \,. $$
$$ \Phi_k : C_k(H)\lo C_k(H) \,,\quad  x \mapsto \sum_{\mu(y) = k} \upsilon(x,y)y  $$  
where $\upsilon(x,y)$ denotes  the sum of connected components of the moduli space  $\MM_{\hyb}(x,y,H,J_0,X)$. By a standard argument using the gluing method $\Phi$ is a chain 
homomorphism. Finally, we order the critical points of $A_H$ by their levels, then $\Phi_k$ gets the form of an upper triangular matrix with $1$'s on the diagonal
and therefore $\Phi$ is an isomorphism as claimed.

\section{The Morse complex}\label{Morse}

In this chapter we construct the Morse complex for the Hamiltonian action functional $A_H$  on the 
component of contractible loops of the  free loop space  $ \Lambda_0(\Tn)$ of the torus $\Tn = \Rn /\Z^{2n}$.
To improve the readability of this paper some proofs of already known statements adapted to our particular setting
are explicitly given.
 
\subsection{The analytical setting}

This section shows how to complete the component of contractible loops of the free loop space of the torus to a Hilbert manifold 
whose structure is induced by the special Sobolev-space $\HR$. Furthermore we show that our choice of the $H^{1/2}$-setup matches perfectly for the Hamiltonian 
action $A_H$  in the sense that its gradient with respect to the $H^{1/2}$-inner product is well defined and smooth. 
Most facts are already discussed in \cite{hof}. Therefore we restrict ourselves to the essential statements.\\

We consider the $2n$-dimensional torus $\Tn$, $n\in\N$, as the quotient $\Tn = \Rn /\Z^{2n}$ with the standard symplectic structure 
$\omega_0 = \sum \limits_{i=1}^n dy_i\wedge dx_i$ induced by the standard symplectic basis $\li\{\p x_i,\p y_i\re\}_{i =1,...,n}$ on $\Rn$
and denote by 
\begin{equation*}
J_0 =\li ( \begin{array}{rc}
 0 & 1 \\
-1 & 0
\end{array}\re)
\end{equation*}
the special complex structure on $\Tn$  such that $(\Tn,\omega_0,J_0)$ becomes a K\"ahler manifold with 
\begin{equation*}
\omega_0(\cdot,J_0\cdot) = \li<\cdot,\cdot\re>\,. 
\end{equation*}
We introduce a smooth, 1-periodic Hamiltonian $H : \mS^1 \times \Tn \lo \R $, its Hamiltonian vector field $X_H$ 
defined by 
\begin{equation*}
 i_{X_H}\omega_0 = -dH  
\end{equation*}
and the corresponding smooth flow $\varphi_{X_H}$, which since $ \Tn$ is compact, is globally defined by requiring that  $\varphi_{X_H} : \R \times \Tn \lo \Tn$ solves
\begin{equation*}
 \frac{d}{dt} \varphi_{X_H}(t,p) = \big(X_H \circ \varphi_{X_H}\big)(t,p)\,, \quad \varphi_{X_H}(0,\cdot) = \mathrm{id_{\Tn}}\,. 
\end{equation*}
It is well-known that the search for contractible 1-periodic orbits of $\varphi_{X_H}$ can be reformulated as a variational problem as follows. 
Consider the {\bf component of contractible loops of the free loop space of the torus}. That is 
\begin{equation*}
 \Lambda_0(\Tn):=\li\{x\in C^{\infty}(\mS^1,\Tn) \mi [x] = 0 \quad \text{in} \quad \pi_1(\Tn) \re \}\,,
\end{equation*}
where $\pi_1(\Tn)$ denotes the first fundamental group. To introduce the {\bf Hamiltonian action functional} on $\Lambda_0(\Tn)$ we lift 
the contractible loop $x \in \Lambda_0(\Tn) $ to a closed curve $\tilde x \in C^{\infty}(\mS^1,\Rn)$ and lift $H$ to a function  
$\tilde H$ that is 1-periodic in all variables on $\mS^1\times \Rn$ and set 
\begin{equation}\label{tildeAH}
 \tilde A_H (x) := \frac{1}{2}\int_0^1 \omega_0 \big( \dot{\tilde{x}}(t), \tilde x(t)\big)dt + 
 \int_0^1 \tilde H\big(t, \tilde x(t)\big)dt \,, \quad x \in \Lambda_0(\Tn)\,.
\end{equation}
Now by partial integration the first term is independent of the chosen lift of $x$ and by the periodicity of $\tilde H$ so is the second. Hence $\tilde A_H$ is 
a well defined functional and if there is no danger of  confusion we will make no notational difference between $x$ and its lift. 
Varying $\tilde A_H$ at $x \in \Lambda_0(\Tn)$ in direction of a vector field  $Y$ along $x$ yields in
\begin{equation*}
 \delta \tilde A_H(x)[Y] = \int_0^1 \li< J_0 \,\dot x(t) +  \nabla H\big(x(t)\big),Y(t)\re>dt \,, 
\end{equation*}
which readily shows the one to one correspondence between critical points of $\tilde A_H$ and the contractible 1-periodic orbits of $\varphi_{X_H}$, which we denote 
by $\Pc_0(H)$. \\

We complete $\Lambda_0(\Tn)$ to a Hilbert manifold as follows. Since $\Lambda_0(\Rn) \subset \LLR$ each curve $x \in \Lambda_0(\Rn)$ can be written as a Fourier 
series with coefficients in $\Rn$
\begin{equation*}
 x(t) = \sum_{k\in\Z} e^{2\pi k J_0 t}x_k \,, \quad x_k \in \Rn\,. 
\end{equation*}
Furthermore the $\Zn$-action acts only on the constants. Therefore  we obtain by identifying those curves whose image on the torus is the same 
\begin{equation*}
 x(t)= \big(x_0,\hat x(t)\big):=  [x_0] + \sum_{k \not = 0} e^{2\pi k J_0 t}x_k \,, 
\quad [x_0] \in \Tn\,, \quad x_k \in \Rn\,.
\end{equation*}
This special Fourier representation allows us to complete the component of contractible loops of the free loop space with respect to a Hilbert structure induced 
by the following {\bf fractional Sobolev spaces}. For $s\geq 0$ we set
\begin{equation*}
 H^s(\mS^1,\Rn) = \Big\{x \in L^2(\mS^1,\Rn) \mi  \sum_{k \not = 0} |k|^{2s}|x_k|^2 < \infty \Big\}\,.  
\end{equation*}
The spaces $H^s(\mS^1,\Rn)$ are Hilbert spaces with inner product and associated norm 
\begin{align*}
\li<x,y\re>_{H^s(\mS^1)} & := \li<x_0,y_0\re> + 2\pi\sum_{k\in \Z}|k|^{2s} \li<x_k,y_k\re> \\
\nl x \nr_{H^s(\mS^1)}^2 &  = \li<x,x\re>_{H^s(\mS^1)}
\end{align*}
We are especially interested in the case $s=1/2$ and the orthogonal splitting 
\begin{equation}\label{split}
\HR = \Rn \times \Hm^+\times \Hm^-=:\Ee\,, \quad \Hm:= \Hm^+\times \Hm^- 
\end{equation}
into the constants and the loops all of whose frequencies are positive or negative respectively given by the corresponding orthogonal projections 
$\PP^{\pm} : \Hm \lo \Hm^{\pm}$, $\PP^0 : \Ee \lo \Rn$  with 
\begin{equation}\label{pro}
 \PP^{\pm}(x) = \sum_{0  < \pm k}e^{2\pi kJ_0t}x_k\,, \quad \PP^0(x) = x_0 \,. 
\end{equation}
The space $\Hm^+$ is a special {\bf Hardy-space} \cite{hardy} on the unit circle, i.e., the space of limits of holomorphic curves on the open 
disc to the boundary and $\Hm^-$ is its anti-holomorphic counter-part. We set 
\begin{equation}\label{ht}
\M:= \HT= \Tn \times \Hm \,,
\end{equation}
then $\M$ carries the structure of a Hilbert manifold with trivial tangent bundle $T\M =\M \times \Ee$.
We often abuse notation slightly by still writing $\PP^{\pm} :\M \lo \Hm^{\pm}$, $\PP^0 : \M \lo \Tn$ for the maps 
$x = (x_0,\hat x) \mapsto \PP^{\pm}\hat x$ and $x = (x_0,\hat x) \mapsto x_0$. 
Furthermore we can extand $\tilde A_H$ to a functional defined on $\M$, by setting
\begin{equation}\label{AH}
 A_H(x) = -\frac{1}{2}\li(\nl \PP^+ x \nr^2_{\HH} - \nl \PP^- x \nr^2_{\HH}\re) +\int_0^1 H\big(t,x(t)\big)dt \,.
\end{equation}
Before we go on we need a better understanding of the spaces $H^s(\mS^1,\Rn)$. Therefore we follow H. Hofer and E. Zehnder  \cite{hof} and give a brief summary 
of the essential properties. \\

For $t \geq s \geq 0$ the spaces decrease,
$$ H^t(\mS^1,\Rn)\subset H^s(\mS^1,\Rn) \subset H^0(\mS^1,\Rn) = L^2(\mS^1,\Rn) \,, $$
while the norms increase : 
$$\nl x\nr^2_{H^t(\mS^1)} \geq \nl x\nr^2_{H^s(\mS^1)} \geq \nl x\nr^2_{L^2(\mS^1)} \quad \text{for} \quad x \in   H^t(\mS^1,\Rn) \,.$$ 
In particular the inclusion maps  $H^t(\mS^1,\Rn) \lo H^s(\mS^1,\Rn)$ are continuous for $t\geq s$; if in particular $t>s\geq 0$, then as proved in \cite{hof} they are 
moreover compact operators. Furthermore we obtain:
\begin{prop}\label{hof1}{\rm(\cite{hof})} Let $s>\frac{1}{2}$. If $x \in  H^s(\mS^1,\Rn)$ then $x \in C^0(\mS^1,\Rn)$. Moreover there is a constant 
$c=c_s$ such that 
$$ \nl x \nr_{C^0(\mS^1)} \leq c\nl x \nr_{H^{s}(\mS^1)}\,, \quad \text{for all} \quad x \in  H^s(\mS^1,\Rn)\,.$$ 
\end{prop}
Since we choose the $H^{1/2}$-structure, this leads us to the borderline case of embeddings into $C^0$. 
Nevertheless, since $\Tn$ possesses a trivial tangent bundle, \eqref{ht} provides us with a Hilbert manifold structure which will be sufficient for our purpose. 
For manifolds with non-trivial tangent bundle one needs pointwise control on the orbits and therefore at least a $H^s$, $s>1/2$ setting to equip the loop space  
with a Banach manifold structure.\\

A central role in further discussions is played by the inclusion 
\begin{equation}\label{j}
 j : \HR \lo \LLR
\end{equation}
and its adjoint $j^* : \LLR \lo \HR$ 
defined as usual by 
\begin{equation*}
\li<j(x),y\re>_{\LLR} = \li<x,j^*(y)\re>_{\HR} \,,
\end{equation*}
for all $x \in \HR,y\in \LLR$. Note that as proved in  \cite{hof}  $j$ is compact and so is $j^*$ as it is well-known. In particular we obtain even more.
\begin{prop}\label{jj}{\rm (\cite{hof})} The adjoint factors 
\begin{equation*}
j^*: \LLR \lo H^1(\mS^1,\Rn) \lo \HR\,,
\end{equation*}
i.e.,
\begin{equation*}
j^*(\LLR) \subset H^1(\mS^1,\Rn) \quad \text{and} \quad \nl j^*(y) \nr _{H^1(\mS^1)} \leq \nl y \nr _{\LS} \,. 
\end{equation*}
\end{prop}
\begin{proof} Let $x \in H^{1/2}(\mS^1,\Rn) \subset L^2(\mS^1,\Rn)$, $y \in L^2(\mS^1,\Rn)$. Then 
\begin{equation*}
\sum_{k\in\Z} \li<x_k,y_k\re> = \li<x_0,j^*(y)_0\re> + 2\pi \sum_{k\in\Z} |k|\li<x_k,j^*(y)_k\re> \,. 
\end{equation*}
So we get the following formula for $j^*$
\begin{equation}\label{jstar}
 j^*(y) = y_0 + \sum_{k \not = 0}\frac{1}{2\pi|k|}e^{2\pi k J_0t}y_k \,. 
\end{equation}
The estimate  $\nl j^*(y) \nr _{H^1(\mS^1)} \leq \nl y \nr _{\LS}$ is now obvious.
\end{proof}

We consider the second part  $b: \M \lo \R$ of the Hamiltonian action $A_H$, i.e., 
\begin{equation*}
b(x) = \int_0^1 H\big(t,x(t)\big)dt\,. 
\end{equation*}
and cite the following theorem. 
\begin{thm}\label{hof3} {\rm (\cite{hof})} $b$ belongs to  $C^{\infty}(\M,\R)$.
\end{thm}
In particular the proof of Theorem \ref{hof3} shows that $b$ possesses the expected derivatives. Summarizing some further statements of \cite{hof} we conclude. 
\begin{cor} \label{cor1} $A_H : \M \lo \R$ is a smooth functional. Let $\PP^{\pm} $ be the orthogonal projections 
from \eqref{pro} then the gradient of $A_H$ with respect to the inner product
on $\M$ is given by 
\begin{equation}\label{gradA}
 \Ah(x) = -\PP^+x+\PP^-x +j^*\nabla H\big(\cdot,x(\cdot)\big)\,, \quad x\in \M\,.
\end{equation}
Moreover $\Ah$ is Lipschitz-continuous on $\M$ with uniform Lipschitz constant $c_L$
and its Jacobian is given by 
\begin{equation}\label{Jac}
 \AH(x) = -\PP^++\PP^- +j^*\nabla^2 H\big(\cdot,x(\cdot)\big)\,,\quad x\in \M\,.
\end{equation}
\end{cor}
We set $\X:= -\Ah$ then the previous result implies that the {\bf negative gradient-flow} associated to $\X$ as usual  as the solution of 
\begin{equation}\label{flow}
\p_s\varphi_{\X}(s,x) = \big(\X\circ \varphi_{\X}\big)(s,x)\,, \quad \varphi_{\X}(0,x) = x \,, \quad \text{for all} \quad x\in\M\,,  
\end{equation}
is uniquely determined, smooth  and globally defined. Note that each critical point  $x \in \M$ of $A_H$ is smooth and solves the Hamilton equation 
\begin{equation*}
\dot x(t) = J_0\nabla H\big(t,x(t)\big)\,, \quad \forall t \in \mS^1\,. 
\end{equation*}
For details we refer again to \cite{hof}. 
\begin{defn} Let $X$ be a $C^1$-vector field on a Hilbert manifold $M$.
We denote by $\rest(X):=\li\{x \in M \mi X(x)=0\re\}$ the set of all {\bf singular points} of $X$. Then $x \in \rest(X)$ is called {\bf hyperbolic} if and only if 
$\spec\big(DX(x)\big) \cap i\R = \emptyset$. $X$ is called a {\bf Morse vector field} if all singular points are hyperbolic. \\

If $DX(x) \in \Lc(E)$ is self-adjoint then the spectrum of $DX(x)$ splits into 
the disjoint sets  
\begin{align*}
\spec^+(DX(x)) & = \li\{\lambda  \in \spec(DX(x)) \mi \lambda > 0 \re\}\,,\\  
\spec^-(DX(x)) & = \li\{\lambda  \in \spec(DX(x)) \mi  \lambda < 0 \re\} \,.
\end{align*}
and by the Spectral Theorem we have that $E = E^u\oplus E^s$ splits into the  {\bf positive  (or unstable)} and {\bf negative (or stable) eigenspaces} 
of $DX(x)$. Often such operators are also called {\bf hyperbolic}. A notion of non-self-adjoint hyperbolic operators can be found in \cite{Al1}.
If there is $f \in C^1(M,\R)$ such that 
$$Df(p)[X(p)]<0\,, \quad \text{for all} \, p \in M\setminus \rest(X) $$ 
then $f$ is called a {\bf Lyapunov function} for $X$. If $X$ is a Morse vector field then $f$ is called {\bf non-degenerate} if $f$ is twice differentiable on $\rest(X)$
and $D^2f(x)$, seen as a symmetric bounded bilinear form on $E$, is such that there is $\ee >0$ with $D^2f(x)[\xi,\xi] \geq \ee \nl \xi\nr^2$ for all $\xi \in E^s(x)$ and 
$D^2f(x)[\xi,\xi] \leq -\ee \nl \xi\nr^2$ for all $\xi \in E^u(x)$.  
The Morse vector field $X$ is called {\bf gradient-like} if it has a non-degenerate Lyapunov function.  
\end{defn}

Clearly the set of critical points of $f$ coincides with $\rest(X)$ if $X$ is the positive or negative gradient of $f$.
If $X$ is a Morse vector field then $\crit(f) = \rest(X)$ even holds if $f$ is just a Lyapunov function for $X$, see \cite{Al1}. Moreover all 
{\bf singular points are isolated} in this case. That is due to the fact that $DX(x)$ is invertible combined with the inverse function theorem, 
see for example \cite{duff}. 

\begin{prop}\label{reg1} There is a residual set $\Hr \subset C^{\infty}(\mS^1\times\Tn,\R)$ of Hamiltonians such that the negative
$H^{1/2}$-gradient $\X$ of the Hamiltonian action $A_H$ is a Morse vector field for every $H \in \Hr$. 
In particular the set of all contractible critical points of $A_H$, denoted by $\Pc_0(H)$, is a finite set. 
\end{prop}
This statement is based on the famous Theorem of Sard whose infinite dimensional version was proven by Smale in \cite{smale}.  
\begin{thm}\label{sard}{\rm (\cite{smale})} {\bf (Sard-Smale Theorem)} Let $X$ and $Y$ be separable Banach spaces and $U \subset X$ be an open set. Suppose that 
$f : U \lo Y$ is a Fredholm map of class $C^l$, where
$$ l \geq \max\li \{1,\ind(f) +1\re\} \,.$$ 
Then the set 
$$ Y_{\reg}(f):=\li\{ y\in Y \mi x \in U, f(x) = y \Longrightarrow \Image (df(x)) = Y \re \} $$
of regular values of $f$ is residual in $Y$ in the sense of Baire, i.e., a countable intersection of open an dense sets.
\end{thm}
Now we prove Proposition \ref{reg1}.
\begin{proof} 
Recall that all critical points $x \in \crit(A_H)$ are smooth and satisfy the Hamilton equation. Clearly the same argument holds for the linearization $\AH(x)$, i.e.,  
for all for $\xi \in \HR$ there holds $\AH(x)\xi =  0$ if and only if 
\begin{equation}\label{nondeg}
 \xi \in C^{\infty}(\mS^1,\Rn) \quad \text{and} \quad  
\dot \xi(t) = J_0\nabla^2H(t,x(t))\xi(t) \,, \forall t \in \mS^1\,.  
\end{equation}
Saying that the condition that  $D\X(x)=-\AH(x)$ is hyperbolic is equivalent to the non-degeneracy condition formulated in Hamiltonian dynamics, i.e., 
the linearized Hamiltonian flow at $x$ possesses no 1-periodic orbits meaning that there are no Floquet-multipliers equal to $1$. Now  it is a well-known fact that 
this nondegeneracy condition is achievable for a residual set of Hamiltonians, see for instance \cite{salhof}. If we finally choose a generic Hamiltonian 
$H \in \Hr$, then the set of contractible critical points $\Pc_0(H)$ of $A_H$ is a family of Hamiltonian orbits of a discrete set of points in $\Tn$. 
Hence $\Pc_0(H)$ has to be finite due the compactness of $\Tn$.
\end{proof}

\subsection{Fredholm pairs and relative dimensions}\label{fredpairs}

In contrast to the case of a smooth Morse function on a finite dimensional manifold, as it will turn out, we have to deal with the 
{\bf strongly indefinite} functional $A_H$ defined on the Hilbert manifold $\M$, i.e., $A_H$ possesses infinite Morse indices and co-indices. 
A. Abbondandolo an P. Majer developed methods in \cite{Al1} which enable them to construct an infinite dimensional Morse-theory for such functionals. \\

We start with some facts about linear Cauchy-problems more  discussed in \cite{Al7}. 
Let $$ A: [-\infty,\infty] \lo \Lc(E)$$ be a continuous  path of bounded linear operators on a Banach space $E$, such that $A(\pm \infty)$ are hyperbolic. 
We denote by  
$$X_A(s) :(-\infty,\infty) \lo \Lc(E)$$
the solution of the linear Cauchy-problem 
$$ \li\{\begin{array}{ccl}
X_A^{\prime}(s) & =  & A(s)X_A(s) \\
X_A(0) & =  & I 
\end{array}\re. \qquad \,.
$$ 
Note that $X_A(s)$ is an isomorphism for every $s$. The inverse $Y(s):=X_A(s)^{-1}$ solves the linear Cauchy-problem 
$$ \li\{\begin{array}{ccl}
Y(s) & =  & -Y(s)A(s) \\
Y(0) & =  & I 
\end{array}\re. \qquad \,.
$$ 
We define the {\bf linear, unstable} and {\bf stable} subspaces of $E$ by  
\begin{align}\label{linstab}
W^u_A &= \li \{ \xi  \in E \mi \lim \limits_{s\rightarrow -\infty}X_A(s)\xi  = 0 \re\} \\
W^s_A &= \li \{ \xi  \in E  \mi \lim \limits_{s\rightarrow +\infty}X_A(s)\xi = 0 \re\}\,.\nonumber
\end{align}
If in particular $A\equiv L$ is constant and hyperbolic then $W^u_L = E^u(L)$ and $W^s_L=E^s(L)$ are direct complements, i.e., 
$W^u_L + W^s_L = E$ and $W^u_L \cap W^s_L =\{ 0\}$. In  general one can show that 
$W^u_A \cong E^u(A(-\infty))$ and $W^s_A \cong E^s(A(+\infty))$. Indeed if $A$ is sufficiently close to $A(+\infty)$ in the $L^{\infty}$-norm, then 
$W^s_A$ is the graph of a bounded operator from $E^s(A(+\infty))$ to $E^u(A(+\infty))$. The statement for general asymptotically hyperbolic operators
then follows from the identity 
\begin{equation}\label{WS}
 W^s_A = X_A(T)^{-1}W^s_{A(T+\cdot)}\,, \quad T \in \R \,. 
\end{equation}
For details see \cite{Al7}.

\begin{prop}\label{FA1}{\bf (\cite{Al1})} Denote by $C^k_0([0,+\infty),E)$ the space of all $C^k$-curves $\xi$ which satisfy 
 $\lim \limits_{s\rightarrow +\infty}\xi^{(l)}(s)=0$ for all $0 \leq l\leq k$ and let $A \in C^0([0,+\infty],\Lc(E))$ be such 
that $A(+\infty)$ is self-adjoint and hyperbolic. 
\begin{enumerate}
 \item  The bounded linear operator 
        \begin{equation*}
         F_A^+ : C^1_0([0,+\infty),E) \lo  C^0_0([0,+\infty),E)\,, \quad \xi \mapsto \xi^{\prime}-A\xi\,,  
        \end{equation*}
        is a left inverse. Moreover $F_A^+$ admits  a right inverse $R_A^+$ such that 
        \begin{equation} \label{FA}
        W^s_A + \li\{R^+_A(\eta)(0) \mi \eta \in C^0_0([0,+\infty),E)\,, \eta(0)=0 \re\} = E 
        \end{equation}
 \item  The evaluation map 
        $$ \ker F_A^+ \lo E\,, \xi \mapsto \xi(0)$$ 
        is a right inverse. Consequently if $A \in C^0([-\infty,0],\Lc(E))$ and $A(-\infty)$ is self-adjoint and hyperbolic then 
 \item  \begin{equation*}
         F_A^- : C^1_0((-\infty,0],E) \lo  C^0_0((-\infty,0]),E)\,, \quad \xi \mapsto \xi^{\prime}-A\xi\,,  
        \end{equation*}
        is a left inverse with right inverse $R_A^-$ such that 
        \begin{equation}\label{FA-} 
         W^u_A + \li\{R^-_A(\eta)(0) \mi \eta \in C^0_0((-\infty,0],E)\,, \eta(0)=0 \re\} = E 
        \end{equation}
\item   $$ \ker F_A^- \lo E\,, \xi \mapsto \xi(0)$$ 
        is a right inverse. 
\end{enumerate}
\end{prop}
We need a further statement of the operator $F_A$ on the whole real line. 
\begin{lem}\label{FA2}{\bf (\cite{Al1})} Let $A : [-\infty,+\infty] \lo \Lc(E)$ be a path of bounded linear operators with hyperbolic asymptotics 
$A(\pm \infty)$. Denote by $F_A$ the bounded linear operator
$$ F_A :   C^1_0(\R,E) \lo  C^0_0(\R,E)\,, \quad \xi \mapsto \xi^{\prime}-A\xi\,. $$ 
Then  
$$ \ker F_A \cong W^u_A\cap W^s_A \quad \text{and} \quad \coker F_A \cong E/(W^u_A + W^s_A ) \,.$$ 
\end{lem}
\begin{proof} The kernel of $F_A$ is the linear subspace 
$$ \ker F_A = \li\{ X_A(s) v \mi v \in W^u_A\cap W^s_A \re \} \subset C^1_0(\R,E) \,.$$
which is therefore by the second statement of Proposition \ref{FA1} isomorphic to $ W^u_A\cap W^s_A$ via the evaluation map $X_A(s)v \mapsto X_A(0)v = v$. 
By the first statement the operators 
\begin{align*}
 F_A^+ :   C^1_0([0,+\infty),E) &\lo  C^0_0([0,+\infty),E)\,, \quad \xi \mapsto \xi^{\prime}-A\xi\,,\\
 F_A^- :   C^1_0((-\infty,0],E) &\lo  C^0_0((-\infty,0],E)\,, \quad \xi \mapsto \xi^{\prime}-A\xi \,,
\end{align*}
have right inverses $R_A^+$ and $R_A^-$. If $\eta \in C^0_0(\R,E)$  and $\xi$ solves $ \xi^{\prime}-A\xi = \eta$, then $\xi$ is of the form  
\begin{align*}
\xi(s) &= X_A(s)\Big(\xi(0) - R_A^+(\eta)(0)\Big) + R_A^+(\eta)(s)\,, \quad \text{for}\quad s\geq 0\,, \\ 
\xi(s) &= X_A(s)\Big(\xi(0) - R_A^-(\eta)(0)\Big) + R_A^-(\eta)(s)\,, \quad \text{for}\quad s\leq 0\,.
\end{align*}
If we require that $\xi \in C^1_0(\R,E)$, then $\xi(0)$ has to be such that 
$$\xi(0)- R_A^+(\eta)(0) \in W^s_A \quad \text{and} \quad \xi(0)- R_A^-(\eta)(0) \in W^u_A \,.$$
In other words, $\eta$ belongs to the range of $F_A$ if and only if the affine subspaces $ R_A^+(\eta)(0) + W^s_A$ and  $ R_A^-(\eta)(0) + W^u_A$ have a non-empty 
intersection, that is if and only if $ R_A^+(\eta)(0) - R_A^-(\eta)(0) \in W^s_A + W^u_A$. So the range of $F_A$ is the linear subspace 
$$ \ran F_A = \li \{ \eta \in C^0_0(\R,E) \mi R_A^+(\eta)(0) -R_A^-(\eta)(0) \in W^s_A + W^u_A \re\} \,.$$
By the remark in \eqref{WS} the spaces $W^s_A$ and $W^u_A$ are closed and so is $\ran F_A$ by the continuity of $R_A^+$, $R_A^-$. Furthermore
due to the properties \eqref{FA} and \eqref{FA-} the operator  
$$ C^0_0(\R,E) \lo E /(W^s_A + W^u_A)\,,\quad  \eta \mapsto [R_A^+(\eta)(0) -R_A^-(\eta)(0)] \,.$$
is onto and factors over $\ran F_A$, proving the isomorphism.
\end{proof}
Clearly, we want to show that the operator $F_A$ is a Fredholm operator, but a priori there is no reason why $\ker F_A$ and $\coker F_A$ should be finite dimensional. 
In particular this will not be true in general. To formulate the correct additional conditions, we use the notion of Fredholm pairs which we briefly 
introduce next. For a deeper presentation see \cite{Al2}.\\

We consider the {\bf Hilbert Grassmannian} $\Gr(E)$ of a separable Hilbert space $E$, which is the set of all closed subspaces of $E$. 
For $V \in \Gr(E)$, we denote by $P_V$ the orthogonal projection onto $V$. The distance between two subspaces $V,W \in \Gr(E)$ shall be defined by
$$ \dist(V,W):=\nl P_V-P_W\nr \,.$$
Indeed $\big(\Gr(E),\dist(\cdot,\cdot)\big)$ becomes a complete metric space with connected components
$$ \Gr_{n,m}(E)=\li\{ V \in \Gr(E) \mi \dim V =n\,, \, \codim V = m \re\}\,,$$
where $n,m \in \N \cup \{\infty\}\,, \, n+m = \infty$
\begin{defn} A pair $(V,W) \in \Gr(E) \times \Gr(E)$ is called a {\bf Fredholm pair} if $V\cap W$ is finite-dimensional and $V+W$ is closed and finite-codimensional. 
In this case the index 
$$ \ind(V,W):= \dim \li(V\cap W\re)-\codim\li(V+W\re)$$
is called the {\bf Fredholm index} of $(V,W)$. 
\end{defn}
In particular an operator $F : X\lo Y$ between Hilbert spaces is Fredholm if and only if $\big(\graph F,X\times\{0\}\big)$ is a Fredholm pair in 
$\Gr(X\times Y)\times \Gr(X\times Y) $. 
The set of all Fredholm pairs, denoted by $\mathrm{Fp}(E)$, is an open subset of $\Gr(E) \times \Gr(E)$ and the Fredholm index is a 
continuous, i.e., locally constant function on it. For details see \cite{kato}.\\

\begin{defn} Let $V,W \in \Gr(E)$. Then $V$ is called a {\bf compact perturbation} of $W$ if the operator $P_V-P_W$ is compact. In this case the pair 
$(V,W^{\perp})$ is Fredholm and its index is defined as the {\bf relative dimension} of $V$ with respect to $W$, i.e.,  
\begin{equation}\label{reldim}
 \dim(V,W):=\ind(V,W^{\perp})= \dim \li(V\cap W^{\perp}\re) -  \dim \li(V^{\perp}\cap W\re) \,.
\end{equation}
\end{defn}
If $(V,W)$ is a Fredholm pair and $U$ is a compact perturbation of $V$, then $(U,W)$ is still a Fredholm pair and its index is given by 
\begin{equation}\label{fred}
 \ind(U,W)=\dim(U,V) +\ind(U,W) \,. 
\end{equation}

\begin{prop}\label{eig}{\rm (\cite{Al6})} Let $F_1$ and $F_2$ be two self-adjoint Fredholm operators on the Hilbert space $E$ with $F_1-F_2$ compact.
Then the positive and negative eigenspaces of $F_1$ are compact perturbations of the positive and negative eigenspaces of $F_2$, respectively.
\end{prop}
\begin{proof}
The operator $p(F_1)-p(F_2)$ is compact for any polynomial $p$, for it belongs to the two-sided ideal spanned by $F_1-F_2$. By density, 
$h(F_1)-h(F_2)$ is then compact for any continuous function on $\spec(F_1) \cup \spec(F_2) \subset \R$. Since $F_1$ and $F_2$ are self-adjoint, 
$0$ is not an accumulation point of $\spec(F_1) \cup \spec(F_2)$. Hence $\chi_{\R^+}$ and $\chi_{\R^-}$ are continuous functions on $\spec(F_1) \cup \spec(F_2)$. 
Then $P_{E^+(F_1)} - P_{E^+(F_2)} = \chi_{\R^+}(F_1) - \chi_{\R^+}(F_2)$ and $P_{E^-(F_1)} - P_{E^-(F_2)} = \chi_{\R^-}(F_1) - \chi_{\R^-}(F_2)$ are compact 
operators as claimed. 
\end{proof}

Using the concepts introduced above we are able to study our non-linear situation as follows.  
\begin{defn}\label{defstab} Let $x \in M$ be a singular point of a $C^1$-vector field $X$ on a Banach manifold $M$ with globally defined flow $\ffi$. We define 
the {\bf unstable} and {\bf stable manifolds} of $x$ by 
\begin{align*}
 \WW^u(x) &= \li \{ p \in M \mi \lim \limits_{s\rightarrow -\infty}\ffi(s,p) = x \re\} \\
 \WW^s(x) &= \li \{ p \in M \mi \lim \limits_{s\rightarrow +\infty}\ffi(s,p) = x \re\}\,.
\end{align*}
Furthemore we define the Jacobian of $X$ at $p\in M$ in the usual way as the homomorphism $DX(p) : T_pM \lo T_pM $  determined by 
\begin{equation*}
 DX(p)[Y(p)](f) = Y(p)(f\circ X)(p)\,,
\end{equation*} 
forall $C^0$- vector fields $Y$ and all $ f\in C^1(TM,\R)$.
\end{defn}
The generalization of the usual {\bf stable / unstable manifold theorem} in this situation is explicitly proven in \cite{Al1}. We only mention it :

\begin{thm}\label{stable}{\bf (\cite{Al1})} Let $x \in M$ be a hyperbolic singular point of a $C^k$-vector field $X$, $k\geq 1$, on a Banach manifold $M$ with 
globally defined flow. Then 
$\WW^u(x)$ and $\WW^s(x)$ are the images of injective $C^k$-immersions of manifolds which are homeomorphic to the unstable and stable eigenspaces 
$E^u_x$ and $E^s_x$ of the Jacobien of $X$ at $x$, respectively. \\

If $M$ is a Hilbert manifold then  $\WW^u(x)$ and $\WW^s(x)$ are actually images of $C^k$-injective immersions of $E^u_x$ and $E^s_x$, respectively. If in addition 
$X$ possesses a Lyapunov function $f$,  non-degenerate at $x$, then these immersions are actually embeddings. 
\end{thm}

Though in general $\WW^u(x)$ and $\WW^s(x)$ are infinite dimensional manifolds, one can hope, that for a generic choice of vector fields $X$  or metrics on $M$, 
their intersection is a {\bf finite dimensional} manifold. A. Abbondandolo and P. Majer formulated two conditions which guarantee that we are in this situation. 
We give a version of these conditions matching our setup.
\begin{defn}\label{admissible}
Let $X$ be a $C^1$- Morse vector field on the Hilbert manifold $\M = \HT$. 
If $V \subset \Ee$ is a closed subspace and $\V = \M \times V$ a constant subbundle of $T\M$. 
Then $\V$ is said to be {\bf admissible}  for $X$ if \\
 
{\bf (C1)} for every singular point $x \in \rest (X)$, the unstable eigenspace $E^u\li(DX(x)\re)$ of the Jacobian  of $X$ at $x$ is a compact perturbation of $V$;\\

{\bf (C2)} for all $p\in M$ the operator $\big[ DX(p),P_V \big]$ is compact, where $[S,T]=ST-TS$, $\forall$ $S,T \in \Lc(\Ee)$ denotes the commutator
and $P_V$ the projection onto $V$. \\
\end{defn}
By {\bf C1} we can define the {\bf relative Morse index} of $x\in \rest(X)$ relative to $\V$ by 
\begin{equation}
 m(x,\V):= \dim\Big(E^u\big(DX(x)\big),V\Big)\,.
\end{equation}
Furthermore  the subbundle $\V$ is $\ffi$-invariant, i.e., $D_2\ffi(s,p)V = V$ for all $(s,p)\in \R\times \M$, if and only if $\big[ DX(p),P_V \big]=0$. 
Hence {\bf C2} is equivalent to the fact that $\V$ is {\bf essentially} $\ffi$-invariant in the sense that 
$D_2\ffi(s,p)V$ is a compact perturbation of $V$ for all $(s,p)\in \R\times \M$.  

\begin{lem}\label{findim} Let $H \in \Hr$ be a generic Hamiltonian in the sense of Proposition \ref{reg1} and $\X = -\Ah$ which is therefore a smooth Morse 
vector field on $\M$. Then the subbundle 
$$\M\times \big(\R^n \times \Hm^+\big) \subset   T\M $$ is admissible for $\X$.
\end{lem}
\begin{proof} Since $\X$ is assumed to be Morse, the operator 
$D \X(x)$ is hyperbolic for all $x \in \rest(\X)$. Hence its spectrum splits into the disjoint sets  $\spec^+\big(D \X(x)\big)$ and $\spec^-\big(D \X(x)\big)$ of positive and negative eigenvalues. Furthermore $D\X(x)$  satisfies the estimate 
\begin{align*}
 \nl \xi \nr_{\HH} & =    \nl \PP^0 \xi + D\X(x) \xi  + j^*\nabla^2H\big(\cdot,x(\cdot)\big)\xi \nr_{\HH} \\
		   & \leq \nl D\X(x) \xi \nr_{\HH} + \nl j^*\PP^0 \xi + j^*\nabla^2H\big(\cdot,x(\cdot)\xi\big) \nr_{\HS}\\
	           & \leq \nl D\X(x) \xi \nr_{\HH} + \big(1+ \nl H \nr_{C^2(\mS^1\times\Tn)}\big)\nl \xi \nr_{\LS}
\end{align*}
for all $\xi \in \Ee$. Now recall that the inclusion $j : \HR \lo \LLR$ from \eqref{j} is compact. Hence the above estimate implies that $D\X(x)$ is a semi-Fredholm
operator as it is well-known in Fredholm analysis, see for instance \cite{duff} or \cite{schwarz-2}. Since $D\X(x)$ is moreover self-adjoint its kernel and co-kernel coincide and 
therefore $D\X(x)$ is indeed Fredholm of index zero. In particular $0$ is not an accumulation point in $\spec(D\X(x))$.
So setting $S_1 =D \X(p) \PP^+$ and $S_2= \PP^+D \X(p)$ for $p \in \M$ we have that 
$$\big[ D \X(p),\PP^+\big] = S_1 -S_2 = j^*\nabla^2H(\cdot,p)\PP^+ -\PP^+j^*\nabla^2H(\cdot,p)$$ 
is compact, which shows {\bf C2}. For $p = x$  we can apply Proposition \ref{eig} to $F_1 = D\X(x)$, $F_2 = \PP^+-\PP^-$ and obtain that the positive eigenspace 
$E^u\big(D\X(x)\big)$ is a compact perturbation of $\Hm^+$ and therefore of $\R^n \times \Hm^+$ proving {\bf C1} .\\
\end{proof}
Note that since $I-\PP^-$ differs from $\PP^+$ by an operator of finite rank, we have that $\big[ D \X(p),\PP^-\big]$ is also compact for all $p \in \M$. In particular this shows that $E^s\big(D\X(x)\big)$ is a compact perturbation of $\Hm^-$ for all $x \in \rest(\X)$. 
Since both spaces $\Hm^+$ and $\Hm^-$ are  infinite dimensional this implies that 
$E^{u}\big(D\X(x)\big)$ and $E^{s}\big(D\X(x)\big)$ are infinite dimensional for any $x \in \rest(\X)$, which in particular shows the 
{\bf strongly indefinite character} of $A_H$, i.e., the usual Morse indices and co-indices are infinite for all singular points implying that by Theorem \ref{stable}  
the stable  and unstable manifolds are infinite dimensional. Further crucial consequences of  the two conditions are the following. 
\begin{prop}\label{man}{\bf (\cite{Al1})} Assume that $X$ is a $C^1$-Morse vector field defined on the Hilbert manifold $\M = \Tn \times \Hm$, 
possesses a Lyapunov function $f$ and an admissible constant subbundle $\V = \M \times V \subset \M\times\Ee$. Then for every $x\in \rest(X)$ there holds :
\begin{enumerate}
 \item  for every $p \in \WW^u(x)$, $T_p\WW^u(x)$ is a compact perturbation of $V$ with 
 $$ \dim(T_p\WW^u(x),V)= m(x,\V)\,. $$
 \item for every $p \in \WW^s(x)$, $\li(T_p\WW^s(x),V\re)$ is a Fredholm pair, with
 $$\ind\big(T_p\WW^s(x),V\big) = -m(x,\V) \,.$$
 \item Let $x,y \in \rest(X)$ and assume that $\WW^u(x)$ and $\WW^s(y)$ meet transversally. Then $\WW^u(x) \cap \WW^s(y)$ is a submanifold of
 dimension 
$$ \dim\big(\WW^u(x) \cap \WW^s(y)\big) = m(x,\V)-m(y,\V)\,.$$  
\end{enumerate}
\end{prop}
\begin{proof} Let $p \in \WW^u(x)$ and let $v(s)=\ffi(s,p)$ be the orbit of $p$. By linearization, using the notation from \eqref{linstab}, we obtain 
\begin{equation}\label{tw} 
 T_p\WW^u(x) = W^u_A
\end{equation}
with $A(s) = DX\big(v(s)\big)$. {\bf C1} implies that $W:=T_x\WW^u(x) = E^u\big(A(-\infty)\big)$ is a compact perturbation of 
$V$. By {\bf C2}, the operator $\big[A(s),P_V\big]$ is compact for every $s$, and so is the operator $\big[A(s),P_W\big]$. Set 
$$ B(s) = A(s) - \big[A(s),P_W\big]\,.$$
Then $B(-\infty) = A(-\infty)=DX(x)$, $E^u\big(B(-\infty)\big)=W$,
$$ P_{W^{\perp}}B(s) = P_{W^{\perp}}\big(A(s)-A(s)P_W\big) =  P_{W^{\perp}}A(s)P_{W^{\perp}}$$ 
and therefore $B(s)W \subset W$ for every $s$. It is readily  seen that these facts imply that 
$$ W^u_B = W \,.$$
Furthermore $B(s)-A(s)$ is compact for every $s$, and thus $W^u_A$ is a compact perturbation of $W^u_B=W$ and therefore of $V$. So we conclude by 
\eqref{tw} that $T_p\WW^u(x)$ is a compact perturbation of $V$ for all $p \in \WW^u(x)$. The dimension formula now follows by continuity and we obtain $(i)$.\\

Since the set $\mathrm{Fp}(E)$ of all Fredholm pairs is open and the index is locally constant {\bf C1} implies that the pair 
$\big(T_p\WW^s(x),V\big)$ is Fredholm of index $-m(x,\V)$ for all $p$ in a neighborhood of $x$. Furthermore the tangent bundle $T\WW^s(x)$ is invariant under the flow $\ffi$ and 
by {\bf C2} the subbundle $\V$ is essentially invariant, i.e., invariant up to compact perturbation. Hence $\big(T_p\WW^s(x),V\big)$ is Fredholm of 
index $-m(x,\V)$ for all $p\in \WW^s(x)$ proving $(ii)$.\\

To prove $(iii)$ let $\mathcal{C}^1_{x,y}(\R,\M)$ be the Banach manifold of all curves $ v \in C^1(\R,\M)$ with $\lim \limits_{s\rightarrow-\infty}v(s) =x$, 
$\lim \limits_{s\rightarrow +\infty}v(s) =y$ and $\lim \limits_{s\rightarrow\pm\infty}v^{\prime}(s) = 0$. We consider the map 
$$f : \mathcal{C}^1_{x,y} \lo  C^0_0(\R,T\M)\,, \quad v \mapsto v^{\prime} - X(v) $$ and its linearization at $v$   
$$ F(v) :C^1_0(\R,\Ee) \lo C^0_0(\R,\Ee)\,, \quad  \xi \mapsto \xi^{\prime} - DX(v)\xi \,.$$
We set $A(s):= DX(v(s))$, then by Lemma \ref{FA2} the operator $F_A:=F(v)$ has range  
$$\ran F_A  = \li \{ \eta \in C^0_0(\R,\Ee) \mi R_A^+(\eta)(0) -R_A^-(\eta)(0) \in W^u_A + W^s_A \re\} \cong W^u_A + W^s_A  \,.$$
and kernel
$$\ker F_A = \li\{X_A(s)w \mi w \in W^u_A\cap W^s_A\re\} \cong W^u_A\cap W^s_A \subsetneq \Ee$$ 
Since we assume that $\WW^u(x)$ and $\WW^s(x)$ meet transversally, 
 $W^s_A + W^u_A = \Ee$, for all $v \in f^{-1}(0)$, which implies that $\coker F_A = \{0\}$.
In this case the dimension formula for the kernel follows from $(i)$, $(ii)$ and \eqref{fred}, i.e.,  
\begin{align*}
\dim(W^u_A \cap W^s_A ) & = \ind (W^u_A,W^s_A) = \dim (W^u_A,V) +  \ind(W^s_A,V) \\
                        & = m(x,\V) -m(y,\V)\,. 
\end{align*}
Hence $f$ is a Fredholm map with regular value $0$ possessing the claimed index. \\
\end{proof}
Note that even if the operator $F_A$ is not onto we have that $\dim(W^u_A,V) +\ind(W^s_A,V)$ is finite by $(i)$ and $(ii)$ and so again by \eqref{fred}  there holds 
\begin{align}
\ind F_A & = \ind(W^u_A,W^s_A) = \dim(W^u_A,V) +\ind(W^s_A,V) \nonumber \\ 
         & = m(x,\V) - m(y,\V) \,. \label{FAind}
\end{align}

\subsection{Genericity of Morse-Smale}\label{2.3}
Though A. Abbondandolo and P. Majer established a transversality result by perturbing the metric on the Hilbert manifold $M$, see \cite{Al1},
in preparation for the transversality result for hybrid type curves we  give an alternative result by  perturbing  $\X=-\Ah$ 
directly by adding a further small compact vector field $K : \M \lo \Ee$. 

\begin{defn} Let $X$ be a  Morse vector field on the Hilbert manifold $\M$  possessing a Lyapunov function $f$ and an admissible constant subbundle $\V$. 
We will say  that $X$ satisfies the {\bf Morse-Smale condition} up to order $k\in \N$ if for every pair of singular points $x,y \in \rest(X)$ with $m(x,\V) -m(y,\V) \leq k$ 
the intersection of the submanifolds $\WW^u(x)$ and $\WW^s(y)$ is transverse. 
\end{defn}
By Proposition \ref{reg1} we can assume that we have chosen $H$ generically such that $\X = -\Ah$ is a Morse vector field. Moreover due to Lemma \ref{findim}, 
$\X$ satisfies the conditions {\bf C1} and  {\bf C2} with respect to the subbundle $\M\times \big(\Hm^+\times\R^n\big)$.
To achieve the transversality result we have to choose a good set of suitable perturbations. That is : \\
 
Let $\Kc(\M,\Ee) \subset C^{3}_b(\M,\Ee)$ be the closed subspace of all $C^3$- vector fields which are compact and bounded on $\M$. 
We choose a function $\theta \in C^1(\M,\R^+)$ such that
\begin{enumerate}
 \item $\theta(x)=0$ for all $x \in \crit(A_H)$ 
 \item $\theta(p) > 0 $ for all $p\in \M \setminus \crit(A_H)$.
 \item $\theta(p) \leq \frac{1}{2}\nl \Ah(p)\nr_{\HS}$ for all $p \in \M$. 
\end{enumerate}
and consider 
\begin{equation*}
\Kc_{\theta}: =\li\{ K\in \Kc(\M,\Ee) \mi  \exists \, c > 0 \,\text{such that}\, \nl K(p)\nr_{H^{1/2}(\mS^1)} \leq c \theta(p)\,, \forall p \in \M \re\}\,,
\end{equation*}
which becomes a Banach space with the norm 
$$ \nl K\nr_{\theta}: = \sup_{p\in\M\setminus \rest(X)}\frac{\nl K(p)\nr_{\HS}}{\theta(p)} + \nl \nabla K\nr_{C^{2}(\M,\Ee)} \,.$$
By $\Kc_{\theta,1}$ we denote the open unit ball in $\Kc_{\theta}$, which as an open subset of a Banach space is a Banach manifold with trivial tangent 
bundle $\Kc_{\theta,1}\times \Kc_{\theta}$. We claim that $\Kc_{\theta,1}$ is a good set of perturbations.
\begin{lem}\label{pert} Let $K \in \Kc_{\theta,1}$ and denote by $\X_K:=-\Ah+K$. Then there holds :
 \begin{enumerate}
 \item $\rest(\X_K) = \crit(A_H)$.
 \item $D\X_K(x) = -D^2A_H(x)$ for all $x \in \crit(A_H)$.
 \item $A_H$ is a Lyapunov function for $\X_K$.
 \item $\V:=\M \times \big(\R^n \times \Hm^+\big) $ is admissible for $\X_K$. 
\end{enumerate}
\end{lem}
\begin{proof} Since $ \theta $ and $d\theta$ vanish on $\crit(A_H)$, we have that $K$ vanishes on $\crit(A_H)$ and obtain $\crit(A_H) \subset \rest(\X_K)$.
Let $x \in \crit(A_H)$ and $y\in \Ee$, then we compute
\begin{align*}
 \nl  DK(x)[hy]  \nr_{\HS} & =    \nl K(x+hy) -K(x)+o(h)\nr_{\HS}\\
			   & =    \nl K(x+hy) \nr_{\HS} +o(h) \\			
		           & \leq \theta(x+hy) + o(h) = \theta(x) +o(h) = o(h)\,.  
\end{align*}
Hence $DK(x) = 0$ proving $(ii)$. To prove $(iii)$ let $p \in \M \setminus \crit(A_H)$, then there holds
\begin{align*}
DA_H(p)[\X_K] & =    -\big<\X(p),\X(p)+K(p)\big>_{\HS} \\
              & \leq -\nl \X(p) \nr_{\HS}^2 + \nl \X(p)\nr_{\HS} \nl K(p)\nr_{\HS} \\
              & \leq -\frac{1}{2} \nl \X(p) \nr_{\HS}^2<0\,
\end{align*}
which implies furthermore that  $\rest(\X_K) \subset \crit(A_H) $ proving $(i)$.
Finally $(i)$ to $(iii)$ imply that the dynamics of $\X_K$ do not qualitatively differ from those of $\X$. Together with the fact that $K$ is compact 
this implies that Lemma \ref{findim} continues to hold for $\X_K$ and we obtain $(iv)$. \\
\end{proof}
The main result of this section is the following Theorem.
\begin{thm}\label{SM} There is a residual subset $\Kc_{\reg}\subset \Kc_{\theta,1}$ of compact vector fields $K : \M \lo \Ee$ such that the perturbed vector field $\X_K = -\Ah+K$  fulfills the Morse-Smale condition up to order 2. 
\end{thm}
The proof of the theorem needs some preparation. For $x\not = y \in \crit(A_H)$ we consider the map 
 \begin{equation} \label{ppssii}
\psi :  \mathcal{C}^1_{x,y}(\R,\M) \times \Kc_{\theta,1} \lo C^0_0(\R,\Ee)\,, \quad v \mapsto v^{\prime}-\X_K(v)\,,
\end{equation}
The derivative with respect to the first variable is given by 
\begin{equation*}
 D_1\psi(v,K) : C^1_0(\R,\Ee) \lo C^0_0(\R,\Ee)\,, \quad \eta \mapsto \eta^{\prime} -\big(D\X(v) + DK(v)\big)\eta \,.
\end{equation*}
Since $K$ is assumed to be bounded in $C^{3}(\M,\Ee)$,  $D_1\psi(v,K)$ is a bounded linear operator. 
Furthermore it is not hard to show that the operator $F_A =  D_1\psi(v,K)$, with $A = D\X_K(v)$,  is Fredholm of index 
$m(x,\V) - m(y,\V)$, where $\V= \M\times(\R^n\times\Hm^+)$, and by Lemma \ref{FA2} $F_A$ is onto if and only if the unstable and stable manifolds $\WW^u(x)$ and 
$\WW^s(y)$ with respect to $\X_K$ meet transversally.\\

Differentiating $\psi$ with respect to the second variable yields 
$$ D_2\psi(v,K) :  \Kc_{\theta} \lo C^0_0(\R,\Ee)\,, \quad \kappa  \mapsto \kappa (v) \,,$$  
which due to the estimate $\nl \kappa(v)\nr_{C^0} \leq \max \limits_{p \in v(\R)}\theta(p) \nl \kappa\nr_{\theta} $ 
is indeed a bounded linear operator.\\

The key point in the proof of the theorem is the following observation.
\begin{lem}\label{MS} For $\mathcal{Z} =\psi^{-1}(0)$, the operator 
$$D\psi(v,K):C^1_0(\R,\Ee) \times \Kc_{\theta}  \lo     C^0_0(\R,\Ee)$$
with 
$$D\psi(v,K)[\eta,\kappa] = D_1\psi(v,K)[\eta] + D_2\psi(v,K)[\kappa] $$
is a left inverse for all $(v,K)\in \mathcal{Z}$. In particular $\mathcal{Z}$ is a Banach manifold.   
\end{lem}

The proof uses the following general linear statement. 
\begin{lem}\label{cker}{\bf (\cite{Al1})} Let $E,F,G$ be Banach spaces and let $A \in \Lc(E,G)$ possess finite codimensional range and complemented kernel. 
Then for every $B \in \Lc(F,G)$ the kernel of the operator 
$$C : E\times F \lo G \,, \quad C(e,f)=Ae-Bf \,, $$ 
is complemented in $E\times F$.
 \end{lem}
Now we prove Lemma \ref{MS}.
\begin{proof}
Since $D_1\psi(v,K)$ is Fredholm we can set  $A =D_1\psi(v,K)$ and $B = -D_2\psi(v,K)$ in Lemma \ref{cker} which proves that $D\psi(v,K)$ possesses 
complemented kernel in $C^1_0(\R,\M) \times \Kc_{\theta}$ and it remains to show that $D\psi(v,K)$ is onto. \\

The range of $D\psi(v,K)$ contains the range of $D_1\psi(v,K)$ which is closed and finite codimensional. So it suffices to show that the range of 
$D_2 \psi(v,K)$ is dense.  If $v$ is a constant flow line of $\varphi_{X_K}$ then transversality is already achieved by the regularity of $H$ 
so we may assume that $v: \R \lo \M$ is a non-constant flow line  and therefore a $C^1$-embedding. Let $\rho  \in C^0_0(\R,\Ee)$; 
then $ \rho \circ v^{-1} : v(\R)\subset \M \lo \Ee $ is bounded and maps bounded sets $U \subset v(\R)$ to precompact sets. Hence we can find 
$\kappa \in \Kc_{\theta}$  
such that $\nl \kappa_{|v(\R)} - \rho \circ v^{-1}\nr_{C^0(v(\R),\Ee)} $ becomes arbitrarily small. This proves the density and therefore the claim. 
\end{proof}
Before we  can finally prove Theorem \ref{SM} we need a further general result. 
\begin{prop}{\bf (\cite{Al1})} Let $E,F,G$ be Banach spaces and $A \in \Lc(E,F)$, $B\in \Lc(E,G)$ be left inverses. Then : 
\begin{enumerate}
 \item  $A_{| \ker B}$ is a left inverse if and only if $B_{| \ker A}$ is a left inverse. 
 \item  $A_{| \ker B}$ is Fredholm if and only if $B_{| \ker A}$ is Fredholm, in which case the indices coincide.
\end{enumerate}
\end{prop}

The Proposition has an immediate consequence. 
\begin{cor}\label{psii}{\bf (\cite{Al1})} Let $M,N,O$ be Banach manifolds and $\phi \in C^1(M,N)$ , $\psi \in C^1(M,O)$ be maps with regular values $p \in N$, $q\in O$. Then :    
 \begin{enumerate}
 \item  $p$ is a regular value of $\phi_{| \psi^{-1}(q)}$ if and only if $q$ is a regular value for  $\psi_{| \phi^{-1}(p)}$. 
 \item   $\phi_{| \psi^{-1}(q)}$  is a Fredholm-map if and only if $\psi_{| \phi^{-1}(p)}$ is Fredholm-map, in which case the indices coincide.
\end{enumerate}
\end{cor}

Now we can state the {\bf proof of Theorem \ref{SM}:} 
\begin{proof} Recall that by Lemma \ref{MS} $\mathcal{Z}= \psi^{-1}(0)$ with $\psi$ from \eqref{ppssii}  is a Banach manifold. 
Let $\pi : \mathcal{C}^1_{x,y}(\R,\M) \times \Kc_{\theta,1} \lo \Kc_{\theta,1}$ be the projection onto the second 
factor. We claim that $\pi_{|\mathcal{Z}}$ is Fredholm of index $m(x,\V)-m(y,\V)$. Indeed by Corollary \ref{psii} applied to $\psi$ from 
$M = \mathcal{C}_{x,y}(\R,\M)\times \Kc_{\theta,1}$ to $O = C^0_0(\R,\Ee)$  and $\phi = \pi $ from $M$ to $N=\Kc_{\theta,1}$ we obtain the claimed result. 
Assume that $m(x,\V) -m(y,\V) \leq 2$ and denote by $\Kc_{\reg}(x,y) \subset \Kc_{\theta,1}$ the regular values of $\pi_{|\mathcal{Z}}$. Again by Corollary \ref{psii} this is the set 
such that  
$$\psi: \mathcal{C}^1_{x,y}(\R,\M) \times \Kc_{\reg}(x,y) \lo C^0_0(\R,\Ee)$$ possesses regular value $0$, i.e.,  $\WW^u(x)$ and $\WW^s(y)$ with respect to 
$\X_K$, $K \in \Kc_{\reg}(x,y)$ meet transversally. Now we would like to apply the Sard-Smale Theorem \ref{sard} but we are in the delicate situation where the manifold 
$\Kc_{\theta,1}$ is not separable. In this case, due to \cite{henry} Theorem \ref{sard} still holds if $\pi_{|\mathcal{Z}}$ is {\bf $\sigma$ - proper}, i.e.,   there 
exists a countable family of sets $\li\{M_k\re\}_{k\in\N}$ which covers $\mathcal{Z}$ such that $\pi_{| M_k}$ is proper. We prove this additional property 
in Section \ref{compinter} (see Lemma \ref{sigprop}) and finish the proof now as follows. Since $\pi_{|\mathcal{Z}}$ is  $\sigma$ - proper and of class $C^3$ 
the Sard-Smale Theorem \ref{sard} tells us 
that  $\Kc_{\reg}(x,y)$ is residual in $\Kc_{\theta,1}$ and so is the set 
$$\Kc_{\reg}:= \bigcap_{\genfrac{}{}{0pt}{}{x,y \in \rest(X)\,,}{0<m(x,\V) -m(y,\V) \leq 2 \,,}}\Kc_{\reg}(x,y) \,.$$
\end{proof}

\subsection{Compact intersections}\label{compinter}

In this section we state the compactness result of A. Abbondandolo and P. Majer proven in \cite{Al5}. In particular it turns out that the unstable manifold 
is essentially vertical in the sense that the projection of the sub-level sets of the unstable manifold  onto $\Hm^-$ is precompact. 
This crucial fact allows us to prove the compactness result for the hybrid type curves in Section 4. 
Here is the precise result : 

\begin{thm}\label{compactM}{\bf (\cite{Al5})} Let $\X_K = -\Ah +K$, $K \in \Kc_{\theta,1}$, $x,y \in \crit(A_H)$ and
$\WW^u(x),\WW^s(y)$ be the unstable, stable manifolds of $x,y$ with respect to $\X_K$ and $c,d \in \R$  be regular values of $A_H$ 
with $c<A_H(x)<d$.
\begin{enumerate}
 \item The intersection $\WW^u(x)\cap \WW^s(y) \subset \M$ is precompact.\\
 \item The sets $\PP^+\Big(\WW^s(x) \cap \{A_H \leq d \}\Big)$ and $\PP^-\Big(\WW^u(x) \cap \{A_H \geq c\}\Big)$ are precompact. 
\end{enumerate} 
\end{thm}
So $(ii)$ says that $\WW^u(x)$ is {\bf essentially vertical}, i.e., vertical up to a precompact set,  and $ \WW^s(x)$ is {\bf essentially horizontal}  
with respect to the splitting $\Hm^-\oplus \Hm^+ = \Hm$. Though $(ii)$ is not explicitly proven in \cite{Al5} its proof is similar to the one of 
$(i)$ and therefore omitted here. Next we state the compactness result for gradient-like trajectories. 

\begin{defn} Let $X$ be a $C^1$- Morse vector field on a Banach manifold $M$ and $x,y \in \rest(X)$. A {\bf broken trajectory} from $x$ to $y$ is a set 
$$ S= S_1 \cup \dots \cup S_k\,, \quad k \geq 1 \,,$$
and $S_i$ is the closure of a flow line from $z_{i-1}$ to $z_i$, where $ x = z_0 \not = z_1 \not = \dots \not =  z_{k-1} \not = z_k = y $ are all singular points.    
\end{defn}
When $k=1$, a broken trajectory is just the closure of a genuine trajectory.
\begin{prop}\label{comp}{\bf (\cite{Al5})} Assume that the Morse vector field $X$ on a Banach manifold $M$  has a Lyapunov function $f$ and that $x,y \in \rest(X)$ are such that
$\overline{\li(\WW^u(x) \cap \WW^s(y)\re)}$ is compact. Let $ p_n \in \WW^u(x) \cap \WW^s(y)$ be a sequence and set 
$S_n:= \varphi_X(\R \times \{p_n\}) \cup \{x,y\}$, where $\varphi_X$ denotes the flow of $X$. Then $S_n$ has a subsequence which converges to a broken flow line from $x$ to $y$, in the Hausdorff-distance. 
\end{prop}

\begin{cor}\label{onedim} Let the assumptions of Proposition \ref{comp} be fulfilled and assume furthermore that $X$
 satisfies the Morse-Smale condition up to order 1 and that the intersection of $\WW^u(x)$  and $\WW^s(y)$ is 1-dimensional,
i.e., $m(x,\V)-m(y,\V)=1$. Then $\WW^u(x) \cap \WW^s(y)$ is already compact.
\end{cor}
\begin{proof} By our assumptions there are no constant trajectories in $\WW^u(x) \cap \WW^s(y)$, which implies that the  $\R$-action 
$$ \R\times  \WW^u(x) \cap \WW^s(y) \lo  \WW^u(x) \cap \WW^s(y)\,, \quad (s,p) \mapsto \varphi_X(s,p) $$ 
acts freely and properly on  the  1-dimensional manifold $\WW^u(x) \cap \WW^s(y)$, which is therefore the union of orbits of 
a discrete set of points. Assume there is a sequence of orbits converging to a broken trajectory with at least on intermediate point $z$. Then by transversality 
$m(x,\V)-m(z,\V),m(z,\V)-m(y,\V)\geq 1$ and therefore $m(x,\V)-m(y,\V)\geq 2$; a contradiction. Hence $\WW^u(x) \cap \WW^s(y)$ is compact and consists not of 
orbits of a discrete set of points but of orbits of finitely many points.
\end{proof}

As promised we prove the additional property of {\bf $\sigma$ - properness} used to prove the transversality result in Theorem \ref{SM}. A proof in a 
more general setup will come up in \cite{albi2}. 
\begin{lem}\label{sigprop} Let $x,y \in \Pc_0(H)$ and $\psi$ be the map from \eqref{ppssii}. Consider furthermore the Banach manifold
$\mathcal{Z} = \psi^{-1}(0) \subset \mathcal{C}^1_{x,y}(\R,\M)\times\Kc_{\theta,1}$ and denote by $\pi :\mathcal{C}^1_{x,y}(\R,\M)\times\Kc_{\theta,1} \lo \Kc_{\theta,1}$ 
the projection to the second factor. Then the restriction of $\pi$  to $\mathcal{Z}$ is $\sigma$ - proper. 
\end{lem}
\begin{proof} We have to show that there exists a countable family of sets $\{M_k\}_{k\in\N}$ which cover $\mathcal{Z}$ such that $\pi_{|M_k}$ is proper. 
The main problem which occurs is that there can be breaking. The $M_k$ are chosen in a way such that breaking does not happen. That is, for 
$\ee>0$ we choose balls $B_{\ee}(x)$ and $B_{\ee}(y)$ of radius $\ee$ centered at $x$ and $y$ and for $N_1,N_2 \in \N$  we denote by $M_{N_1,N_2,\ee}$ 
the set which consists of those $(v,K) \in \mathcal{Z}$ such that
\begin{enumerate}
 \item $v(s) \in B_{\ee}(x)$ for all $s \leq - N_1$
 \item $v(s) \in B_{\ee}(y)$ for all $s \geq  N_1$
 \item $\nl v^{\prime}(s) \nr_{\HS} \geq \frac{1}{N_2}$ for all $s \in [-N_1,N_1]$
\end{enumerate}
Obviously for each fixed $\ee$ we have that $$\mathcal{Z} = \bigcup \limits_{(N_1,N_2)\in \N^2}M_{N_1,N_2,\ee}\,.$$ Furthermore we can assume that $\ee$ is chosen small 
enough such that the balls at $x$ and $y$ contain no other critical points than $x$ or $y$ respectively. Now let $(v_n,K_n) \in M_{N_1,N_2,\ee}$ and assume that 
$K_n$ converges to $K \in \pi(M_{N_1,N_2,\ee}) \subset \Kc_{\theta,1}$ then we claim that a subsequence of $v_n$ converges to some $v \in \mathcal{C}^1_{x,y}$. To see this let 
$$\mathbb{W}_{N_1,N_2,\ee} = \Big(\WW^u_{\X_K}(x)\cap \WW^s_{\X_K}(y)\Big) \cap \li\{ u(s) \mi (u,K) \in M_{N_1,N_2,\ee}\,,  s\in \R \re \} $$ 
where $\WW^u_{\X_K}(x),\WW^s_{\X_K}(y)$ denote the unstable / stable manifolds of $x$ and $y$ with respect to $\X_K = \X +K$. Since the set $M_{N_1,N_2,\ee}$ is closed and 
$\WW^u_{\X_K}(x)\cap \WW^s_{\X_K}(y)$ is precompact due to Theorem \ref{compactM}, we have that $\mathbb{W}_{N_1,N_2,\ee}$ is compact. Now each $v_n$ solves 
$ v_n^{\prime}(s) = \X_{K_n}\big(v_n(s)\big)$ for all $ s \in [-N_1,N_2] $ and 
$$ v_n(s) \in B_{\ee}(x) \quad \text{for} \quad s \leq -N_1 \quad v_n(s) \in B_{\ee}(y) \quad \text{for} \quad s \geq N_2\,. $$
So by the continuous dependence of ODE's on the initial data see for instance  \cite{lang}, we have that $\dist\big(v_n(\R),\mathbb{W}_{N_1,N_2,\ee}\big)$ becomes arbitrary small as long as 
$\ee$ is chosen small enough and $n$ large enough such that $K_n$ becomes close enough to $K$. In other words we can find a sequence of trajectories 
$w_n \in  \mathcal{C}^1_{x,y}(\R,\M)$ with respect to $\X_K$ with trace in $\mathbb{W}_{N_1,N_2,\ee}$ such that $\dist\big(v_n,w_n\big)$ becomes infinitesimal.
Since the sequence $w_n$ is compact we can choose a subsequence $v_{n_k}$ which converges in $C^0$ to a curve $v \in  \mathcal{C}^0_{x,y}(\R,\M)$. 
Since all curves $v_{n_k}$ are trajectories  of a $C^3$- vector field we finally obtain the claimed convergence in $\mathcal{C}^1_{x,y}(\R,\M)$.
\end{proof}

\subsection{The boundary operators}
By the previous sections we can assume that the perturbed vector field $\X_K = -\Ah +K$, $ K \in \Kc_{\reg}$, is a $C^3$- Morse vector field and satisfies the Morse-Smale 
condition up to order 2. Furthemore the compactness statements of the previous section enable us to construct a Morse complex for the action functional $A_H$. 
Since the functional $A_H$ is strongly indefinite and therefore the stable and unstable manifolds are infinite dimensional they cannot be oriented. Nevertheless
it is possible to introduce an oriantion setup by orienting the Fredholm pairs $\big(T_x\WW^s(x), \V(x)\big)$, $x \in \crit(A_H)$, where $\V= \M\times(\R^n\times\Hm^+)$.
Since this paper focuses on the Fredholm property and the compactness statements for the hybrid type curves we omit a discussion of orientation here and restrict ourselves 
to the case of $\Z_2$- coefficiants.\\

Let $x,y \in \crit(A_H)$, $\V = \M \times (\R^n \times\Hm^+)$ the constant subbundle which is admissible for $\X_K$, $K \in \Kc_{\theta,1}$ in the sense of 
Definition \ref{admissible} and assume that $m(x,\V)-m(y,\V) = 1$. Hence due to Corollary \ref{onedim} the intersection $\WW^u(x) \cap \WW^s(y)$ is 1-dimensional, compact
and consists of all the connecting trajectories between $x$ and $y$. If we denote by 
$$\rho(x,y):= \# \Big(\big(\WW^u(x)\cap\WW^s(y)\big)/\R \Big)\mod 2 $$
then $\rho(x,y)$  is well defined for all $x,y \in \Pc_0(H)$ and the 
{\bf Morse complex} for the action functional $A_H$ on the Hilbert manifold $\M$ can be constructed in the usual way. We consider the abelian groups 
$$C_k(H) = \bigoplus_{x \in \crit_k(A_H)}\Z_2 x \,,$$
where $\crit_k(A_H) $ is the set of all $x \in \crit(A_H)$ with $m(x,\V) =k$
and define the boundary operators  $\p^M_k : C_k(H) \lo C_{k-1}(H)$ by requiring that 
\begin{equation*}
 \p^M_k x = \sum_{y\in \crit_{k-1}(A_H)}\rho(x,y)y
\end{equation*}
for all $x \in \crit(A_H)$ with $m(x,\V) = k$. Now a gluing statement detailed proven in \cite{Al5} by using the graph transform method shows that indeed 
$\p^M_{k-1}\circ \p^M_k = 0$ holds for all $k \in \Z$ and therefore the above setup describes a complex. The homology of this complex is given by 
$$ HM_k:= \frac{\ker \p^M_k}{\ran \p^M_{k+1}} $$   
and is called the {\bf  Morse homology }  of $\M$ with respect to $(A_H,\X_K)$. If we replace $\X_K$ by another vector field 
$\X_{K^{\prime}}$, $K,K^{\prime} \in \Kc_{\reg}$, which therefore still satisfies conditions 
{\bf C1} and {\bf C2} with respect to the same bundle $\V= \M\times(\R^n\times\Hm^+)$ leads to an isomorphic chain complex. The argument is similar to the one 
in Theorem 2.26 in \cite{Al1}. Thus the homology is independent of the chosen perturbation $K$  and can be denoted by $HM_*(A_H)$. By the same argument
the homology is independent of the particular choice of $H$, which asserts why we could expect all critical points of $A_H$ to be contractible.

\section{The Floer complex in the $H^1$-setup} \label{Floer}

The construction of the Floer complex on a {\bf compact symplectic} and {\bf monotone manifold} was the essential tool of Floer's proof of the {\bf Arnold conjecture} 
and we refer to \cite{floer3} for further details. Since we use the $H^1$-setup instead of a $W^{1,p}, p >2$ setup we find it convenient to restrict ourselves to 
the case of a constant, almost complex structure $J_0$ and a generically chosen Hamiltonian $H$ such that we achieve transversality. In this situation we can prove an 
elliptic regularity statement which enables us to show that the moduli spaces are compact in the usual $C^{\infty}_{\loc}$- Gromov sense. 
As a preparation for the next chapter we  prove an alternative $H^1_{\loc}$-compactness statement.

\subsection{The setup}\label{floer1.1}

By Proposition \ref{reg1} we can assume we have chosen a smooth non-degenerate Hamiltonian $H \in \Hr \subset C^{\infty}(\mS^1\times\Tn,\R)$  on the $2n$-dimensional 
torus $\Tn=\Rn /\Zn$, $n\in\N$ with associated Hamiltonian vector field  $X_H$  defined with respect to the standard symplectic structure 
$\omega_0 = \sum \limits_{i=1}^n dy_i\wedge dx_i$. Moreover we denote by
\begin{equation}\label{JJO}
J_0 =\li ( \begin{array}{rc}
 0 & 1 \\
-1 & 0
\end{array}\re)
\end{equation}
the special complex structure on $\Tn$  such that $(\Tn,\omega_0,J_0)$ becomes a K\"ahler manifold with 
$$ \omega_0(\cdot,J_0\cdot) = \li<\cdot,\cdot\re>\,.$$
Instead of studying gradient-like trajectories of the Hamiltonian action 
$$A_H : \M \lo \R $$
defined by 
$$A_H(x)= -\frac{1}{2}\li(\nl \PP^+x \nr^2_{\HH} - \nl \PP^-x \nr^2_{\HH}\re) +\int_0^1 H\big(t,x(t)\big)dt\,,$$
which possesses the same smooth and contractible critical points, denoted by $\Pc_0(H)$, as its restriction to 
component of contractible loops of the free loop space $\Lambda_0(\Tn)$, introduced in \eqref{tildeAH}, we are  now  interested in solutions 
of a Cauchy-Riemann type partial differential equation. As it will turn out, the solutions  $u = u(s,t) : \R\times \mS^1 \lo \Tn $ of the {\bf Floer-equation}, 
which since the tangent bundle of $\Tn$ is trivial, $\mathrm{T}\Tn \cong \Tn \times \Rn$,  can be written as  
\begin{equation}\label{dh}
\dH \big(u(s,t)\big) = \p_su(s,t) +J_0\big(\p_t u(s,t) -X_H(t,u(s,t))\big) = 0\,,
\end{equation}
give rise to  a set of countable moduli spaces. To give a precise definition of these moduli spaces we first have to make some remarks on the chosen analytical setting. 
As usual,  for a domain  $\Omega \subset \R^2$, we denote by $W^{k,p}(\Omega,\Rn)$ the Sobolev-spaces of all $L^p$-curves possessing weak derivatives up to order $k$ 
and denote by 
$$ \nl u \nr_{W^{k,p}(\Omega)} = \li(\sum_{|\alpha|\leq k} \nl D^{\alpha}u\nr_{L^p(\Omega)}^p  \re)^{1/p} \,, \quad \alpha \in \N^2 $$
the usual norm. By $C^k_0(\R\times\mS^1,\Rn)$ we denote the Banach space of all $C^k$-curves $u$ with 
$$ \lim_{s \rightarrow \pm\infty} D^{\alpha} u(s,t) = 0 \,, \quad \text{for all}  \quad  t \in \mS^1 \,,  \alpha \in \N^2\,, \quad \text{with} \quad |\alpha|\leq k  \,.$$ 
Let $u \in \WRR $ since 
\begin{equation*}\label{HLLH}
 \WRR = L^2(\R,H^1(\mS^1,\Rn))\cap H^1(\R,L^2(\mS^1,\Rn))
\end{equation*}
the Fourier representation of $u$ is given by 
$$ u(s,t) = u_0(s) + \sum_{k \not = 0 }e^{2\pi k J_0t}u_k(s)\,, \quad u_k \in H^1(\R,\Rn) \,.$$
As a consequence of the Sobolev embedding theorem, see \cite{adams-2}, similarly to Proposition \ref{hof1}, all the $u_k$ belong already to $C^0_0(\R,\Rn)$ and satisfy 
\begin{equation}\label{H1}
|u_0(s)|^2 + 4\pi^2\sum_{k\in\Z}|k|^2|u_k(s)|^2< \infty 
\end{equation}
almost everywhere. Now we identify those curves whose image on the torus is the same and denote by 
$$\WRTl:= \WRRl/\Zn \,.$$
Consider the evaluation map 
$$ \widetilde{\ev}_{0} : \CoR \lo C^{\infty}(\mS^1,\Rn) \,, \quad u(\cdot,\cdot) \mapsto u(0,\cdot) \,.$$ 
Then there holds:
\begin{lem}\label{ev} The extension of $\widetilde{\ev}_{0}$ to 
$$\ev_{0} : H^1([0,\infty)\times \mS^1,\Rn) \lo \Ee\,, \quad u(\cdot,\cdot) \mapsto u(0,\cdot) \,,\quad  \Ee=H^{1/2}(\mS^1,\Rn)$$ 
is a well defined smooth submersion. In particular there holds
$$ \nl \ev_{0} (u) \nr_{\HH} \leq \sqrt{2}  \nl u \nr_{H^1([0,\infty)\times\mS^1)} \,.$$ 
 \end{lem}
Though one can find several generalizations of this statement, for example in \cite{conny}  or \cite{Taylor}, we  give 
a short proof for our particular situation.
\begin{proof} 
Let $v \in \CoR$ and  $v =\sum \limits_{k\in\Z}e^{2\pi kJt}v_k(s)$ be its Fourier series. Observe that 
\begin{equation}
\li< x, y \re>_{H^{1/2}(\mS^1)} = \li<x_0,y_0\re> + \li< x, J_0\big(\p_t \PP^-y-\p_t \PP^+ y\big) \re>_{L^{2}(\mS^1)} 
\end{equation}
for all $x,y \in C^{\infty}(\mS^1,\Rn)$, where $\PP^{\pm}$ denote the projections from \eqref{pro}.
Hence
\begin{align*}
\nl v(0,\cdot)\nr^2_{\Ee} 
&=     \Big| \int_{0}^{\infty} \p_s \nl v(s,\cdot) \nr^2_{\HS}ds\Big|\\
&=    2\Big|\int_{0}^{\infty}  \li< \p_s v, v \re>_{H^{1/2}(\mS^1)}ds\Big| \\
&=    2\Big|\int_{0}^{\infty} \int_0^1\Big(\li<\p_s v_0,v_0\re> + \li< \p_s v, J_0\big(\p_t \PP^- v-\p_t \PP^+ v\big) \re>\Big)dtds\Big| \\
&\leq 2\nl v\nr^2_{H^1([0,\infty)\times \mS^1)} \,.
\end{align*}
By density  the estimate has to hold for all $u \in H^1([0,\infty)\times\mS^1,\Rn)$  and enables us to extend $\widetilde{\ev}_{0}$ to $\ev_{0}$.
Moreover $\ev_0$ is a linear operator which shows continuity and the existence of all Fr$\acute{\mathrm{e}}$chet derivatives. Now, if 
$x = \sum \limits_{k\in \Z}e^{2\pi kJ_0 t}x_k \in \Ee$ then it is readily seen that for instance the curve 
$$\xi(s,t) = \sum \limits_{k\in \Z}e^{2\pi kJ_0 t}e^{-s|k|}x_k$$
belongs to $H^1([0,\infty)\times \mS^1,\Rn)$ and is therefore a pre-image under $D\ev_0$ of $x$.
\end{proof}
Now we can define the { \bf moduli spaces of Floer-cylinders}. That is : 
\begin{defn}\label{defF} Let $H\in \Hr$ and  $x^-,x^+ \in \Pc_0(H)$  and consider the following set
\begin{align*}
 \MM_F(x^-,x^+,H,J_0):=\Big\{u \in \WRTl \mi & \dH (u) = 0 \,\, \text{with} \\
					    & \lim_{s\rightarrow \pm \infty} u(s,\cdot) = x^{\pm} \Big\}\,,
\end{align*}
with $\dH$ from \eqref{dh}. Here $ \dH (u) = 0$ should be  understood in the weak sense, and the convergence to the orbits at infinity should be understood as follows.
Let $\dot\Omega_T: = (-T,T)\times\mS^1$ and write
$$u(s,t) = [u_0(s)] + \sum_{k \not = 0 }e^{2\pi k J_0t}u_k(s,t)\,, $$
where $[u_0] \in C^0((-T,T),\Tn)$, $u_k \in C^0((-T,T),\Rn)\,, |k|\geq 1 $ be the Fourier representation for the restriction of $u$ to $\dot\Omega_T \subset \R\times\mS^1$. 
Let $ c_{x^-,x^+} \in \CC$ be a reference cylinder such that  
\begin{equation}\label{vv}
 c_{x^-,x^+}(s,\cdot) = \li\{
\begin{array}{ll}
 x^-\,, & \text{for} \quad s\leq-1\\
 x^+\,, & \text{for} \quad s \geq  1\quad.
\end{array}\re.
\end{equation}
Then using the additive structure of $\Tn$ and the evaluation map from Lemma \ref{ev} we say that $\lim \limits_{s\rightarrow \pm \infty} u(s,\cdot) = x^{\pm}$ 
if and only if 
\begin{equation}\label{con1}
 \lim_{s\rightarrow \pm \infty}(\PP^0)^{\perp}\ev_s(u -  c_{x^-,x^+}) = 0 \quad \text{in} \,\, \Hm\,,
\end{equation}
where $\PP^0$ denotes the restriction to the constant loops defined in \eqref{pro} and 
$$ \lim_{s\rightarrow \pm \infty}\PP^0\ev_s(u -  c_{x^-,x^+}) = 0 \quad \text{on} \,\, \Tn \,.$$
\end{defn}
Next we show that the solutions $u \in  \MM_F(x^-,x^+,H,J_0)$ are actually smooth. For this,  we use the inner regularity property of the Laplace operator which is 
proven by using the Calderon-Zygmund inequality. For details see \cite{duff} (Appendix B).
% \begin{thm}\label{cal}{\bf (\cite{duff})} Let $1<p <\infty$, $k\geq 0$ be an integer, and $\Omega \subset \R^n$ be an open domain. If $u \in L^1_{\loc}(\Omega)$ is a weak
% solution of 
% $$ \Delta u = f\,, \quad f \in W^{k,p}_{\loc}(\Omega)\,,$$
% then $u \in  W^{k+2,p}_{\loc}(\Omega)$. Moreover, for every bounded subset $\Omega^{\prime}$ whose closure is contained in $\Omega$ there exists a constant 
% $c = c(k,p,n,\Omega^{\prime},\Omega)>0$ such that, for every $u \in C^{\infty}(\overline \Omega)$, 
% $$ \nl u \nr_{W^{k+2,p}}(\Omega^{\prime}) \leq c\Big(\nl \Delta u \nr_{W^{k,p}(\Omega)} + \nl u \nr_{L^{p}(\Omega)} \Big)\,. $$
% \end{thm}
\begin{lem} \label{el1}{\bf (local regularity)} Let $u \in \WRTl$ be a weak solution of $\dH(u)=0$ on every open and bounded domain 
$\Omega \subset \R\times\mS^1$. Then $u \in C^{\infty}(\R\times \mS^1,\Tn)$.
\end{lem}
\begin{proof}
Since $u \in \WRTl$ and $H \in C^{\infty}(\mS^1,\Tn)$ we have that $J_0X_H(u)$ is of class $H^1_{\loc}$. Denote by $\p_{J_0} = \p_s -J_0(\cdot)\p_t$ and by 
$\dd = \p_s +J_0\p_t$. Then $u$ solves  
$$ \Delta u = \p_{J_0}\circ \dd (u) = \p_{J_0} \circ  \big(J_0(u) X_H(u)\big) =:f \in L_{\loc}^{2}(\R\times\mS^1,\Rn) $$ 
in the weak sense on every open and bounded domain $\Omega \subset \R \times \mS^1$. Since every $L^2_{\loc}$-curve is certainly of class $L^1_{\loc}$, any lift of
$u$ to a curve in $\WRR$ has to be  be  a $H_{\loc}^{1,2}$-curve due to the regularity property of the Laplace operator. Hence $u$ is of class $H^1$, which implies that 
$J_0X_H(u)$ is of class $H^1$  and therefore $u$ is of class $H^{3}$. Iterating this argument shows that $u$ is a $H_{\loc}^{\infty}$-curve and therefore, by 
Sobolev embedding, of class $C^{\infty}$.
\end{proof}
\begin{defn} Let $u : \R \times \mS^1 \lo \Tn$ be a smooth solution of \eqref{dh}. We define the {\bf energy} of $u$ by 
\begin{equation*}
 E(u):= \frac{1}{2}\int_{-\infty}^{\infty}\int_0^1 \li(\li|\frac{\p u}{\p s}\re|^2 + \li|\frac{\p u}{\p t} - X_H(u)\re|^2\re)dsdt \,.
\end{equation*}
\end{defn}
For $T>0$ one easily verifies that
\begin{equation*}
E_T(u) 
:=
\int_{-T}^{T}\int_0^1 |\p_su(s,t)|^2dtds 
% &= 
% -\int_{-T}^{T}\int_0^1 \li<\p_su(s,t),J_0\big(\p_t u(s,t) -X_H(u(s,t)\big)\re>dtds \nonumber\\
% &=  
% -\int_{-T}^{T} \p_s  A_H(u(s,\cdot))ds\nonumber\\
= 
A_H(u(-T,\cdot)-A_H(u(T,\cdot))\,. 
\end{equation*}
Now recall that the action $A_H$ is smooth on $\M$ and therefore
\begin{align*}
E(u)  & = \lim_{T\rightarrow-\infty} A_H\big(\ev_T(u)\big) -\lim_{T\rightarrow+\infty} A_H\big(\ev_T(u)\big)\nonumber \\
      & = A_H(x^-)-A_H(x^+)  <\infty\,.    
\end{align*}
Now by Lemma \ref{el1} every $u \in \MM_F(x^-,x^+,H,J_0)$ is smooth and satisfies $E(u)< \infty$, which as it is well-known, see \cite{sal}, is equivalent to the fact that  there exist constants $c=c(u),\delta=\delta(u) >0$ such that
\begin{equation}\label{en}
  |\p_su(s,t)| \leq ce^{-\delta|s|}
\end{equation}
for all $(s,t)\in \R\times \mS^1$,
which again is equivalent to the fact that the limits 
\begin{equation}\label{en2}
\lim_{s\rightarrow \pm \infty} u(s,\cdot)= x^{\pm}\,, \qquad \lim_{s\rightarrow \pm \infty} \p_su(s,\cdot)= 0
\end{equation}
are uniform in $t$.
Thus in particular our choice of the convergence behavior at infinity turns out to be equivalent to 
the classical uniform convergence. \\ 

The next result states that the moduli spaces are affine translations of subsets of $\WRR$ in $\WRTl$, which will be crucial 
especially for our compactness statement in Theorem \ref{compact}.  
\begin{prop} \label{Pfi} Let $c_{x^-,x^+}$ be as in \eqref{vv}. Then the non-linear operator
\begin{align} 
 \Phi_{c_{x^-,x^+}} : \WRR &\lo \LRR \,, \label{phi}\\ 
w &\mapsto \dH(w + c_{x^-,x^+}) \,. \nonumber
\end{align}
is well defined, smooth, and there holds
\begin{equation}\label{aff}
 \MM_F(x^-,x^+,H,J_0) = c_{x^-,x^+} +  \Phi^{-1}_{c_{x^-,x^+}}(0) \subset \WRTl\,. 
\end{equation} 
\end{prop}
\begin{proof} We show that $\Phi_{c_{x^-,x^+}}$ is well defined. For that we use the estimate 
\begin{align*}
 \nl  \Phi_{c_{x^-,x^+}}(w) \nr^2_{L^2(\R\times\mS^1)} & \leq \nl \p_s (w+c_{x^-,x^+})\nr^2_{L^2(\R\times\mS^1)}\\ 
                                                       & +    \nl \p_t (w+c_{x^-,x^+})-X_H(w+c_{x^-,x^+})\nr^2_{L^2(\R\times\mS^1)}  
\end{align*}
and show that both terms are bounded. By our assumptions on $c_{x^-,x^+}$, we have that for $T >1$ and $S_T = [T,+\infty) \times\mS^1$ there holds
$$ \nl \p_s\big(w+c_{x^-,x^+}\big)\nr_{L^2(S_T)} = \nl \p_s w\nr_{L^2(S_T)}  \leq \nl w\nr_{H^1(\R\times\mS^1)} \,.$$ 
Now recall that $ X_H = J_0 \nabla H$; by the mean value theorem we obtain by setting $u = w +c_{x^-,x^+}$
\begin{align*}
\nl \p_tu -X_H(u)\nr_{L^2(S_T)} & =  \nl \p_tw+\p_t x^+ -X_H(w+x^+)\nr_{L^2(S_T)}\\
							    &\leq \nl \p_tw \nr_{L^2} + \nl H\nr_{C^2(\mS^1\times\Tn)}\nl w\nr_{L^2(S_T)}\,.
\end{align*}
The analogous estimates for $T<-1$  show that both terms are bounded on the cylindrical ends and therefore on the whole cylinder.
Since smoothness can be proven by using the same argument as used in the proof of Theorem \ref{hof3} we omit a detailed discussion here.\\

Finally since the linear operator $\dd : \WRR \lo \LRR$ is smooth, we have that $\Phi_{c_{x^-,x^+}}$ is smooth.\\

It remains to  prove \eqref{aff}. Therefore let $ u \in \MM_F(x^-,x^+,H,J_0)$ and set $w = u -c_{x^-,x^+}$. Recall that by Lemma \ref{el1} $u$ and $c_{x^-,x^+}$ 
are smooth and so is $w$. By the decay in \eqref{en} this leads to 
$\p_s w \in \LRR$. With $ w(s,t) = -\int \limits_s^{\infty}\p_{\tau}w(\tau,t)d\tau $ we  again use \eqref{en} and observe that at the cylindrical ends,
i.e., for $R >>1$ sufficiently large there holds 
\begin{align*}
\int_{R}^{\infty}\int_0^1|w(s,t)|^2dtds &   =  \int_{R}^{\infty}\int_0^1\Big|\int_s^{\infty}\p_{\tau}w(\tau,t)d\tau\Big|^2dtds \\
                                        & \leq \int_{R}^{\infty}\int_0^1\li(\,\int_s^{\infty}|\p_{\tau}w(\tau,t)|d\tau\re)^2 dt ds\\
                                        & \leq \frac{c^2}{\delta^2}\int_{R}^{\infty}e^{-2|s|\delta}ds<\infty \,, 
\end{align*}
and analogously on the other end. Hence $w \in \LRR$. Now for all $(s,t) \in \R \times \mS^1$ with $|s|>1 $ we get 
\begin{align*}
|\p_t w(s,t)|  &  =   \li|\p_s u(s,t) + J_0\Big(\p_t x^{\pm}(t) - X_H\big(w(s,t) +x^{\pm}(t)\big)\Big)\re| \\
               &\leq \li|\p_s u(s,t)\re| +\nl H\nr_{C^2(\mS^1\times \Tn)}\li|w(s,t)\re|\,.
\end{align*}
Hence $\p_t w \in \LRR$ and every $u \in  \MM_F(x^-,x^+,H,J_0)$ can be written as 
$$u  = w+c_{x^-,x^+}\,, $$
for some $w \in \Phi^{-1}_{c_{x^-,x^+}}(0)\subset \WRR$. 
\end{proof}

\subsection{The Fredholm-problem and Transversality}\label{transF}\label{secfred}

In this section we prove that for 1-periodic solutions $x^-,x^+ \in \Pc_0(H)$ the associated maps $\Phi_{c_{x^-,x^+}}$ are Fredholm, i.e.,  
their differentials are Fredholm operators. Furthermore for a generic set of smooth Hamiltonians $H$,
zero will be a regular value for $\Phi_{c_{x^-,x^+}}$. \\

Consider the operator $F : H^1(\R \times \mS^1, \Rn)\lo L^{2}(\R \times \mS^1, \Rn)$ defined by 
\begin{equation}\label{F}
 F\xi = \p_s\xi +J_0\p_t\xi + S\xi\,, \quad J_0\in\J \quad \text{as in} \,\, \eqref{JJO}\,,
\end{equation}
where $S$ is assumed to be a continuous symmetric matrix valued function, i.e.,  $S \in C^0(\R\times\mS^1,\Rn\times\Rn)$, with
\begin{equation}\label{konver0}
\lim \limits_{s\rightarrow+\infty }S(s,t) = S^+(t) \in C^0(\mS^1,\Rn\times\Rn) 
\end{equation}
uniformly in $t$. We define a symplectic path $\Psi^+ \in C^1([0,1],\Sp(2n,\R))$ by 
\begin{equation*}
 \p_t \Psi^+(t) = J_0S^+(t)\Psi^+(t)\,, \quad \Psi^+(0) = \Id\,. 
\end{equation*}
 and recall that the  {\bf Conley-Zehnder index} $\mu(\Psi^+)$ is well defined if 
\begin{equation}\label{gen}
\det\big(I-\Psi^+(1)\big) \not = 0\,. 
\end{equation}
% where $S \in C^0(\R^2,\Rn\times\Rn)$ is a continuous matrix valued function on $\R^2$ satisfying 
% $$ S(s,t) = S(s,t)^T = S(s,t+1)\,, \quad \text{ for all } \quad (s,t) \in \R^2 \,.$$ 
% Furthermore we assume that  $S(s,t)$ converges uniformly in $t$ , as $s$ tends to $\pm \infty$ and we set 
% \begin{equation}\label{konver0}
%  S^{\pm}(t):=\lim_{s \rightarrow \pm \infty}S(s,t)\,. 
% \end{equation} 
% Now we use the standard complex structure $J_0$ to define paths $\Psi^{\pm} : [0,1] \lo \Sp(2n,\R)$ in the symplectic group by 
% $$ \p_t \Psi^{\pm}(t) = J_0S^{\pm}(t)\Psi^{\pm}(t)\,, \quad \Psi^{\pm}(0) = \Id\,.$$
% We set 
% \begin{equation}\label{gen}
% P^*:=\li\{ \gamma \in C^0\big([0,1],\Sp(2n,\R)\big) \mi \gamma(0) = I\,, \quad   \det\big(I-\gamma(1)\big) \not = 0 \re\}\,.
% \end{equation}
% If $\Psi^{\pm}\in P^*$ then the  {\bf Conley-Zehnder index} $\mu :  P^* \lo \Z $ is well defined. 
We will not give a further discussion here, 
for details see \cite{salduff},\cite{salrob},\cite{salzeh}. For given nondegenerate Hamiltonian $H$ and $x\in \Pc_0(H)$ we abbreviate by 
$\mu(x)$ the index of the associated symplectic path of $\nabla^2H\big(\cdot,x(\cdot)\big)$.\\

For the next well-known statement we refer to \cite{salzeh}; alternative proofs can be found, for instance, in \cite{floer2},\cite{locke}, \cite{schwarz}.

\begin{thm}\label{floerfred}{\bf (\cite{salzeh})} Let $S \in C^0(\R^2,\Rn\times\Rn)$ and assume that  \eqref{konver0} and \eqref{gen} hold. Then 
$F$ is a Fredholm operator with index 
$$ \ind (F) = \mu(\Psi^-) - \mu(\Psi^+)\,.$$
\end{thm}

\begin{cor}\label{cor2} Let $x^-,x^+ \in \Pc_0(H)$ be two contractible, 1-periodic solutions, $w\in \WRR$, $c_{x^-,x^+}$ as in \eqref{vv} and $H \in \Hr$. 
Then the linearized operator  
\begin{equation} \label{Phiop}
D\Phi_{c_{x^-,x^+}}(w): \WRR \lo \LRR 
\end{equation}
$$D\Phi_{c_{x^-,x^+}}(w)\xi =   \p_s\xi +J_0\p_t\xi  + \nabla^2H(\cdot,w+ c_{x^-,x^+})\xi $$ 
is a Fredholm operator with index 
$$ \ind \big(D\Phi_{c_{x^-,x^+}}(w)\big) = \mu(x^-) - \mu(x^+)\,.$$
\end{cor}
\begin{proof}Let $w \in \WRR$ then we can approximate $w$ by a smooth curve $v$ and consider the operator 
$$ \widetilde{D\Phi}_{c_{x^-,x^+}}(v)\xi =   \p_s\xi + J_0(v+c_{x^-,x^+})\p_t\xi  + \nabla^2H(\cdot,v+c_{x^-,x^+})\xi $$ 
Since $H \in \Hr$  the operator $F$ in \eqref{F} with 
$S= \nabla^2H(\cdot,v+c_{x^-,x^+})$ satisfies condition \eqref{gen} and coincides with $\widetilde{D\Phi}_{c_{x^-,x^+}}(v)$. Now $F$ satisfies the assumptions of 
Theorem \ref{floerfred} and is therefore Fredholm of index $\mu(x^-) - \mu(x^+)$. Since $\Phi_{c_{x^-,x^+}}$ is smooth, the set of Fredholm operators 
is open with respect to the uniform operator topology and the index is locally constant,
see \cite{duff}, we have that $D\Phi_{c_{x^-,x^+}}(w)$ is also Fredholm and of the same index as $F$.
\end{proof}
The statement has an immediate consequence. 
\begin{cor} Let $H_0 \in \Hr \subset C^{\infty}(\mS^1\times\Tn,\R)$ in the sense of  
Proposition \ref{reg1}. Then there exists a residual set $\Hrr(H_0)  \subset C^{\infty}(\mS^1\times\Tn,\R)$ of Hamiltonians that coincide with $H_0$ on $\Pc_0(H)$ 
up to second order such that for all $x^-,x^+ \in \Pc_0(H_0)$ with $x^-\not = x^+$ the moduli spaces 
$\MM_F(x^-,x^+,H,J_0)$ are smooth manifolds of dimension 
$$ \dim\big(\MM_F(x^-,x^+,H,J_0)\big) = \mu(x^-)-\mu(x^+) \,.$$ 
\end{cor}
\begin{proof}
By the transversality result of A. Floer, H. Hofer and D. Salamon, see \cite{FHS},
there exists a residual set of Hamiltonians $\Hr(H_0) \subset C^{\infty}(\mS^1\times \M,\R)$  which coincide with $H_0$ up to second order such that $(H,J_0)$ is a {\bf regular pair} for all $H \in \Hr(H_0)$, i.e.,
there are no degenerate 1-periodic solutions and the associated operator $F_{J_0,H}(u)$ from \eqref{F} with $S= \nabla^2H(\cdot,u)$
is onto for every $u \in \MM_F(x^-,x^+,H,J_0)$ with $x^- \not = x^+$. Hence we conclude the proof by the implicit function theorem, see \cite{duff}.
\end{proof}

To be precise, the proof in \cite{FHS}  uses  the Sard-Smale-Theorem \ref{sard}, which requires a Banach manifold setting. 
For this reason the  result   is proven with respect to the  $W^{1,p}$, $p>2$ setup to assure that the  spaces $W^{1,p}(\R\times\mS^1,M)$ are Banach manifolds, 
which would not hold if one loses the embedding into $C^0(\R\times\mS^1,M)$. In the special case where $M$ possesses a trivial tangent bundle, as in our case,
the spaces $H^1(\R\times\mS^1,M)$ and $L^{2}(\R\times\mS^1,M)$ are Hilbert manifolds. Thus the above result is not sensitive to the choice of 
the functional spaces and can therefore be applied.\\

Finally we have to treat the case when $x^- = x^+$. That is : 

\begin{lem} Assume we have chosen a regular pair $(H,J_0)$ and  $x \in \Pc_0(H)$. Then the moduli space $\MM_{F}(x,x,H,J_0)$ 
is a zero dimensional manifold which consists only of the constant solution. Moreover the linearized operator 
$$D\Phi_{c_{x,x}}(w)\xi =   \p_s\xi +J_0\p_t\xi  + \nabla^2H(\cdot,w+ c_{x,x})\xi \,, \quad c_{x,x}(s,\cdot) =x \quad \forall s \in \R \,,  $$ with $w= x-c_{x,x} = 0 \in  H^1(\R\times\mS^1,\Rn)$, 
is an isomorphism. 
\end{lem}
\begin{proof} Let $u \in H^1_{\loc}(\R\times\mS^1,\Tn)$ be a solution connecting $x$ with itself. Then we have that 
$$ E(u) = \int_{-\infty}^{\infty}\nl \p_s u(s,\cdot)\nr_{\LS}^2ds = A_H(x) - A_H(x) = 0 $$
which implies that $u$ is constant as well proving that $\MM_{F}(x,x,H,J_0)$ consists only of the constant 
solution and is therefore a submanifold of $H^1_{\loc}(\R\times\mS^1,\Tn)$. Now $D\Phi_{c_{x,x}}(w)$ is Fredholm of index zero due to Corollary \ref{cor2}. 
Let $\xi \in \ker D\Phi_{c_{x,x}}(w)$ then we have that 
$$ \xi(s,t) = e^{-(J_0\p_t+S)s}\xi(0,t)\,, \quad \text{with} \quad S(t) =\nabla^2 H(t,x(t))\,.$$
Using the fact that $\xi(s,t) \lo 0 $ for $s \lo \pm\infty $
we observe that $\xi = 0$.
Hence the kernel of $ D\Phi_{c_{x,x}}(w)$ is trivial implying that $ D\Phi_{c_{x,x}}(w)$ is an isomorphism, which proves the claimed result.
\end{proof}

\subsection{Compactness}

As seen in section \ref{floer1.1} our choice of the $H^1$-setup does not alter the fact that the moduli spaces consist of smooth curves.
Thus the classical elliptic bootstrapping methods for elliptic operators can be applied and the usual $C^{\infty}_{\loc}$-Gromov-compactness results hold.
In preparation for the compactness statements in chapter 4 we  give an alternative proof. \\

Since we are in the special situation of the torus the following {\bf crucial formulas} can be applied. 
\begin{lem}\label{formform} Let $T>0$, $\Omega_T = [-T,T]\times\mS^1$ and let $u \in H^1(\Omega_T,\Rn)$. Denote furthermore by 
$\overline{\p}_{J_0} = \p_s +J_0\p_t$, $\p_{J_0} = \p_s -J_0\p_t$ with $J_0$ from \eqref{JJO} and by $\PP^{\pm}$ the projections from \eqref{pro} then there holds :  
\begin{align}
 \nl \overline{\p}_{J_0} u \nr^2_{L^2(\Omega_T)} &= \nl \p_su \nr^2_{L^2(\Omega_T)} + \nl \p_t u \nr^2_{L^2(\Omega_T)} \nonumber \\ 
						 &  + \nl \PP^-u(T,\cdot)\nr^2_{\HS} -\nl \PP^+u(T,\cdot)\nr^2_{\HS}\nonumber\\
					         &- \nl \PP^-u(-T,\cdot)\nr^2_{\HS} +\nl \PP^+u(-T,\cdot)\nr^2_{\HS}\,,\label{delbarform}
\end{align}
\begin{align}
 \nl \p_{J_0} u \nr^2_{L^2(\Omega_T)}  &= \nl \p_su \nr^2_{L^2(\Omega_T)} + \nl \p_t u \nr^2_{L^2(\Omega_T)} \nonumber \\
				       &+ \nl \PP^+u(T,\cdot)\nr^2_{\HS} -\nl \PP^-u(T,\cdot)\nr^2_{\HS}\nonumber\\
				       &- \nl \PP^+u(-T,\cdot)\nr^2_{\HS} +\nl \PP^-u(-T,\cdot)\nr^2_{\HS}\,.\label{delbarformm}
\end{align}
\end{lem}
\begin{proof} By Lemma \ref{ev} for any $s_0 \in [-T,T]$ the evaluation $u(s_0,\cdot)\in \M$  of $u$ is a loop which can be represented with respect to $J_0$ by  
$$ u(s_0,t) = [u_0(s)]+ \sum_{k \not = 0}e^{2\pi k J_0 t}u_k(s_0)\,, \quad  [u_0(s)] \in \T^n\,, u_k(s_0) \in \Rn  $$
Hence 
\begin{align*}
2\li<\p_su,J_0\p_tu\re>_{L^2(\Omega_T)}  &= 2\int_{-T}^T \li<\p_su(s,\cdot), \PP^-u(s,\cdot)\re>_{\HS}ds \\
					 &- 2\int_{-T}^T \li<\p_su(s,\cdot), \PP^+u(s,\cdot)\re>_{\HS}ds \\
                                         &=  \int_{-T}^T \p_s\nl \PP^-u(s,\cdot)\nr^2_{\HS}ds\\
                                         &-  \int_{-T}^T \p_s\nl \PP^+u(s,\cdot)\nr^2_{\HS}ds \\
                                         &=  \nl \PP^-u(T,\cdot)\nr^2_{\HS} -\nl \PP^+u(T,\cdot)\nr^2_{\HS}\\
                                         &-  \nl \PP^-u(-T,\cdot)\nr^2_{\HS} +\nl \PP^+u(-T,\cdot)\nr^2_{\HS}
\end{align*}
and  we conclude the proof by the identities 
\begin{align*}
 \nl \overline{\p}_{J_0} u \nr^2_{L^2(\Omega_T)}
&=&
\nl \p_su \nr^2_{L^2(\Omega_T)} + \nl \p_t u \nr^2_{L^2(\Omega_T)} + \li<\p_su,J_0\p_tu\re>_{L^2(\Omega_T)} \\
\nl \p_{J_0} u \nr^2_{L^2(\Omega_T)}
&=&
\nl \p_su \nr^2_{L^2(\Omega_T)} + \nl \p_t u \nr^2_{L^2(\Omega_T)} - \li<\p_su,J_0\p_tu\re>_{L^2(\Omega_T)} \,.
\end{align*}
\end{proof}

Recall that, by Proposition \ref{Pfi} for any $x^-,x^+ \in \Pc_0(H)$ the moduli space is 
the affine subset 
\begin{equation}\label{moduliF}
\Phi_{ c_{x^-,x^+}}^{-1}(0) + c_{x^-,x^+} = \MM_F(x^-,x^+,H,J_0) \subset \WRTl\,. 
\end{equation}
We wish to show  that $\Phi_{ c_{x^-,x^+}}^{-1}(0)$ is a $H^1_{\loc}$- precompact set. For this we need the following result.
\begin{lem}\label{wbound}
Let $x^-,x^+ \in \Pc_0(H)$, $w \in \Phi_{c_{x^-,x^+}}^{-1}(0)$. Then the constant part $\PP^0w(s,\cdot)$ of $w$ is continous in $s$ and  
there is a constant $C=C(x^-,x^+)>0$ such that
\begin{equation*}
 \nl \PP^0w\nr_{C^0(\R)} \leq C \quad\text{for all}\quad w \in \Phi_{c_{x^-,x^+}}^{-1}(0)\,.
\end{equation*}
\end{lem}
\begin{proof} That $\PP^0w \in C^0(\R,\Rn)$ was already mentioned in our remark in \eqref{H1}. To prove the existence of $C$ we argue by contradiction. 
Assume there is a sequence $w_n \in \Phi_{c_{x^-,x^+}}^{-1}(0)$ such that 
$$ \lambda_n:=  \nl \PP^0w_n\nr_{C^0(\R)} \lo \infty \quad \text{for} \quad n \lo \infty\,.$$ 
We set $\tilde w_n = \frac{1}{\lambda_n}\PP^0w_n$ then $\tilde w_n$ defines    
a sequence of continuous loops $\tilde w_n : \R\cup \{\infty\} \lo \Rn$ starting at $0$, with $\nl \tilde w_n \nr_{C^0(\R \cup \{\infty\})} = 1 $ for all $n \in \N$. 
Now each $\tilde w_n$ solves the equation 
\begin{equation}\label{w0}
 \p_s \tilde w_n = - \frac{1}{\lambda_n}\PP^0\big(\nabla H(w_n +c_{x^-,x^+}) + \p_sc_{x^-,x^+}\big) 
\end{equation}
in the weak sense. By a standard bootstrapping argument each $\tilde w_n$ is actually a strong solution and by \eqref{w0} we obtain even more: 
$$ \lim_{n\rightarrow \infty}\nl \p_s \tilde w_n\nr_{C^0(\R\cup\{\infty\})} = 0\,.$$
Due to the Arzel\`a-Ascoli Theorem the sequence $\tilde w_n$ is compact. Thus there is a subsequence $\tilde w_{n_k}$ which converges uniformly on 
$\R\cup\{\infty\}$ to a constant loop $\tilde w_{\infty}$ with  $\nl \tilde w_{\infty} \nr_{C^0(\R \cup \{\infty\})} = 1 $ and $\tilde w_{\infty}(\infty) = 0$, a
contradiction. Hence the sequence $\lambda_n$ is bounded, which proves the claim.
\end{proof}

\begin{thm}\label{compact} Let $x^-,x^+ \in \Pc_0(H)$, $H\in \Hr$. Then the moduli space of Floer cylinders $\MM_F(x^-,x^+,H,J_0)$ is a  $H^1_{\loc}$-precompact set.
\end{thm}
\begin{proof} We show that $\Phi_{ c_{x^-,x^+}}^{-1}(0)$ is a  $H^1_{\loc}$-precompact set. Then the claim follows immediately from \eqref{moduliF}. \\

{\bf Step 1:}  Let $T>0$, $\Omega_T=[-T,T]\times\mS^1$ and $w \in \Phi_{c_{x^-,x^+}}^{-1}(0)$. We show that $\nl w \nr_{H^1(\Omega_T)}$ is uniformly bounded  
in terms of the energy and of the reference cylinder  $c_{x^-,x^+}$.
Set $u = w +c_{x^-,x^+}$. Then we compute
\begin{align*}
 \nl \p_s w \nr_{L^2(\Omega_T)} & \leq  \nl \p_s u \nr_{L^2(\Omega_T)}  + \nl \p_s c_{x^-,x^+} \nr_{L^2(\Omega_T)}  \\
				& \leq \sqrt{E(u)} + \nl \p_s c_{x^-,x^+} \nr_{L^2(\Omega_T)}  =: c_0(x^-,x^+,c_{x^-,x^+}) \,. 
\end{align*}
and 
\begin{align*}
\nl \p_t w \nr_{L^2(\Omega_T)} &\leq \nl \p_t u -X_H(u) \nr_{L^2(\Omega_T)} +  \nl X_H(u) \nr_{L^2(\Omega_T)} + \nl \p_t c_{x^-,x^+} \nr_{L^2(\Omega_T)}\\ 
                               &\leq \sqrt{E(u)} +  2T\nl H \nr_{C^1(\mS^1\times\Tn)} + \nl \p_t c_{x^-,x^+} \nr_{L^2(\Omega_T)} \\
                               &=:c_1(x^-,x^+,T,H,c_{x^-,x^+})\,.
\end{align*}
Observe that 
\begin{align*}
\nl (\PP^0)^{\perp} w \nr^2_{L^2(\Omega_T)}  & = \int_{-T}^T \sum_{k \not = 0}|w_k(s)|^2 ds \leq 4\pi^2\int_{-T}^T \sum_{k\not = 0}|k|^2|w_k(s)|^2 ds  \\
                                             & = \nl \p_t w \nr^2_{L^2(\Omega_T)} \leq c_1^2 \,,
\end{align*}
where $\PP^0$ denotes the restriction to the $0$-order coefficients. Hence the $t$-derivative already bounds the non-constant part of $w$. 
Denote by $w_0 = \PP^0w$ the constant part of $w$. Then by Lemma \ref{wbound} there is a constant $C>0$ such that $|w_0(0)| \leq C$. Writing 
$$ w_0(s) = \int_{0}^s \p_{\tau} w_0(\tau)d\tau + w_0(0) $$
and applying  H\"older's inequality, we see that 
\begin{align*}
 \nl w_0 \nr_{L^2(\Omega_T)}^2 
&= 
\int_{-T}^{T}\int_{0}^{1}\Big| \int_0^s \p_{\tau} w_0(\tau)d\tau + w(0)\Big|^2dtds \\
&\leq
      2\int_{-T}^{T}\li[|w(0)|^2 + \li( \int_0^s |\p_{\tau} w_0(\tau)|d\tau\re)^2\re]ds \\
&\leq 
4TC^2 + 2 T\int_{-T}^{T} \int_0^s|\p_{\tau} w_0(\tau)|^2d\tau ds \\
&\leq
 4TC^2  + 2 T\int_{-T}^{T}  \nl \p_s w \nr^2_{L^2(\Omega_T)} ds\\
&\leq 
4TC^2 + 4T^2 c^2_0 
\end{align*}
Hence there is a constant $c_2= c_2(x^-,x^+, T,H,c_{x^-,x^+},n)$ bounding $\nl w \nr_{L^2(\Omega_T)}$. 
Setting  $c = (c_0,c_1,c_2)$  we obtain $\nl w \nr_{H^1(\Omega_T)} \leq |c|$ 
and we are done with step 1.\\

{\bf Step 2:} Let $T>T^{\prime}>0$,  $w_1,w_2 \in \Phi_{c_{x^-,x^+}}^{-1}(0)$, set $\delta w = w_1-w_2$ and $u_i = w_i + c_{x^-,x^+}$,
$i=1,2$. We claim there is a constant $C_1 = C_1(T, T^{\prime},H)>0$ such that 
\begin{equation}\label{ester}
\nl \nabla \delta w \nr^2_{L^2(\Omega_{T^{\prime}})} \leq C_1\Big(\nl \dd(u_1) - \dd(u_2) \nr^2_{L^2(\Omega_{T})} + \nl\delta w \nr^2_{L^2(\Omega_{T})}\Big)  \,,
\end{equation}
where 
$$\nl \nabla \delta w \nr^2_{L^2(\Omega_{T^{\prime}})} = \nl \p_s\delta w \nr^2_{L^2(\Omega_{T^{\prime}})} +\nl \p_t\delta w \nr^2_{L^2(\Omega_{T^{\prime}})}\,.$$
Assume that $\eqref{ester}$ holds. Since $\dH (u_i) = 0$ we obtain by the mean value theorem that 
\begin{align}
\nl \dd(u_1) -\dd(u_2)  \nr_{L^2(\Omega_{T})} &  =  \nl  J_0\big(X_H(u_1) -X_H(u_2)\big) \nr_{L^2(\Omega_{T})} \nonumber\\
                                              &\leq \nl H\nr_{C^2(\mS^1\times\Tn)} \nl \delta w \nr_{L^2(\Omega_{T})}\,.  \nonumber
\end{align}
So we can find another constant $C_2 = C_2(T, T^{\prime},H) >0$ such that   
\begin{equation}\label{ester2}
 \nl  \delta w \nr_{H^1(\Omega_{T^{\prime}})} \leq C_2 \nl  \delta w \nr_{L^2(\Omega_{T})}  
\end{equation}
By step 1  the restriction of any sequence $(w_i) \in \Phi_{c_{x^-,x^+}}^{-1}(0)$ to $\Omega_T$ is uniformly bounded. 
Since the embedding of $H^1(\Omega_T,\Rn)$ into $L^2(\Omega_T,\Rn)$ is compact $(w_i)$ possesses a subsequence convergent in $ L^2(\Omega_T,\Rn)$, which by  
\eqref{ester2} has to converge already in $H^1(\Omega_{T^{\prime}},\Rn)$. This shows that $ \Phi_{c_{x^-,x^+}}^{-1}(0)$ is $H^1_{\loc}$-precompact 
and so is $\MM_F(x^-,x^+,H,J_0)$. Thus it remains to prove \eqref{ester}. For this purpose let $\beta$ be a smooth function such that 
$$ 
\beta(s) = \li\{
\begin{array}{crl}
 1 &,& |s| \leq T^{\prime} \\
 0 &,& s \geq T^{\prime} + |T-T^{\prime}|/2\\
 0 &,& s \leq -T^{\prime} - |T-T^{\prime}|/2
\end{array}\re.\,.
$$
Now we apply  the formula \eqref{delbarform} of Lemma \ref{formform} to $\beta \delta w$ and obtain  
\begin{align*}
\nl \p_s\delta w \nr^2_{L^2(\Omega_{T^{\prime}})} +\nl \p_t\delta w \nr^2_{L^2(\Omega_{T^{\prime}})} 
&\leq \nl \p_s\beta\delta w \nr^2_{L^2(\Omega_{T})} +\nl \p_t\beta\delta w \nr^2_{L^2(\Omega_{T})} \\ 
&   = \nl \overline{\p}_{J_0}(\beta \delta w) \nr^2_{L^2(\Omega_{T})}
\end{align*}
We set $C_1 =  2(1+\nl \beta^{\prime}\nr^2_{C^0(\R)})$ to finally get 
$$
\nl \overline{\p}_{J_0}(\beta \delta w) \nr^2_{L^2(\Omega_{T})} \leq C_1 \Big(\nl \big(\overline{\p}_{J_0}(u_1) 
                                                                  -    \overline{\p}_{J_0}(u_2)\big) \nr^2_{L^2(\Omega_{T})} 
								  +  \nl \delta w \nr^2_{L^2(\Omega_T)}\Big)
$$
which proves \eqref{ester} and we conclude the proof.
\end{proof}
As usual we define the notian of {\bf broken trajectories} as follows. 
\begin{defn} Let $x^-,x^+ \in \Pc_0(H)$, $u \in \WRTl$ and $\hat u \in \WRTl/\R$ be the equivalence class of $u$ under the $\R$-action 
$$ \R \times \WRTl \lo \WRTl\,, \quad  (\tau,u) \mapsto u(\cdot +\tau,\cdot) \,.$$ 
Then $\hat u$ is called a {\bf broken trajectory} of order $r$ connecting $x^-$ and $x^+$ if and only if
there are finitely many 1-periodic solutions
$x_i \in \Pc_0(H)$, $i=0,\dots,r$ and associated reparametrization times $\tau_i \in \R$ such that 
$$ u(\cdot +\tau_i,\cdot) \in \MM_F(x_i,x_{i+1},H,J_0)\,, \quad \forall \, i = 0,\dots,r\,, \quad x_0=x^-\,,x_{r+1} =x^+\,.$$
\end{defn}
Using standard methods established in \cite{schwarz-2} and \cite{schwarz}, Theorem \ref{compact} has an immediate consequence. 
\begin{prop} \label{braking} Let $(H,J_0)$ be a regular pair and let furthermore 
$x^-,x^+ \in \Pc_0(H)$ with $x^- \not =x^+$. 
Then the quotient under the $\R$-action denoted by
$$\widehat{\MM}_F(x^-,x^+,H,J_0):=\MM_F(x^-,x^+,H,J_0)/\R$$
is a $\mu(x^-)-\mu(x^+)-1$ dimensional $H^1_{\loc}$-precompact manifold in the sense that for any sequence $(\hat u_n)_{n\in \N}$ in $\widehat{\MM}_F(x^-,x^+,H,J_0)$ 
there exists subsequence $(\hat u_{n_k})_{k\in \N}$ converging to some broken trajectory.
\end{prop}

\subsection{The Floer complex}

In this section we follow \cite{sal} to give a short summary of the construction of the Floer complex. The main ingredients are the compactness and transversality 
results proven in the previous sections and Floer's gluing theorem. \\ 

Let $H \in \Hr$ be given then by the results in section \ref{transF} we can assume that up to a small perturbation which leaves the periodic orbits fixed, $H$ is such that 
$(H,J_0)$ is a regular pair. Denote by 
$C_k(H)$ the Abelian group generated by the elements $x \in \Pc_0(H)$ with Conley-Zehnder index $\mu(x)=k \in \Z$, i.e.,  
$$ C_k(H) = \bigoplus_{\genfrac{}{}{0pt}{}{x \in \Pc_0(H)\,,}{\mu(x)=k }}\Z_2 x \,.$$
We set 
$$\nu(x,y)= \# \big(\MM_F(x,y,H,J_0)/ \R\big)\quad \mod 2  $$
which is  a well defined finite number due to compactness and transversality.
The boundary operator is then defined on the generators by 
$$ \p^F_k(H,J_0) : C_k(H) \lo C_{k-1}(H)\,, \quad \p^F_k(H,J_0) x = \sum_{\genfrac{}{}{0pt}{}{y \in \Pc_0(H)\,,}{\mu(x)-\mu(y)=1 }}\nu(x,y)y \,. $$
The fact that indeed $\p^F_{k-1}(H,J_0)\circ \p^F_{k}(H,J_0) = 0$ is a consequence of Floer's gluing theorem stated in \cite{floer}. 
Therefore the above setup describes a complex
which we call the {\bf Floer - complex} of $(H,J_0)$. The homology of such a complex is called the {\bf Floer homology} of $(H,J_0)$ : 
$$ HF_k:= \frac{\ker \p^F_k(H,J_0)}{\ran \p^F_{k+1}(H,J_0)}$$  
Due to \cite{floer3} this homology is known to be chain isomorphic to the singular homology of the torus. A consequence of this fact 
was Floer's proof of the Arnold conjecture. A more detailed discussion can also be found in \cite{hof}.

\section{The chain isomorphism}\label{4}

The construction of the isomorphism is based on considering the {\bf moduli spaces of hybrid type curves}. This involves stating a {\bf new non-Lagrangian boundary value problem} for a type of 
Cauchy-Riemann equations 
% To handle this situation and develop new proofs of the Fredholm property and the compactness statements are the aims of this chapter. 
and is the essential reason for the non-triviality of Theorem 1. In this section we develop new methods to prove that this new coupling condition leads to a 
well posed Fredholm problem and to the fact that the moduli spaces are precompact. It is crucial for this approach that the torus is compact, possesses trivial tangent 
bundle and an additive structure.

\subsection{Moduli spaces of  hybrid type curves}
$ $ \\
Since we must perturb the vector field $\X = -\Ah$ to achieve transversality 
we formulate our results in a more general setting. That is, we consider a smooth Morse vector field  
$X$ with globally defined flow such that $A_H$ is a Lyapunov function for $X$ and the subbundle $\V = \M\times\big(\R^n \times\Hm^+\big)$ is admissible for $X$ according to
Definition \ref{admissible}.\\

For the whole section  $Z$ will denote the half cylinder $Z:=[0,\infty)\times\mS^1$. 

\begin{defn} Given a Hamiltonian  $H \in \Hr$, and a $C^3$- Morse vector field $X$ on $\M=\HT$ with globally defined flow
such that $A_H$ is a Lyapunov function for $X$ then we have $\rest(X)=\crit(A_H)$ and for contractible 1-periodic orbits $x^{\pm} \in \rest(X)$ we abbreviate by 
$\MM_{\hyb}:=\MM_{\hyb}(x^-,x^+,H,J_0,X)$ the {\bf moduli-spaces of hybrid type curves} which we define as the following set 
\begin{align*}
\MM_{\hyb}:= \Big \{ u \in H^1_{\loc}(Z,\Tn) \mi   \dH (u) = 0 \,,\,  u(0,\cdot)\in& \WW^u_X(x^-)\,, \\
                                                    & u(+\infty,\cdot) = x^+\Big\}\,, 
\end{align*}
where $\WW^u_X(x^-)$ denotes the unstable manifold of $x^-$ with respect to $X$ and $\dH (u) = 0 $ should be understood in the weak sense. \\

To state concretely the boundary conditions we recall  that by Lemma \ref{ev}  the evaluation $u \mapsto u(s,\cdot)$ is smooth in $u$ and therefore the condition 
$u(0,\cdot)\in \WW^u_X(x^-)$ is well-posed. Similarly to Definition \ref{defF} we choose  a smooth reference cylinder
$ c_{x^+} \in C^{\infty}([0,\infty)\times\mS^1,\Tn)$ satisfying 
\begin{equation}\label{cx}
 c_{x^+}(s,\cdot) = \li\{
\begin{array}{ll}
 [0]\,, & \text{for} \quad s \leq 1\\
 x^+\,, & \text{for} \quad s\geq 2\quad
\end{array}\re.
\end{equation}
and recall that by our observation in \eqref{HLLH} each $u \in H^1_{\loc}(Z,\Tn)$ can be written as 
$$u(s,t) = [u_0(s)] + \sum_{k \not = 0 }e^{2\pi k J_0t}u_k(s)\,, $$ 
with Fourier-coefficients that dependent continuously on $s$. So we say that
$\lim \limits_{s\rightarrow +\infty} u(s,\cdot) = x^{+}$  if and only if 
\begin{align*}
 \lim_{s\rightarrow + \infty}(\PP^0)^{\perp}\ev_s(u -  c_{x^+}) &= 0 \quad  \text{in} \quad \Hm\,,\\
 \lim_{s\rightarrow +\infty}\PP^0\ev_s(u -  c_{x^+}) &= 0  \quad  \text{on} \quad  \Tn \,,
\end{align*}
where $\PP^0$ denotes the restriction to the constant loops defined in \eqref{pro}. 
\end{defn}
We note some crucial facts : \\

\begin{enumerate}
\item By the inner regularity property, see Lemma \ref{el1}, the hybrid type curves $u \in \MM_{\hyb}(x^-,x^+,H,J_0,X)$ turn out to be smooth on $(0,\infty)\times\mS^1$. 
Since the evaluation map is smooth in $u$ but not even continuous in $s$, the loop  $u(0,\cdot)$ needs not be continuous, 
and can be any arbitrary element of $\M$.\\

\item Let $T>0$, $u \in\MM_{\hyb}(x^-,x^+,H,J_0,X)$. Then on easily verifies 
\begin{align*}
E_T(u)&:=\int^T_{1/T}\int_0^1 |\p_su(s,t)|^2dtds \\
% &= -\int^T_{1/T}\int_0^1 \li<\p_su(s,t),J\big(\p_t u(s,t) -X_H(u(s,t)\big)\re>dtds\\
% &=  -\int^T_{1/T} \int_0^1\omega_0\Big(\p_su(s,t),\p_tu(s,t)-X_H(u(s,t))\Big)dtds\\
% &=  -\int^T_{1/T} \p_s  A_H(u(s,\cdot))ds\\
&= A_H\big(u(1/T,\cdot)\big)- A_H\big(u(T,\cdot)\big) < \infty \,.
\end{align*}
By continuity of $A_H$ on $\M$  we get $E(u) =  A_H\big(u(0,\cdot)\big)- A_H(x^+) $. If furthermore $A_H$ is a Lyapunov function for $X$ 
then 
$$A_H\big(u(0,\cdot)\big) \leq A_H(x^-)$$
holds.\\

\item By our remarks in  \eqref{en} and \eqref{en2} we have that $ \lim \limits_{s\rightarrow  \infty} u(s,\cdot) = x^{+}$
and $\lim \limits_{s\rightarrow \infty} \p_s(u,\cdot)=0$ are uniform in $t$ and for $T>0$ large enough and every $u\in\MM_{\hyb}(x^-,x^+,H,J_0,X)$, there exist constants $c=c(T,u),\delta=\delta(T,u) >0$ such that 
$$ |\p_su(s,t)| \leq ce^{-\delta|s|}  $$
for all $(s,t)\in [T,\infty)\times \mS^1$.
\end{enumerate}
As in the previous chapter we want to understand the moduli spaces as the affine translation by the reference cylinder $c_{x^+}$ of the zero set of a certain map. 
In this regard we need the following statement. 
\begin{lem}\label{Hilbertmfd}
Let $H$ be a nondegenerate Hamiltonian, and let $X$ be a $C^k$- Morse vector field, $k\geq 1$,  on $\M$ with globally defined flow, such that the Hamiltonian action $A_H$ is a Lyapunov function for $X$ and furthermore
$\V = \M\times\big(\R^n\times\Hm^+\big)$ is admissible for $X$ as in Definition \ref{admissible}. Then the set 
\begin{equation*}
 H^1_{\WW^u_X(x)}(Z,\Rn) := \li\{ w \in H^1(Z,\Rn) \mi [0] + w(0,\cdot)  \in \WW^u_X(x)\re\}\,, 
\end{equation*}
$ [0] \in \Tn $, $x \in \rest(X)$ is a $C^k$- Hilbert submanifold of $\WZR$ with inner product $\big<\cdot,\cdot\big>_{H^1(Z)}$.
\end{lem}
\begin{proof} Since we are in a Hilbert manifold setting and $A_H$ is a Lyapunov function for $X$ by Theorem \ref{stable},  there is a $C^k$-embedding  
$$ F^u : E^u \lo \M\,, \quad \text{with} \quad F^u(E^u) = \WW^u_X(x) \,,$$  
where $E^u$ denotes the unstable eigenspaces of $DX(x)$. We consider the map 
\begin{equation*}
 \lambda : \WZR \times E^u \lo \M\,,\quad  (w,v) \mapsto w(0)-F^u(v)\,,
\end{equation*}
Observe that 
\begin{equation}\label{lambda}
 \lambda^{-1}([0])=\li\{ (w,v) \in  H^1(Z,\Rn) \times E^u \mi w(0)- F^u(v) = [0]\re\} \,. 
\end{equation}
The derivative at $(w,v)$ is given by 
$$  D\lambda(w,v) : \WZR \times E^u\lo \HR\,,\quad  (\xi,\eta) \mapsto \xi(0,\cdot)-DF^u(v)\eta\,,$$
which is obviously onto. Furthermore if $p = F^u(v)$ then 
$$ \ker D\lambda(w,v) \cong   H^1_{T_p\WW^u_X(x)}(Z,\Rn) $$ 
is complemented in $\WZR \times E^u$ if $T_p\WW^u_X(x)$ is complemented in $\Ee$. The latter is certainly true since $\V$ is admissible for $X$ and therefore $T_p\WW^u_X(x)$ is a 
compact perturbation of $\V$ which is complemented by the infinite dimensional space $\R^n \times \Hm^-$. Hence $\lambda^{-1}([0])$ is a $C^k$- Hilbert manifold with 
inner product. Let 
$$\pi : \WZR \times E^u \lo \WZR$$
be the projection onto the first factor, then, by \eqref{lambda}, the restriction of $\pi$  to $\lambda^{-1}([0])$ 
is a $C^k$- embedding into $\WZR$ with 
$$\pi\big(\lambda^{-1}([0])\big) = H^1_{\WW^u_X(x)}(Z,\Rn)\,,$$ which proves the claim. 
\end{proof}
\begin{prop}\label{affine trans} Assume the assumptions of Lemma \ref{Hilbertmfd} be fulfilled. Then the non-linear operator
$$ \Theta_{x^-,c_{x^+}} : H^1_{\WW^u_X(x^-)}(Z,\Rn) \lo L^2(Z,\Rn)\,, \quad w \mapsto \dH(w+c_{x^+}) \,. $$
is well defined, of class $C^k$, and  
\begin{equation*}
 \MM_{\hyb}(x^-,x^+,H,J_0,X) = c_{x^+} +\Theta_{x^-,c_{x^+}}^{-1}(0) \subset \WZTl   \,. 
\end{equation*}
\end{prop}

\begin{proof} The proof is completely analogous to the proof of Proposition \ref{Pfi}.
\end{proof}

\subsection{The Fredholm problem and Transversality}

In this section we prove that if the Morse vector field $X$ on $\M$ is of the form $\X_K = -\Ah +K$ with $K\in \Kc_{\theta,r}$, $r>0$, as in section \ref{2.3} and is sufficiently close 
to $-\Ah$ then for all 1-periodic solutions $x^-,x^+ \in \Pc_0(H)$ the associated maps $\Txx$ are Fredholm. Furthermore, we extend our transversality result from section \ref{2.3} 
to show that there is a residual set of compact maps $\widehat \Kc_{\reg} \subset\Kc_{\theta,r}$ such that for $K \in \widehat \Kc_{\reg}$,  $0$ becomes a regular value for 
$\Theta_{x^-,c_{x^+}}$ and $\X_K$ satisfies the Morse-Smale condition up to order 2. So both, the intersections of 
stable and unstable manifolds with respect to $\X_K$ and the moduli spaces of hybrid type curves are manifolds. If, in addition, $(H,J_0)$ is a regular pair then 
also the moduli spaces of Floer cylinders are  manifolds, which shows that for generic data all three problems can be treated simultaneously.\\

% Recall that 
% $\Kc(\M,\Ee) \subset C^{3}_b(\M,\Ee)$ denotes the closed subspace of all bounded and compact  $C^3$- vector fields on $\M$. Furthermore for  $\theta \in C^1(\M,\R^+)$ with 
% \begin{enumerate}
%  \item $\theta(x)=0$ for all $x \in \crit(A_H)$ 
%  \item $\theta(p) > 0 $ for all $p\in \M \setminus \crit(A_H)$.
%  \item $\theta(p) \leq \frac{1}{2}\nl \Ah(p)\nr_{\HS}$ for all $p \in \M$. 
% \end{enumerate}
% we consider the Banach space 
% \begin{equation*}
% \Kc_{\theta}: =\li\{ K \in \Kc(\M,\Ee) \mi  \exists \, c > 0 \,\text{such that}\, \nl K(p)\nr_{H^{1/2}(\mS^1)} \leq c \theta(p)\,, \forall p \in \M \re\}
% \end{equation*}
% with norm 
% $$ \nl K \nr_{\theta}: = \sup_{p\in\M\setminus \rest(X)}\frac{\nl K(p)\nr_{\HS}}{\theta(p)} + \nl \nabla K \nr_{C^2(\M,\Ee)}\,.$$
% Let $\Kc_{\theta,r}$, $r>0$ be the open ball of radius $r$ in $\Kc_{\theta}$ then 
Let $r=1$  and $\X_K:=-\Ah+K$, $K \in \Kc_{\theta,r}$ then in   Lemma \ref{pert} we have shown that 
 \begin{enumerate}
 \item $\rest(\X_K) = \crit(A_H)$.
 \item $D\X_K(x) = -D^2A_H(x)$ for all $x \in \crit(A_H)$.
 \item $A_H$ is a  Lyapunov function for $\X_K$.
\end{enumerate}
So the dynamics of $\X_K$ do not qualitatively differ from those of $-\Ah$.
\begin{thm}\label{fred2} Let $H \in \Hr$, $r>0$ and $K \in \Kc_{\theta,r}$. 
 Then, if $r$ is chosen  sufficient small, for any $x^{\pm} \in \Pc_0(H)$ and any $w \in H^1_{\WW^u_{\X_K}(x^-)}(Z,\Rn)$,  
the linearized operator 
\begin{equation}\label{Dhyb}
 D\Txx(w) : H^1_{T_p\WW^u_{\X_K}(x^-)}(Z,\Rn)\lo  L^2(Z,\Rn)  
\end{equation}
with 
$$  D\Txx(w)\eta = \p_s\eta + J_0\p_t\eta +  \nabla^2H(\cdot,w+c_{x^+})\eta $$
and $p:= [0] + w(0,\cdot)$, is a Fredholm operator of index 
$$ \ind  D\Txx(w)  = m(x^-,\V)-\mu(x^+)\,,$$
where  $m(x^-,\V)$ denotes the relative Morse index of $x^-$ relative to $\V=\M\times\big(\R^n\times\Hm^+\big)$, and $\mu(x^+)$ the Conley-Zehnder index of $x^+$.
\end{thm}
The proof splits into several results. First we prove the Fredholm-property, then we compute the index in a special situation and in the end use a homotopy 
argument to obtain the desired statement. \\

As in section \ref{secfred} we consider a continuous symmetric matrix valued function  $S \in C^0(Z,\Rn\times\Rn)$, and assume that 
\begin{equation}\label{Scon}
\lim \limits_{s\rightarrow+\infty }S(s,t) = S^+(t) \in C^0(\mS^1,\Rn\times\Rn) 
\end{equation}
uniformly in $t$. We define a symplectic path $\Psi^+ \in C^1([0,1],\Sp(2n,\R))$ by 
\begin{equation}\label{cz}
 \p_t \Psi^+(t) = J_0S^+(t)\Psi^+(t)\,, \quad \Psi^+(0) = \Id\,. 
\end{equation}
 and recall that the  Conley-Zehnder index $\mu(\Psi^+)$ is well defined if 
\begin{equation}\label{detpsi}
\det\big(I-\Psi^+(1)\big) \not = 0\,. 
\end{equation}
\begin{lem} \label{Splus} Let $S \in C^0(Z,\Rn\times\Rn)$ such that $\det\big(I-\Psi^+(1)\big) \not = 0$ holds for $\Psi^+$ as in \eqref{cz}. Then the operator 
\begin{equation*}
\overline{D}_{J_0} + S^+  = \p_s +J_0\p_t + S^+ : \WRR \lo \LRR
\end{equation*}
and its formal adjoint  $D_{J_0}+S^+= -\p_s +J_0\p_t + S^+ $ are isomorphisms. In particular there are constants  
$c_1,c_2>0$ such that 
\begin{align*}
\nl \xi \nr_{H^1(\R\times\mS^1)} &\leq c_1 \nl (\overline{D}_{J_0} + S^+)\xi \nr_{L^{2}(\R\times\mS^1)}\,, \\
\nl \xi \nr_{H^1(\R\times\mS^1)} &\leq c_2 \nl (D_{J_0} + S^+)  \xi \nr_{L^{2}(\R\times\mS^1)}\,.  
\end{align*}
for all $\xi \in \WRR$. 
\end{lem}
Detailed proofs can be found in \cite{schwarz} or \cite{sal}.
The most difficult part in the proof of the Fredholm property in Theorem \ref{fred2} is to show  that the cokernel is finite dimensional. Therefore we need a certain 
elliptic regularity result, which replaces the already known result for Floer half cylinders with Lagrangian boundary condition, see \cite{duff} or \cite{schwarz}. 
The result  is established in step 2 of the proof of the following statement, using again the notion of compact perturbations introduced in section \ref{fredpairs}. 
\begin{prop}\label{ann}  
Let $W\subset \Ee $ be a compact perturbation of $\R^n\times\Hm^+$ and $S \in C^0(Z,\Rn\times\Rn)$ be symmetric such that \eqref{Scon} and \eqref{detpsi} hold. Then 
$$ 
\overline{D}_{W}:=\overline{D}_{J_0} + S : H^1_{W}(Z,\Rn) \lo  L^2(Z,\Rn)\,,
$$
with $\overline{D}_{J_0}  = \p_s+ J_0\p_t$, is a Fredholm operator and its cokernel is given by
$$\coker (\overline{D}_{W}) = \li\{ \rho \in H^1(Z,\Rn) \mi (D_{J_0} +S)  \rho = 0 \,, \quad \rho(0,\cdot)\in 
W^{\perp}\subset \Ee\re\}\,.$$
\end{prop}
\begin{proof}{\bf Step 1:} 
By the formula in  \eqref{delbarform} we have that
\begin{align*}
 \nl \p_s \xi\nr_{L^{2}(Z)}^2 + \nl \p_t \xi  \nr_{L^{2}(Z)}^2 & = \nl \overline{D}_{J_0} \xi  \nr_{L^{2}(Z)}^2 \\
                                                               & - \nl \PP^+\xi(0,\cdot)  \nr_{\HS}^2 + \nl \PP^- \xi(0,\cdot)  \nr_{\HS}^2  \,. 
\end{align*}
We recall that $\nl S \nr_{C^0(Z)}>0$. Therefore we can choose a constant $c_1\geq 1+1/\nl S \nr_{C^0(Z)}$ and obtain  
\begin{align}
\nl  \xi\nr_{H^1(Z)}  \leq c_1\Big(\nl (\overline{D}_{J_0} +S) \xi  \nr_{L^{2}(Z)}  & + \nl S\nr_{C^0(Z)}\nr \xi\nl_{L^2(Z)}\nonumber \\
                         & + \nl \PP^- \xi(0,\cdot)  \nr_{\HS}^2\Big)\,. \label{f1}
\end{align}
Since by Lemma \ref{Splus} the operator $\overline{D}_{J_0} +S^+ :H^1(Z,\Rn) \lo L^2(Z,\Rn)\big)$ is an isomorphism there is $R>0$ and $c_2>0$ such that 
\begin{equation}\label{f2}
 \nl \xi \nr_{H^1([R,+\infty)\times \mS^1  )} \leq c_2 \nl (\overline{D}_{J_0} +S) \xi \nr_{L^{2}([R,+\infty)\times \mS^1  )}\,, 
\end{equation}
for all $\xi$ with $\supp \xi \subset [R,\infty) \times \mS^1$.
We choose a smooth cut-off function $\beta$ with 
$$\beta(s) = \li\{\begin{array}{rl}
1 \,, &  s \leq R  \\
0 \,, &  s \geq R +1 
 \end{array}\re.
$$
use \eqref{f1} for $\beta \xi $ and \eqref{f2} for $(1-\beta)\xi$ and compute 
\begin{align}
 \nl \xi \nr_{H^1(Z)} & \leq \nl \beta \xi \nr_{H^1(Z)} + \nl (1-\beta) \xi \nr_{H^1(Z)}  \nonumber \\ 
                          & \leq c_1\Big(\nl (\overline{D}_{J_0} +S)\beta \xi  \nr_{L^{2}(Z)}  + \nl S\nr_{C^0(Z)}\nl \beta \xi\nr_{L^2(Z)} \nonumber \\
                          &  \qquad \qquad\qquad\qquad\qquad\qquad + \nl \PP^- \xi(0,\cdot)  \nr_{\HS}^2\Big)\nonumber  \\
                          & + c_2\nl(\overline{D}_{J_0} +S)(1-\beta)\xi \nr_{L^{2}(Z)} \nonumber \\ 
                          & \leq c_3\Big(\nl (\overline{D}_{J_0} +S) \xi \nr_{L^{2}(Z)} + \nl  \xi\nr_{L^2([0,R+1]\times \mS^1)}\nonumber \\ 
                          &  \qquad \qquad\qquad\qquad\qquad\quad\,\, + \nl \PP^- \xi(0,\cdot)  \nr_{\HS}^2\Big) \label{haha} \,.
\end{align} 
Since $W$ is assumed to be a compact perturbation of $\R^n\times\Hm^+$, the operator $\PP^-:  W \lo \Hm^-$ is compact, furthermore the restriction
$H^1(Z) \lo L^2([0,R+1]\times \mS^1)$ is also compact. Hence \eqref{haha} implies that $\overline{D}_{J_0} +S$ is a semi-Fredholm operator, i.e., it is a bounded 
operator with finite dimensional kernel and closed range, see \cite{schwarz-2}\,.\\

{\bf Step 2:} We have to show that the cokernel is finite dimensional. Therefore let 
$\eta \in L^2(Z,\Rn)$ be an annihilator of the range of 
$$ 
\overline{D}_{J_0} + S : H^1_{W}(Z,\Rn) \lo  L^2(Z,\Rn)\,,
$$ 
i.e.,
$\li <\overline{D}_{J_0} + S \xi,\eta\re>_{L^2(Z)} = 0 $ for all 
$$\xi \in H^1_{W}(Z,\Rn) = \li\{\rho \in H^1(Z,\Rn) \mi \rho(0,\cdot) \in W\re\}\,.$$
Then $\eta$ is a weak solution of $D_{J_0}\eta = -S\eta =:-f$ on $Z$, i.e., for all $\phi \in C^{\infty}_0(Z,\Rn)$ with $\phi(0,\cdot) \in W $
\begin{equation}\label{weak}
   \big <\overline{D}_{J_0}\phi,\eta\big>_{L^2(Z)} =-\big <\phi,f\big>_{L^2(\Omega)}\,, 
\end{equation} 
with $\nl f\nr_{L^2(Z)} \leq \nl S\nr_{C^0(Z)}\nl\eta\nr_{L^2(Z)}\,. $
We claim that  $\eta$ is of class $H^1$ and more precisely
$$\eta \in \li\{  \rho \in H^1(Z,\Rn) \mi (D_{J_0} +S)   \rho = 0 \,, \quad  \rho(0,\cdot)\in 
W^{\perp}\subset \HR\re\}\,.$$
We write $\eta(s,t) = \sum \limits_{k\in \Z} e^{2\pi k J_0t}\eta_k(s)$ and $f(s,t) = \sum \limits_{k\in \Z} e^{2\pi k J_0t}f_k(s) $ as Fourier series. 
Then the right candidates for the weak derivatives  are 
\begin{align*}
 \p_t \eta(s,t) & = 2\pi\sum_{k\in \Z} J_0ke^{2\pi k J_0t}\eta_k(s)\,,  \\
\p_s \eta(s,t)  & = \sum_{k\in \Z}e^{2\pi k J_0t}f_k(s) - 2\pi\sum_{k\in \Z} ke^{2\pi k J_0t}\eta_k(s) \,. 
\end{align*}
We have to show that they are bounded in the $H^1$-norm. 
Since  $W$ is assumed to be a compact perturbation of $\R^n\times\Hm^+$ the spaces $\PP^-W$ and $\PP^+W^{\perp}$ are finite dimensional. Let $N\in \N$ and denote by 
$\PP^+_N$ the projection onto the positive frequencies bounded by $N$, i.e.,   
$$\PP^+_Nv =  \sum_{0<k\leq N} e^{2\pi k Jt}v_k\,, \quad \text{for all} \quad v\in \Ee\,. $$ 
Then we can approximate  $\PP^+$ by $\PP^+_N$ on $W^{\perp}$, i.e., for all $\ee >0$ there exists 
$N(\ee)\in \N$ such that 
\begin{equation}\label{proapp}
\nl (\PP^+ - \PP^+_N)q \nr_{\HS} \leq \ee\,, \quad \text{for all} \quad q \in W^{\perp}\,.   
\end{equation}
Assume that $\eta$ is smooth, then \eqref{weak} implies that $\eta(0,\cdot) \in W^{\perp}$, and since $D_{J_0}\eta =-f$ 
by the formulas for the weak derivatives, we observe that 
\begin{equation*}
 \eta_k(s) = e^{-2\pi ks}\li[\int_0^s e^{2\pi k\tau}f_k(\tau)d\tau + \eta_k(0)\re]\,.
\end{equation*}
Hence $\eta_k(0) =  e^{2\pi ks}  \eta_k(s) - \int \limits_0^s e^{2\pi k\tau}f_k(\tau)d\tau $ and for $T>0$ 
\begin{equation*}
 | \eta_k(0) | \leq  e^{2\pi kT}|\eta_k(s) | + \sqrt{\frac{e^{4\pi kT}-1}{4\pi k}} \nl f_k \nr_{L^2([0,T])} \,, \quad \text{for} \quad 0\leq s \leq T\,,
\end{equation*}
where we used H\"older's inequality for the second term. Hence there is $d_1=d_1(N,T)>0$ such that 
\begin{align*}
\nl \PP^+_N\eta(0,\cdot)\nr^2_{\HS} & = \sum_{0<k\leq N}2\pi k |  \eta_k(0)  |^2 \\
                                    & \leq d_1 \Big( \sum_{0<k\leq N}|\eta_k(s) |^2 + \nl f \nr^2_{L^2(Z_T)}\Big) \,, \quad \text{for} \quad 0\leq s \leq T\,,
\end{align*}
where $ Z_T=[0,T]\times \mS^1$. Integrating over $[0,T]$ yields 
\begin{equation*}
T\nl \PP^+_N\eta(0,\cdot)\nr^2_{\HS}  \leq d_1 \Big( \nl P^+_N\eta\nr_{L^2(Z_T)}^2 + T\nl f \nr^2_{L^2(Z_T)}\Big)\,.
\end{equation*} 
Choose $\ee = 1/2$ in \eqref{proapp} and  $N=N(\ee)$ then clearly there is a constant $d_2=d_2(N,T)$ such that 
\begin{align*}
\nl \PP^+ \eta(0,\cdot)) \nr_{\HS}^2  & \leq 2 \nl \PP^+_N \eta(0,\cdot)) \nr_{\HS}^2 \\ 
                                      & \leq d_2\Big(\nl \eta \nr_{L^2(Z_T)}^2 + \nl f \nr_{L^2(Z_T)}^2\Big)\,. 
\end{align*}
Now by  \eqref{delbarformm} and \eqref{weak} we can find a further constant $d_3 = d_3(N,T,S) $ such that  
\begin{align}
\nl \p_s \eta\nr_{L^{2}(Z)}^2 + \nl \p_t \eta \nr_{L^{2}(Z)}^2  & =   \nl D_{J_0} \eta \nr_{L^{2}(Z)}^2 \nonumber\\ 
                                                                & +   \nl \PP^+\eta \nr_{\HS}^2 - \nl \PP^- \eta \nr_{\HS}^2 \nonumber \\
                                                                &\leq \nl f \nr_{L^{2}(Z)}^2 \nonumber \\
                                                                & +   d_2\Big(\nl \eta \nr_{L^2(Z_T)}^2 + \nl f \nr_{L^2(Z_T)}^2\Big)  \nonumber \\
                                                                &\leq d_3\nl \eta \nr_{L^2(Z)}^2 \,. \nonumber 
\end{align}
By density the weak derivatives are bounded which shows that indeed $\eta \in H^1(Z,\Rn)$ and we obtain the formula of $\ker(D_{J_0}+S)$ by \eqref{weak}.
Finally we use the cut-off function $\beta$ and \eqref{delbarformm} as in Step 1 to establish the estimate 
$$\nl \eta\nr_{H^1(Z)} \leq d_4\Big( \nl (D_{J_0}+S)\eta \nr_{L^{2}(Z)} + \nl \eta\nr_{L^2([0,R+1]\times\mS^1)} + \nl \PP^+\eta \nr_{\HS}\Big)\,, $$
with $d_4>0$, which shows that the operator 
$$D_{J_0}+S  : H^1_{W^{\perp}}(Z,\Rn) \lo  L^2(Z,\Rn)$$
is a semi-Fredholm operator whose finite dimensional kernel coincides with the 
cokernel of $\overline{D}_{J_0}+S$, which is therefore indeed a Fredholm operator. 
\end{proof}
It remains to calculate the index. To do this we reformulate the Fredholm-problem studied so far in an equivalent way.
Recall that $W^u_A,W^s_A \subset \Ee$ denote the linear unstable and stable  space of the Cauchy-problem 
$$ X_A^{\prime}(s)  =   AX_A(s) \,, \quad X_A(0)  =  I \quad \text{with}\quad A \in C^0([-\infty,+\infty],\Ee)\,,$$ 
where $A(\pm \infty)$ is assumed to be hyperbolic. We denote by $P_{W^{u}_A}, P_{W^{s}_A}$ the associated projections and by $ C^k_0((-\infty,0],\Ee)$, $k \in \N$ 
the space of all $C^k$-curves $\eta$ with $\lim \limits_{s\rightarrow-\infty}\eta^{(l)}(s)=0$ for all $0\leq l\leq k$.
\begin{lem}\label{DDhyb} Let $A :[-\infty,0]\lo \Lc(\Ee)$ be a continuous path of bounded linear operators such that $A(-\infty)$ is hyperbolic.  Assume that $W^u_A$ is 
a compact perturbation of $\R^n\times\Hm^+$ and let $S \in C^0(Z,\Rn\times\Rn)$ be symmetric such that \eqref{Scon} and \eqref{detpsi} hold. 
Then the operator
$$ \overline{D}_{\hyb} : C^1_0((-\infty,0],\Ee)\times \WZR \lo  C^0_0((-\infty,0],\Ee)\times \LZR \times \Ee $$
defined by 
$$  \overline{D}_{\hyb}(\eta,\xi) = \Big(\eta^{\prime}-A\eta,(\overline{D}_{J_0}+S)\xi, \eta(0)-\xi(0,\cdot)\Big)\,,$$
with $\overline{D}_{J_0}  = \p_s+ J_0\p_t$, is Fredholm and its index coincides with the index of the operator 
$$\overline{D}_W:= \overline{D}_{J_0} + S : H^1_{W}(Z,\Rn) \lo  L^2(Z,\Rn)$$ 
with $W = W^u_A$.
\end{lem}
\begin{proof} We show that the respective kernels and cokernels are isomorphic. If $(\eta,\xi)\in \ker \overline{D}_{\hyb}$ then  
$\eta^{\prime}-A\eta =0$ and $\lim \limits_{s\rightarrow-\infty} \eta(s) = 0 $ imply that  
$\eta(0)\in W^u_A$. Furthermore, $\eta$ is uniquely determined by $\eta(0)=\xi(0,\cdot)$. Hence 
$$\ker(\overline{D}_{\hyb}) \cong \li\{\xi \in \WZR \mi (\overline{D}_{J_0} + S)\xi =0\,,\quad \xi(0,\cdot) \in W\re\} = 
\ker (\overline{D}_{W})\,.$$
By Proposition \ref{FA1} the operator $F_A^- : C^1_0\big((-\infty,0],\Ee\big) \lo C^0_0((-\infty,0],\Ee)$, $\eta \mapsto \eta^{\prime}-A\eta$ is a left inverse with 
right inverse $R_A^-$ such that 
\begin{equation*} \label{scheiss}
 W^u_A + \li\{R^-_A(\kappa)(0) \mi \kappa \in C^0_0((-\infty,0],\Ee)\,, \kappa(0)=0 \re\} = \Ee \,.
\end{equation*} Moreover
$\ker(F_A^-) = \li\{X_A(s)v \mi v \in W\re\}$, $s\in\R$ which readily implies the inclusions
\begin{equation}\label{subco}
\li\{0\re\}\times \coker(\overline{D}_W) \times \{0\} \subset \coker(\overline{D}_{\hyb})
\end{equation}
and
$$ \li\{ \kappa \in C^0_0((-\infty,0],\Ee)\mi \kappa(0)=0 \re\}\times \ran(\overline{D}_W) \times \Ee \subset \ran(\overline{D}_{\hyb})\,. $$ 
Therefore, it remains to consider the case where $(\kappa,\rho,p)$ are in the complement of $\li\{0\re\}\times \coker(\overline{D}_W) \times \{0\}$ with 
$\kappa(0)\not = 0$ and $\supp \kappa \subset (-\infty,0]$ is compact. Then we can find $\xi \in H^1_{W}(Z,\Rn)$ with $\overline{D}_{W}\xi = \rho$ and set  
\begin{align*}
\eta(s) & = X_A(s)P_W\Big(\int_0^s X_A(\tau)^{-1}\kappa(\tau) d\tau + p+\xi(0,\cdot)\Big)\\
        & - X_{-A}(s)P_{W^\perp}\Big(\int_s^0 X_{-A}(\tau)^{-1}\kappa(\tau) d\tau - p\Big)\,.
\end{align*}
Since $\kappa$ possesses compact support $\eta$ belongs to  $ C^1_0((-\infty,0],\Ee)$ and therefore $\overline{D}_{\hyb}(\eta,\xi) = (\kappa,\rho,p)$. Thus  
\eqref{subco} is an equality and we obtain $\coker(\overline{D}_{W}) \cong \coker(\overline{D}_{\hyb})$ 
% Finally the operator $\overline{D}_{W}$ is $\ee$-close to the operator $\overline{D}_{J_0} + S$ from Proposition \ref{ann} and therefore Fredholm for any sufficiant small $\ee$ 
concluding the proof.
\end{proof}

Before we compute the Fredholm index in a special situation, recall that  
$j^* : \LLR\lo \HR$ denotes the adjoint of the embedding $j :\HR \lo \LLR$ from \eqref{j}. Furthermore, note that the Conley-Zehnder index with respect to 
$$ J_0 =\li ( \begin{array}{rc}
 0 & 1 \\
-1 & 0
\end{array}\re)\quad \text{from}\quad \eqref{JJO}
$$
of the symplectic path $ \Psi(t)=e^{J_0S_{\lambda}t}\in C^0\big([0,1],\Sp(2n,\R)\big)$ with $S_{\lambda}=\lambda I$, $\lambda \in \R$ is given by 
\begin{equation}\label{czdiag}
 \mu(\Psi)= -2n\li\lfloor\frac{\lambda}{2\pi}\re\rfloor- n \,,
\end{equation}
where $\li\lfloor x \re\rfloor := \max \li\{k\in \Z \mi k\leq x\re\}$, $x\in\R$ denotes the Gau\ss-bracket. Explicit calculations can be found in \cite{salzeh}.

\begin{lem}\label{indform} Let $a,b \in \R$, $a,b \not \in 2\pi \Z$  and $S_a:= aI\,, S_b:= bI \in \mathrm{Gl}(2n,\R)$. Denote by  $W=W^u_{S_a}\subset \HR$  the linear unstable space of the Cauchy-problem 
$$ X_A^{\prime}(s)  =   AX_A(s) \,, \quad X_A(0)  =  I \quad \text{with}\quad A=\PP^+-\PP^- -j^*S_a$$ 
and by $\Psi_b \in C^1([0,1],\Sp(2n))$ the symplectic path defined in \eqref{cz} with $S=S_b$. Then the linear operator
$$\overline{D}_{a,b} : C^1\big((-\infty,0],\Ee\big)\times \WZR \lo  C^0((-\infty,0],\Ee)\times \LZR \times \Ee $$
defined by 
$$ \overline{D}_{a,b}(\eta,\xi)= \Big(\eta^{\prime}-(\PP^+-\PP^--j^*S_a)\eta,(\overline{D}_{J_0}+S_b)\xi, \eta(0)-\xi(0,\cdot)\Big)$$
is Fredholm of finite index 
\begin{equation*}
 \dim (W^u_{S_a},\R^n\times\Hm^+) - \mu(\Psi_b) \,,
\end{equation*}
where  $\dim (W^u_{S_a},\R^n \times \Hm^+) = \dim\big(W^u_{S_a} \cap(\R^n\times \Hm^+)^{\perp}\big)-\dim\big((W^u_{S_a} )^{\perp}\cap( \R^n\times\Hm^+)\big)$ denotes the 
relative dimension of $W^u_{S_a}$ with respect to $\R^n\times\Hm^+$ introduced in \eqref{reldim}. In particular, if  $a=b$, then $ \overline{D}_{a,b}$ is an isomorphism.
\end{lem}
\begin{proof}In this particular situation we obviously have that 
$$ X_A(s) = e^{(\PP^+ -\PP^--j^*S_a)s}\,, \quad \forall \,s\in\R\,.$$ 
Recall that, by Proposition \ref{jj}, for all $q \in L^2(\mS^1,\Rn)$ 
$$j^*(q)(t)= q_0 + \sum \limits_{k\not = 0}e^{2\pi kJ_0t}\frac{1}{2\pi|k|}q_k\,.$$ 
Since  the matrices $S_a,S_b$ possess no null-direction  
$\lim \limits_{s\rightarrow-\infty} X_A(s)p = 0$ holds if and only if 
\begin{equation*}
p \in \li\{p \in \HR  \mi p(t) = \sum_{k\geq K}e^{2\pi kJ_0t}p_k\,, \quad 
K= \li\lfloor\frac{a}{2\pi}\re\rfloor +1\re\}=W^u_{S_a} \,.
\end{equation*}
So $W^u_{S_a}$ is a compact perturbation of 
$\R^n\times\Hm^+$ with
\begin{equation}\label{rdim}
\dim (W^u_{S_a}, \R^n\times\Hm^+) = -2n\li\lfloor\frac{a}{2\pi}\re\rfloor-n \,.
\end{equation}
Since $\Psi_a,\Psi_b$ satisfy \eqref{detpsi} if and only if $a,b \not \in 2\pi \Z$, we can apply Proposition \ref{ann} and Lemma \ref{DDhyb}
to obtain that  $ \overline{D}_{a,b}$ is Fredholm of finite index. Let $0\not=(\eta,\xi) \in \ker( \overline{D}_{a,b})$, then by Lemma \ref{DDhyb} 
we have that the kernel is isomorphic to 
$$\li\{ \xi \in H^1(Z,\Rn) \mi (\overline{D}_{J_0}+S_b) \xi = 0\,, \quad \xi(0,\cdot)\in W^u_{S_a}\re\}\,.$$
Therefore we can use the Fourier representation 
$$\xi(s,t) = \xi_0(s) + \sum \limits_{k\in \Z}e^{2\pi kJ_0t}\xi_k(s)$$ 
to observe that
$ \dot\xi_k(s)-(2\pi k-b)\xi_k(s)=0\,.$
Hence 
$$\xi_k(s) = e^{(2\pi k-b)s}\xi_k(0) \,.$$
Now, since we must have $\xi_k(s) \lo 0$ for $s\rlo+\infty$ and $\xi(0,\cdot)\in W^u_{S_a}$, we obtain 
\begin{equation}\label{K1}
\li\lfloor\frac{a}{2\pi}\re\rfloor+1 \leq k\leq \li\lfloor\frac{b}{2\pi}\re\rfloor\,.
\end{equation}
On the other hand, by  Proposition \ref{ann} and Lemma \ref{DDhyb},
the cokernel of $ \overline{D}_{a,b}$ is  canonically isomorphic to 
$$\li\{ \rho \in H^1(Z,\Rn) \mi (D_{J_0}+S_b) \rho = 0\,, \quad \rho(0,\cdot)\in (W^u_{S_a})^{\perp}\re\}\,.$$
If $0\not = \rho $ the Fourier-coefficients satisfy $\dot\rho_k(s)+(2\pi k-b)\rho_k(s)=0$. Hence 
$$\rho_k(s) = e^{-(2\pi k-b)s}\rho_k(0)$$
and we have to require that 
\begin{equation}\label{K2}
\li\lfloor\frac{b}{2\pi}\re\rfloor+1 \leq k\leq \li\lfloor\frac{a}{2\pi}\re\rfloor
\end{equation}
to assure that  $\rho_k(s) \lo 0$ for $s\rlo+\infty$.
So if $a=b$  none of the conditions  \eqref{K1} or \eqref{K2} can be fulfilled, which implies that  kernel and cokernel are trivial and therefore 
$\overline{D}_{a,b}$  is an isomorphism in this case. If $a,b$ are such that  either \eqref{K1} or \eqref{K2} applies, then we obtain, by \eqref{czdiag} and \eqref{rdim}
$$\ind(\overline{D}_{a,b}) = -2n\li\lfloor\frac{a}{2\pi}\re\rfloor +2n\li\lfloor\frac{b}{2\pi}\re\rfloor =  \dim (W^u_{S_a},\R^n \times \Hm^+ ) - \mu(\Psi_b) \,.$$
\end{proof}
{\bf Proof of Theorem \ref{fred2}}:
\begin{proof}
We approximate $w \in H^1_{\WW^u_{\X_K}(x^-)}(Z,\Rn)$ by a curve $\tilde w$ that is smooth in the interior $(0,\infty)\times \mS^1$, 
and satisfies $\tilde w(0,\cdot) =w(0,\cdot)$ and consider the operator 
$$\widetilde{D\Theta}_{x^-,c_{x^+}}(\tilde w)\xi = \p_s\xi + J_0\p_t\xi + \nabla^2 H(\cdot,\tilde w +c_{x^+})\xi\,.$$
In this situation \eqref{Scon} and \eqref{detpsi} are satisfied and, moreover, by Lemma \ref{pert}, we have that $W:= T_p\WW^u_{\X_K}(x^-)$ with $p = w(0,\cdot)$ is a 
compact perturbation of $\R^n \times \Hm^+$. Thus all assumptions of Proposition \ref{ann} hold and therefore 
$$ \overline{D}_{J_0} +S : H^1_W(Z,\Rn) \lo L^2(Z,\Rn)\,, \quad S = \nabla^2 H(\cdot,\tilde w +c_{x^+})$$ is a Fredholm operator, which implies 
that $\widetilde{D\Theta}_{x^-,c_{x^+}}(\tilde w)$
is also Fredholm and of equal index. Moreover, $\Txx$ is smooth, and therefore $D\Txx(w)$ is also Fredholm possessing the same index as $\overline{D}_{J_0} +S$. 
By Lemma \ref{DDhyb} the Fredholm operator $\overline{D}_{\hyb}$ with $A = D\X_K(v)$ for the trajectory $v$ of $\ffiXK$ with $v(0) = w(0,\cdot) \in \WW^u_{\X_K}(x^-)$ 
possesses the same index as $D\Txx(w)$. Hence, we can apply Lemma \ref{indform} to compute the index.\\

Observe that if we choose $A^{\prime}=D\X(v)$ instead of $A = D\X_K(v)$, then this changes the operator  $\overline{D}_{\hyb}$ only by a small perturbation. 
Hence the index does not change as long as we have chosen $r>0$ small enough. 
Now set $S_1 = \nabla^2H(\cdot,v)$, $S_2 = \nabla^2H(\cdot,\tilde w +c_{x^+})$, then surely we can find a homotopy sending $S_1$ to $S_1^-=  \nabla^2H(\cdot,x^-)$ and 
$S_2$ to $S_2^+ = \nabla^2H(\cdot,x^+)$. Since $x^-,x^+$ are nondegenerate, the Conley-Zehnder indices of the symplectic paths induced by $S_1^-$ and $S_2^+$ are well 
defined. Moreover, the indices of two symplectic paths coincide if and only if they are homotopic, see \cite{salrob}. Thus we can find numbers $a,b\in \R$
and a homotopy from $S_1^-$ to $S_a = aI$ and $S_2^+$ to $S_b = bI$ without changing the indices,
which, by Lemma \ref{indform} and the fact that the Fredholm index is a locally constant function proves that the operator 
$$ \overline{D}_{\hyb}(\eta,\xi)	= \Big(\eta^{\prime}-(\PP^+-\PP^-+S_1)\eta,(\overline{D}_{J_0}+S_2)\xi, \eta(0)-\xi(0,\cdot)\Big)\,,$$
possesses the claimed index formula and thus so does  $D\Theta_{x^-,c_{x^+}}(w)$.
\end{proof}
In particular the above argument shows that $m(x,\R^n \times \Hm^+) = \mu(x)$ for non degenerate singular points which was already  known in \cite{Al8}. 
Nevertheless, we decided to follow \cite{AlSchw} and calculate the index directly in this situation.\\

In order to prove the transversality result for the hybrid type curves, we apply the Carleman similarity principle, which is detaily discussed in \cite{hof}.  
% following well-known statement. 
% \begin{thm}\label{carl}{\bf (\cite{hof}) (Carleman similarity principle)} Denote by 
% $$D = \li\{ z \in \C^n \mi |z| \leq 1 \re \}$$
% the unit disc and let 
% $S = S(z) \in L^{\infty}\big(D,\Lc_{\R}(\C^n)\big)$ be a $z=s+it\in \C$ -dependent $\R$-linear map and $J \in C^{\infty}\big(\C^n,\Lc_{\R}(\C^n)\big)$ 
% be a $z$-dependent complex structure on $\C^n$. Assume that $ \varphi : \mathring{D} \lo \C^n$ is a solution of  
% $$ \p_s\vartheta +J(z)\p_t\vartheta + A(z) \vartheta= 0 \,, \quad \text{and} \quad  \vartheta(0) = 0 \,. $$
% Then there is $0< \delta <1$, a holomorphic map 
% $$\psi : D_{\ee}= \li\{ z \in \C^n \mi |z| \leq \delta \re \} \lo \C^n$$
% and a continuous map 
% $\Gamma : D_{\delta} \lo Gl_{\R}(\C^n)$\,, 
% $$ \Gamma \in  \bigcap_{2<p<\infty} W^{1,p}\Big(D_{\delta},Gl_{\R}(\C^n)\Big) \,, $$ 
% satisfying on $D_{\delta}$, 
% $$J(z)\Gamma(z) = \Gamma(z)i\,, \quad \text{and} \quad \varphi(z) = \Gamma(z)\psi(z) \,.$$
% \end{thm}
% If in particular, $\varphi = 0$ on a non-discrete subset of $\mathring{D}_{\delta}$, the same will hold for $\psi$, which, since it is holomorphic will therefore 
% vanish identically, which implies that $\varphi \equiv 0$ on $D_{\delta}$.

\begin{thm}\label{res} Let $H \in \Hr$, $r>0$. If $r$ is chosen small enough there exists a residual set $\widehat{\Kc_{\reg}}\subset \Kc_{\theta,r}$ of compact vector 
fields $K$ such that for all $x^-,x^+ \in \Pc_0(H)$ with $x^-\not = x^+$ and $ m(x^-,\V) - \mu(x^+) \leq 2$, $\V = \R^n\times\Hm^+$, and any 
$u \in \MM_{\hyb}(x^-,x^+,H,J_0,\X_K)$ the operator 
\begin{equation*}
 D\Txx(w) :  H^1_{T_p\WW^u_{\X_K}(x^-)}(Z,\Rn)\lo  L^2(Z,\Rn)  
\end{equation*}
from \eqref{Dhyb} with $w= u-c_{x^+} \in  H^1_{\WW^u_{\X_K}(x^-)}(Z,\Tn)$, $p = u(0,\cdot) \in \WW^u_{\X_K}(x^-)$ and $c_{x^+}$ as in \eqref{cx}, 
is onto and the vector field $\X_K = -\Ah +K$ satisfies the Morse-Smale condition up to order $2$. 
\end{thm}
\begin{proof} We denote by $\mathcal{C}_{x^-}\big((-\infty,0],\M \big)$ the Banach manifold of all $C^1$-curves $v$ satisfying  $\lim \limits_{s \rlo -\infty}v(s) = x^-$,
$\lim \limits_{s \rlo -\infty}v^{\prime}(s) = 0$, and set 
\begin{align*}
M &=\mathcal{C}_{x^-}\big((-\infty,0],\M \big)\times H^1(Z,\Rn)\times\Kc_{\theta,r} \,, \quad  N= \Kc_{\theta,r} \,,\\
O &= C^0_0((-\infty,0],\Ee)\times \LZR \times \M  
\end{align*}
and consider the map $\sigma : M\lo  O $
$$\sigma(v,w,K) =\Big(v^{\prime}-\X_K(v),\dH(w+c_{x^+}), v(0)- ([0] + w(0,\cdot))\Big)\,.$$
{\bf Step 1 :} We claim that $0$ is a regular value of $\sigma$. Indeed, the linearization $D\sigma = D\sigma(v,w,K)$ at some point $(v,w,K)\in \sigma^{-1}(0)$ is given by  
$$ D\sigma : C^1_0\big((-\infty,0],\Ee \big)\times H^1(Z,\Rn)\times\Kc_{\theta}  \lo  C^0_0((-\infty,0],\Ee)\times \LZR \times \Ee $$ 
with 
$$ D\sigma(v,w,K)[\eta,\xi,L] =  \Big(\eta^{\prime}-D\X_K(v)\eta + L(v),F_{J_0,H}(w+c_{x^+})\xi, \eta(0)-\xi(0,\cdot)\Big)\,,$$ 
where 
$$F_{J_0,H}(w+c_{x^+})\xi  =  \p_s\xi +J_0\p_t\xi + \nabla^2H(\cdot,w+c_{x^+})\xi\,.$$
We show that $D\sigma(v,w,K)$ is a left inverse. Note that the operator 
$$D_1\sigma(v,w,k) +D_2\sigma(v,w,k)$$ 
coincides with the operator $\overline{D}_{\hyb}$ with $S= \nabla^2H(\cdot,u)$, which, 
due to Lemma \ref{DDhyb} and Theorem \ref{fred2}, is Fredholm of index $m(x^-\V)-\mu(x^+)$ if $r$ is chosen small enough. 
So by Lemma \ref{cker} $D\sigma(v,w,k)$ has complemented kernel and it remains to show that $D\sigma(v,w,k)$ is onto. We set 
$W = W^u_A\,, A = D\X_K(v)$ 
and use the fact that both sides of \eqref{subco} are equal, as proved in Lemma \ref{DDhyb},  to observe that 
\begin{equation}\label{co1}
\coker\big(D\sigma(v,w,K)\big)  \subset \{0\}\times \coker(\overline{D}_W)\times\{0\} =\coker(\overline{D}_{\hyb})  \,.
\end{equation}
Since $\coker(\overline{D}_W)$ is finite dimensional, we can choose a basis $\{\rho_i\}_{i=1,\dots,N}$
and consider the problem
\begin{equation}\label{prob} \li\{ \begin{array}{rcl} \zeta^{\prime} -A\zeta + L(v) &=& \varrho \\
          			                       \zeta (0)                       &=& \rho_i(0,\cdot)\,, \quad 1\leq i \leq N\,,
         \end{array}\re.
\end{equation}
for given $\varrho\in C^0_0\big((-\infty,0],\Ee\big)$. Now, by Proposition \ref{FA1}, we have that $F_A^-\zeta = \zeta^{\prime} -A\zeta $ is a left inverse  
and therefore we can choose $\tilde \zeta \in C^1_0\big((-\infty,0],\Ee\big)$ such that $\tilde\zeta^{\prime} -A\tilde \zeta = \varrho$. 
Assume that $(v,w,K)\in \sigma^{-1}(0)$ and $v$ is a constant flow line of $\X_K$, then $u(0,\cdot)=x^-$ holds, i.e., $u$ runs into a critical point of $A_H$ in 
finite time. Though it seems obvious, it is a  nontrivial fact that $u$ is constant in this case, see \cite{hof} (section 6.4). By assumption we have that $x^-\not =x^+$ and 
therefore this  case cannot occur. Thus  $v$ is a non-constant flow line, i.e., a $C^{1}$-embedding and  $v^{-1} : v(\R) \subset \M \lo (-\infty,0]$ is a compact map.
Accordingly, for each $\zeta_i \in C^1_0\big((-\infty,0],\Ee\big)$, that possesses compact support and satisfies $\zeta_i(0) = \rho_i(0,\cdot)-\tilde \zeta(0)$, we can choose 
$L_i \in  \Kc_{\theta}$ such that 
$$ \big(\zeta_i^{\prime}- A\zeta_i\big)\circ \big(v^{-1}\big)(p) + L_i(p)=0 \,, \quad \forall\, p \in v(\R)\subset \M\,.  $$
Hence we can solve \eqref{prob} and therefore \eqref{co1} actually becomes
$$ \coker \big(D\sigma(v,w,K)\big)  \subset \{0\}\times \coker(\overline{D}_{\Ee})\times\{0\}$$
with 
\begin{align*}
 \coker(\overline{D}_{\Ee}) = \Big\{ \rho \in \WZR \mi \p_s\rho - J_0\p_t\rho - \nabla^2H(\cdot,u)\rho & = 0 \,,\\
                                                                                                        \rho(0,\cdot) &=0\Big\}\,.
\end{align*}
Since $\rho(0,\cdot) =0$ is obviously a totally real boundary condition, we can use Schwarz-reflection
and a standard elliptic regularity argument, see \cite {duff}, to obtain that all $\rho \in  \coker(\overline{D}_{\Ee})$ are smooth. 
Furthermore, we can apply the Carleman similarity principle, see \cite{hof}, to the reflected curves which therefore all vanish identically on an open 
neighborhood $U$ of $\{0\}\times\mS^1 \subset \R\times \mS^1$ implying that $\rho \equiv 0$ holds for all $\rho \in \coker(\overline{D}_{\Ee})$ on some  
open neighborhood $V$ of $ \{0\}\times\mS^1  \subset Z$. Thus for $\rho \in \coker(\overline{D}_{\Ee})$ the set 
\begin{align*}
 \Sigma = \Big\{(s,t) \in Z \mi \rho(s,t)=0 \, \text{and } \, \exists \, (s_m,t_m) \rlo (s,t) \,, (s_m,t_m) \not = (s,t) \\
                                                                 \text{satisfying} \,\, \rho(s_m,t_m)=0\Big\}  \,. 
\end{align*}
is closed and by our previous discussion, non-empty. If $(s_0,t_0) \in \Sigma$ we can  apply the previous argument once again and deduce that $(s_0,t_0)$ 
is an interior point of $\Sigma$. Hence $\Sigma = Z$ and consequently $\coker\big(D\sigma(v,w,k)\big) = \{0\}$ for all $(v,w,k) \in \sigma^{-1}(0)$, proving Step 1.\\

{\bf Step 2 :}  By step 1 we have that  $\mathcal{M}:=\sigma^{-1}(0)$ is a Banach manifold.  Consider the projector
$$\tau : \mathcal{C}^1_{x^-}\big((-\infty,0],\M \big)\times H^1_{x^+}(Z,\Tn)\times\Kc_{\theta,r}  \lo \Kc_{\theta,r}  \,.$$ 
In analogy to the proof of Theorem \ref{SM} we claim that the restriction of $\tau$ to $\sigma^{-1}(0)$ is Fredholm of index  
$m(x^-,\R^n\times\Hm^+) - \mu(x^+)$. Everything follows from Corollary \ref{psii} with 
$M,N,O$ as defined above and $\psi = \sigma$, $\phi = \tau $. Denote by $\widehat\Kc_{\reg}(x^-,x^+) \subset \Kc_{\theta,r}$ the regular values of $\tau$, then again by 
Corollary \ref{psii} this is the set such that 
$$D_1\sigma(v,w,K) + D_2\sigma(v,w,K) $$ 
is a left inverse for all 
$(v,w,K)\in \pi^{-1}\big(\widehat\Kc_{\reg}(x^-,x^+)\big)$ and as a consequence of Lemma \ref{DDhyb} so is the operator $D\Txx(w)$ for small $r$ and all $w \in \Txx^{-1}(0)$. 
Now, by a similar argument as in Lemma \ref{sigprop}, combining the results of section \ref{hybcompactness}, one can show that the map $\tau_{|\mathcal{M}}$ is $\sigma$ - proper. 
Hence the Sard-Smale Theorem \ref{sard} applies and 
tells us that  $\widehat\Kc_{\reg}(x^-,x^+) $ is residual in $\Kc_{\theta,r}$. Thus so is the set 
$$\widehat \Kc_{\reg}:= \bigcap_{\genfrac{}{}{0pt}{}{x^-,x^+ \in \Pc_0(H)\,,}{0<m(x^-,\V)-\mu(x^+)\leq 2}}\widehat\Kc_{\reg}(x^-,x^+) \cap \Kc_{\reg} \,,$$
where $\Kc_{\reg}$ is the residual set of compact vector fields $K$ from Theorem \ref{SM}, such that $\X_K$ satisfies the Morse-Smale condition up to order 2. 
Hence  $\widehat \Kc_{\reg}$ has the required properties and we conclude the proof.
\end{proof}

So far we have not treated the case of $x^- = x^+$. This is done next.  

\begin{lem}\label{zero} Assume we have chosen a regular pair $(H,J_0)$ and  $K\in \widehat \Kc_{\reg}$. Let $x \in \Pc_0(H)$  then the moduli space $\MM_{\hyb}(x,x,H,J_0,\X_K)$ 
is a zero dimensional manifold which consists of the constant solution only. Moreover, the linearized operator 
\begin{equation*}
 D\Txx(w) :  H^1_{T_x\WW^u_{\X_K}(x)}(Z,\Rn)\lo  L^2(Z,\Rn)  
\end{equation*}
from \eqref{Dhyb} with $w= x-c_{x} \in  H^1_{\WW^u_{\X_K}(x)}(Z,\Rn)$ and $c_{x}$ as in \eqref{cx}, 
is an isomorphism. 
\end{lem}
\begin{proof} Let $(v,u)\in C^1_x\big((-\infty,0],\M\big)\times H^1_{\loc}(Z,\Tn)$ be a hybrid type solution connecting $x$ with itself, i.e.,  
$$ v^\prime -\X_K(v) = 0\,, \quad \dH(u) = 0\,, \quad v(0) = u(0,\cdot)\,,$$ 
and
$$ v(-\infty) = x\,,\quad  u(+\infty,\cdot) = x  \,.$$
Since $A_H$ is a Lyapunov function for $\X_K$ we have that 
$A_H(x) \geq A_H\big(v(0)\big)$. On the other hand 
$$ E(u) = \int_0^{\infty}\nl \p_s u(s,\cdot)\nr_{\LS}^2ds = A_H\big(u(0,\cdot)\big) - A_H(x) \geq 0 \,. $$
Hence $A_H\big(v(0)\big) = A_H(x)$, which implies that $v(0) \in \rest(\X_K)$. Thus $v(0)=x$ and 
we have that $v$ is constant. Since $E(u) = 0$, we have that $u$ is constant as well. Hence $\MM_{\hyb}(x,x,H,J_0,\X_K)$ consists only of the constant 
solution and is therefore a submanifold of $H^1_{\loc}(Z,\Tn)$. 
Now, since $DK$ vanishes on $\Pc_0(H)$, the linearization $\overline{D}_{\hyb}:=\overline{D}_{\hyb}(v,u)$ at $(v,u)$ is of the form 
$$ \overline{D}_{\hyb}: C^1_0((-\infty,0],\Ee)\times \WZR   \lo  C^0_0((-\infty,0],\Ee)\times \LZR \times \Ee $$
with 
$$ \overline{D}_{\hyb}(\eta,\xi) = \Big(\eta^{\prime}-(\PP^+-\PP^--j^*S)\eta,\p_s\xi+J_0\p_t \xi  +S\xi, \eta(0)-\xi(0,\cdot)\Big)\,,$$
where $S=\nabla^2H(\cdot,x)$. Then, by Lemma \ref{DDhyb}, this is a Fredholm operator and, by Lemma \ref{indform}, $\overline{D}_{\hyb}(v,u)$ 
is of index $0$. Now let $(\eta,\xi) \in \ker \overline{D}_{\hyb}(v,u)$ and assume that $\eta(0) = p = \xi(0,\cdot)$ is an eigenvector of 
$S$ with respect to the eigenvalue $\lambda \in \R$. Then we have that 
$$\eta(s) = e^{(\PP^+-\PP^- -j^*\lambda I )s}p \,.$$
Using the Fourier representation $\eta(s) = e^{(\PP^+-\PP^- -j^*\lambda I )s} \sum \limits_{k\in \Z}e^{2\pi J_0kt}p_k$ and the fact that $\eta(s) \lo 0 $ for $s \lo -\infty $
we observe that 
$$ k \geq \li\lfloor \frac{\lambda}{2\pi} \re\rfloor+ 1 \,, $$
where $\li\lfloor \cdot \re\rfloor$ denotes the Gau\ss-bracket. On the other hand we have that 
$$ \xi(s,t) = e^{-(J_0\p_t+\lambda I)s}p(t)\,.$$
Using again the Fourier representation $\xi(s,t) = e^{-(J_0\p_t+\lambda I)s}\sum \limits_{k\in \Z}e^{2\pi J_0kt}p_k$ and the fact that $\xi(s,t) \lo 0 $ for $s \lo +\infty $
we observe that 
$$k \leq \li\lfloor \frac{\lambda}{2\pi} \re\rfloor \,.$$
Hence the kernel of $\overline{D}_{\hyb}(v,u)$ is trivial, which shows that $\overline{D}_{\hyb}(v,u)$ is an isomorphism. The result now follows from Lemma \ref{DDhyb}.
\end{proof}
By the implicit function theorem we obtain the following corollary.
\begin{cor}\label{mani} Let the assumptions of Theorem \ref{res} be fulfilled and $K \in \widehat \Kc_{\reg}$. Then, even if $x^- = x^+$, we have that the moduli spaces of 
hybrid type curves $\MM_{\hyb}(x^-,x^+,H,J_0,\X_K)$ are $C^3$-manifolds of dimension 
$$m(x^-,\V) - \mu(x^+)\,.$$ 
\end{cor}

\subsection{Compactness}\label{hybcompactness}

In this section we generalize the already given $H^1_{\loc}$-compactness statement for the moduli spaces of Floer-cylinders to the case 
of moduli spaces of hybrid type curves. The compactness results of chapter \ref{Morse} and \ref{Floer} are used to handle our non-Lagrangian boundary condition. \\

As in Lemma \ref{wbound}  we use an indirect argument to achieve a uniform bound on the curves of constant loops. 
\begin{lem}\label{wbound2}
Let $x^-,x^+ \in \Pc_0(H)$, $K \in  \Kc_{\theta,1}$ and recall that  $\Theta_{x^-,c_{x^+}}$ denotes the non linear operator  
$$  \Theta_{x^-,c_{x^+}} : H^1_{\WW^u_{\X_K}(x^-)}(Z,\Rn) \lo L^2(Z,\Rn)\,, \quad w \mapsto \dH(w+c_{x^+}) \,. $$
Let $w \in\Theta_{x^-,c_{x^+}}^{-1}(0)$ and denote by $\PP^0w(s,\cdot) \in H^1([0,\infty),\Rn)$ the constant part of $w$. 
Then $\PP^0w \in C^0([0,\infty),\Rn)$ and there is a 
constant $C=C(x^-,x^+)>0$ such that
\begin{equation*}
 \nl \PP^0w\nr_{C^0([0,\infty)} \leq C \quad\text{for all}\quad w \in \Theta_{x^-,c_{x^+}}^{-1}(0)\,.
\end{equation*}
\end{lem}
\begin{proof} We choose a sequence  $w_n \in \Theta_{x^-,c_{x^+}}^{-1}(0)$ and denote furthermore  by   
$v_n \in C^{1}_{x^-}((-\infty,0],\M)$ the sequence of half trajectories connecting $x^-$ with $w_n(0,\cdot)$, i.e., 
$$ v_n^{\prime} - \X_K(v_n) = 0 \quad \text{and} \quad v_n(-\infty) = x^-\,, \quad v_n(0)= [0] + w_n(0,\cdot)\,.$$ 
Let $c_{x^-} \in  C^{\infty}_{x^-}((-\infty,0],\M)$ be such that 
\begin{equation*}
 c_{x^-}(s) = \li\{
\begin{array}{ll}
 [0]\,, & \text{for} \quad s \geq -1\\
 x^-\,, & \text{for} \quad s\leq -2 
\end{array}\re.
\end{equation*}
Let $u_n\in C^1_0((-\infty,0],\Ee)$ be the sequence such that $v_n = u_n +c_{x^-}$. Then we set 
\begin{equation*}
 p_n(s) =  \li\{
\begin{array}{ll}
 \PP^0u_n(s) \,, & \text{for} \quad s \leq  0 \\
 \PP^0w_n(s,t) \,, & \text{for} \quad s > 0 
\end{array}\re.\,.
\end{equation*}
Since $p_n(\pm \infty ) = 0$, the $p_n$'s are a sequence of continuous loops $p_n : \R\cup\{\infty\} \lo \Rn$ starting at $0$. Now assume that 
$$ \lambda_n:= \nl p_n \nr_{C^0(\R)} \lo \infty \quad \text{for} \quad n \lo \infty \,.$$ 
We set $\tilde p_n:= \frac{1}{\lambda_n}p_n$ so that $\lim \limits_{s\rightarrow \pm \infty}\tilde p_n(s) = 0$. Thus we have 
a sequence of continuous loops $\tilde p_n : \R\cup \{\infty\} \lo \Rn$ starting at $0$ with $\nl \tilde p_n \nr_{C^0(\R \cup \{\infty\})} = 1 $ for all $n \in \N$. 
Now each $\tilde p_n$
solves the equation 
\begin{align}
 \p_s \tilde p_n &= - \frac{1}{\lambda_n}\PP^0\big(\nabla H(u_n +c_{x^-}) + \p_sc_{x^-}\big)\,, \text{for} \quad s \leq 0 \nonumber \\
 \p_s \tilde p_n &= - \frac{1}{\lambda_n}\PP^0\big(\nabla H(w_n +c_{x^+}) + \p_sc_{x^+}\big)\,, \text{for} \quad s > 0  \label{w02}
\end{align}
in the corresponding sense. 
By a standard bootstrapping argument each $\tilde p_n$ is actually smooth and by \eqref{w02} we obtain
$$ \lim_{n\rightarrow \infty}\nl \p_s \tilde p_n\nr_{C^0(\R\cup\{\infty\})} = 0\,.$$
Due to the Arzel\`a-Ascoli Theorem the sequence $\tilde p_n$ is compact. So there is a subsequence $\tilde p_{n_k}$ which converges uniformly on 
$\R\cup\{\infty\}$ to a constant loop $\tilde p_{\infty}$ with  $\nl \tilde p_{\infty} \nr_{C^0(\R \cup \{\infty\})} = 1 $ and $\tilde p_{\infty}(\infty) = 0$; a
contradiction. Hence the sequence $\lambda_n$ is bounded proving the claim.
\end{proof}
\begin{thm}\label{compact2} Let $x^-,x^+ \in \Pc_0(H)$ and $K \in  \Kc_{\theta,1}$. Then we have that $\MM_{\hyb}(x^-,x^+,H,J_0,\X_K)$, 
with $\X_K = -\Ah +K$, is a  $H^1_{\loc}$-precompact set.
\end{thm}
\begin{proof} 
{\bf Step 1 :} Let $w \in \Theta_{x^-,c_{x^+}}^{-1}(0)$. We show that for any  bounded domain  
$Z_T= [0,T] \times \mS^1 \subset Z$ the restriction 
of the curve $w$ to $Z_T$ is uniformly bounded. 
Set $u = w +c_{x^+}$ then the energy of $u$ is given by 
$$E(u)= A_H\big(u(0,\cdot)\big) - A_H(x^+) \leq A_H(x^-) - A_H(x^+)\,.$$ 
Therefore we compute
\begin{align*}
 \nl \p_s w \nr_{L^2(Z_T)} &\leq \nl \p_s u \nr_{L^2(Z_T)}  + \nl \p_s c_{x^+} \nr_{L^2(Z_T)} \\
				&\leq \sqrt{E(u)} + \nl \p_s c_{x^+} \nr_{L^2(Z_T)} \\
				&\leq \sqrt{A_H(x^-) - A_H(x^+)} + \nl \p_s c_{x^+} \nr_{L^2(Z_T)}\\ 
                                & =:  c_0(x^-,x^+,c_{x^+}) \,.
\end{align*}
and 
\begin{align*}
\nl \p_t w \nr_{L^2(Z_T)} &\leq \nl \p_t u -X_H(u) \nr_{L^2(Z_T)} +  \nl X_H(u) \nr_{L^2(Z_T)} + \nl \p_t c_{x^+} \nr_{L^2(Z_T)}\\ 
                               &\leq \sqrt{E(u)} + T\nl H \nr_{C^1(\mS^1\times\Tn)} + \nl \p_t c_{x^+} \nr_{L^2(Z_T)} \\
                               &\leq \sqrt{A_H(x^-) - A_H(x^+)} + T\nl H \nr_{C^1(\mS^1\times\Tn)} + \nl \p_t c_{x^+} \nr_{L^2(Z_T)}\\
                               &=:c_1(x^-,x^+,T,H,c_{x^+})\,.
\end{align*}
Observe that 
\begin{align*}
\nl (\PP^0)^{\perp} w \nr^2_{L^2(Z_T)} & =  \int_{0}^T \sum_{k \not = 0  }|w_k(s)|^2 ds \leq \int_{0}^T \sum_{k\not = 0 }|k|^2|w_k(s)|^2 ds \\
		                            & =  \nl \p_t w \nr^2_{L^2(Z_T)} \leq c_1^2 \,,
\end{align*}
where $\PP^0$ denotes the restriction to the $0$-order coefficients. This shows that the $t$-derivative already bounds the non-constant part of $w$. 
Due to Lemma \ref{wbound2} we get the estimate $|w_0(0)| \leq C$ with $C=C(x^-,x^+)> 0$. Thus by writing 
$$ w_0(s) = \int_{0}^s \p_{\tau} w_0(\tau)d\tau + w_0(0) $$
and by H\"older's inequality, we obtain that 
\begin{align*}
 \nl w_0 \nr_{L^2(Z_T)}^2 &=       \int_{0}^{T}\int_{0}^{1}\Big| \int_0^s \p_{\tau} w_0(\tau)d\tau + w(0)\Big|^2dtds \\
                               &\leq   2\int_{0}^{T}\li[|w(0)|^2 + \li( \int_0^s |\p_{\tau} w_0(\tau)|d\tau\re)^2\re]ds \\
                               &\leq   2TC^2 + 2T\int_{0}^{T} \int_0^s|\p_{\tau} w_0(\tau)|^2d\tau ds \\
                               &\leq   2TC^2 + 2T\int_{0}^{T}\nl \p_s w \nr^2_{L^2(Z_T)}ds \\
                               &\leq   2TC^2 + 2T^2c_0^2   
\end{align*}
Hence there is a constant $c_2= c_2(x^{\pm},T,H,c_{x^+},n)$ bounding $\nl w \nr_{L^2(Z_T)}$ and we are done with step 1.\\

{\bf Step 2 :} Let $T>T^{\prime}>0$, $w_1,w_2 \in \Theta_{x^-,c_{x^+}}^{-1}(0)$. Then set $\delta w = w_1-w_2$. We claim
that there is a constant $C=C(T^{\prime},T,H)$ such that
\begin{equation}\label{est0}
\nl \delta w\nr_{H^1(Z_{T^{\prime}},\Rn)} \leq C \Big( \nl \delta w\nr_{L^2(Z_T,\Rn)} + \nl \PP^-\delta w(0,\cdot)\nr_{\HS}\Big) \,. 
\end{equation}
Assume that \eqref{est0} holds, then by step 1 we have that the restriction of any sequence 
$(w_n)_{n\in\N} \subset \Theta_{x^-,c_{x^+}}^{-1}(0)$ to $Z_{T}$ is uniformly bounded 
in $H^1(Z_{T},\Rn)$. Since the embedding 
$$H^1_{\WW^u(x^-)}(Z_{T},\Rn) \hookrightarrow L^2(Z_T,\Rn)$$ 
is compact due to the Sobolev embedding Theorem, see \cite{adams-2}, and 
due to Theorem \ref{compactM} the set $\PP^-\Big(\WW^u(x^-)\cap \big\{A_H \geq A_H(x^+)\big\}\Big)$ is precompact in $\M$. The estimate \eqref{est0} implies that 
the restriction of $(w_n)_{n\in\N}$ to $Z_{T^{\prime}}$ possesses an in 
$H^1(Z_{T^{\prime}},\Rn)$ convergent subsequence which proves the  claimed result.     
So it remains to prove \eqref{est0}. Therefore let $\beta$ be a smooth function such that 
$$ 
\beta(s) = \li\{
\begin{array}{crl}
 1 &,& s \leq T^{\prime} \\
 0 &,& s \geq T^{\prime} + (T-T^{\prime})/2
\end{array}\re.\,.
$$
Now we apply  formula \eqref{delbarform} of Lemma \ref{formform} to $\beta \delta w= \beta(w_1-w_2)$, $w_1,w_2 \in \Theta_{x^-,c_{x^+}}^{-1}(0)$ and obtain  
\begin{align*}
\nl \nabla  (\beta \delta w)  \nr^2_{L^2(Z_{T})} &=    \nl \overline{\p}_{J_0}(\beta \delta w) \nr^2_{L^2(Z_{T})} \\ 
                                                      &+    \nl \PP^-\delta w(0,\cdot)\nr^2_{\HS} - \nl \PP^+\delta w(0,\cdot)\nr^2_{\HS}\\
                                                      &\leq  c\Big(\nl \overline{\p}_{J_0}\delta w \nr^2_{L^2(Z_{T})} + \nl \delta w \nr^2_{L^2(Z_T)} \\ 
                                                      & \qquad\qquad\qquad\qquad\quad \,  +  \nl \PP^-\delta w(0,\cdot)\nr^2_{\HS}\Big)   
\end{align*}
with $c = 2(1+\nl \beta^{\prime}\nr^2_{C^0(\R)})$. Now by the mean value theorem we compute
\begin{align*}
\nl \overline{\p}_{J_0}\delta w \nr_{L^2(Z_{T})} & =   \nl \overline{\p}_{J_0}u_1 -  \overline{\p}_{J_0}u_2  \nr_{L^2(Z_{T})} 
                                                        =   \nl X_H(u_1) -  X_H(u_2)  \nr_{L^2(Z_{T})} \\
                                                      &\leq \nl H\nr_{C^2(\mS^1\times \Tn)} \nl \delta w\nr_{L^2(Z_{T})} \,.
\end{align*}
Therefore we can find a constant $C^{\prime} = C^{\prime}(T^{\prime},T,H)$ such that     
\begin{align*}
\nl \nabla  (\beta \delta w)  \nr^2_{L^2(Z_{T})} & = \nl \p_s  (\beta \delta w)  \nr^2_{L^2(Z_{T})}  + \nl \p_t  (\beta \delta w)  \nr^2_{L^2(Z_{T})} \\
                                                      &\leq C^{\prime}\Big(   \nl \delta w\nr^2_{L^2(Z_{T})} +  \nl \PP^-\delta w(0,\cdot)\nr^2_{\HS} \Big)   \,,
\end{align*}
which implies that there is $C = C(T^{\prime},T,H)$ such that \eqref{est0} holds. 
\end{proof}
Combining the results of Propositions \ref{comp} and \ref{braking} we obtain the following statement. 
\begin{prop} Let $H\in \Hc_{\reg}$ be a nondegenerate Hamiltonian and $K\in \Kc_{\theta, 1}$. 
Let $x^-,x^+ \in \Pc_0(H)$ and let $(u_n)_{n\in\N} \in  \MM_{\hyb}(x^-,x^+,H,J_0,\X_K)$ be a sequence of curves and $v_n = \varphi_{\X_K}\big(s,u_n(0.\cdot)\big)$, 
$s<0$ be the trajectories through $u_n(0,\cdot) \in \M$.
Then there exist $x_0,\dots,x_{r_1},y_0,\dots,y_{r_2} \in \Pc_0(H) $ with $r_1,r_2 \in\N$ and 
$$A_H(x_0)> A_H(x_1)> \dots >A_H(x_{r_1})>A_H(y_0)>A_H(y_1)>\dots >A_H(y_{r_2})$$
connecting trajectories 
$$V_1 \subset \WW^u(x_0)\cap\WW^s(x_1),\dots,V_{r_1} \subset  \WW^u(x_{r_1-1})\cap\WW^s(x_{r_1})\,,$$
and curves
\begin{align*}
U_1 &\in \MM_{\hyb}(x_{r_1},y_0,H,J_0,\X_K)\,,\\
U_2 &\in \MM_{F}(y_0,y_1,H,J_0),\dots, U_{r_2} \in \MM_{F}(y_{r_2-1},y_{r_2},H,J_0) 
\end{align*}
such that a subsequence $(v_{n_k},u_{n_k})$ converges to $(V_1,\dots,V_{r_1},U_1,\dots,U_{r_2})$ in the following sense. There exist reparametrization times
$(\tau_k^j)_{k\in\N} \subset (-\infty,0]$, $j\in \{1,\dots,r_1\}$ and $(s_k^j)_{k\in\N} \subset [0,\infty), j\in \{1,\dots,r_2\}$ such that 
$$ \varphi_{\X_K}(\tau_k^1,v_{n_k}) \lo V_1,\dots ,\varphi_{\X_K}(\tau_k^{r_1}, v_{n_k}) \lo V_{r_1} \quad \text{in the Hausdorff-distance} $$
and 
$$u_{n_k} \lo U_1, u_{n_k}(\cdot+s_k^1,\cdot)\lo U_2, \dots , u_{n_k}(\cdot+s_k^{r_2},\cdot)\lo U_{r_2} \quad \text{in} \quad H^1_{\loc} \,.  $$    
\end{prop}
We want to remark that, indeed there is no braking at the boundary. This would result in a curve which runs into a critical point in finite time which is therefore 
constant, see \cite{hof}. Of course also the Morse trajectory would be constant in this case so the whole curve would converge to some critical point which contradicts 
the assumption that $A_H$ possesses regular values in the interior. 

\subsection{The isomorphism in $\Z_2$- coefficients}
 
By the results of section \ref{transF} we can choose a generic Hamiltonian $H$ such that $(H,J_0)$ is a regular pair. Furthermore due to Theorem \ref{res} we can choose 
$K\in \widehat \Kc_{\reg}$ such that $\X_K$ satisfies the Morse-Smale property up to order 2, i.e., the unstable and stable manifolds meet transversally 
and the moduli spaces $\MM_F(x^-,x^+,H,J_0)$ and $\MM_{\hyb}(x^-,x^+,H,J_0,\X_K)$ are manifolds for all $x^{\pm} \in \Pc_0(H)$ with $\mu(x^-)-\mu(x^+) \leq 2$. 
In this section we show that, in this situation, there is a chain isomorphism 
$$\Phi_* :\big(C_*(H),\p^M_*\big) \lo \big(C_*(H),\p^F_*\big)\,, \quad  i.e., \Phi_{*-1}\circ \p^M_* = \p^F_*\circ \Phi_{*} \,.$$ 
For technical reasons we reduce ourselves to the case of $\Z_2$-coefficients, i.e., 
$$ C_k(H) = \bigoplus_{\genfrac{}{}{0pt}{}{x \in \Pc_0(H)\,,}{\mu(x)=k }}\Z_2 x \,.$$ 
The first part is to prove that $\Phi_*$ is a chain homomorphism. This fact is an immediate consequence of the following result. 
  
\begin{prop}\label{propf} Assume we have chosen a regular pair $(H,J_0)$ and  $K\in \widehat \Kc_{\reg}$. Let $x^-,x^+ \in \Pc_0(H)$
with $\mu(x^-) -\mu(x^+) =1$. Then the moduli space $\overline{\MM_{\hyb}(x^-,x^+,H,J_0,\X_K)}$ is a $H^1_{\loc}$- compact manifold with boundary of dimension 1. 
Its boundary consists of all the broken hybrid type trajectories from $x^-$ to $x^+$.
\end{prop}
\begin{proof} By Corollary \ref{mani} $\MM_{\hyb}(x^-,x^+,H,J_0,\X_K)$ is a manifold of dimension 1, which  
is $H^1_{\loc}$- precompact according to Theorem \ref{compact2}. We claim that if $\p \MM_{\hyb}(x^-,x^+,H,J_0,\X_K)$ is non-empty, then 
$$W\in \overline{\MM_{\hyb}(x^-,x^+,H,J_0,\X_K)}\setminus \MM_{\hyb}(x^-,x^+,H,J_0,\X_K)$$ 
is a broken hybrid type trajectory, that is of the following form. Either there is $y\in \Pc_0(H)$ with $\mu(y)=\mu(x^+)$ and $W= (V,U)$, where $V$ is a connecting 
trajectory from $x^-$ to $y$ and $U \in \MM_{\hyb}(y,x^+,H,J_0,\X_K)$ or $\mu(y)=\mu(x^-)$ and $W=(V^{\prime},U^{\prime})$, 
where $V^{\prime} \in \MM_{\hyb}(x^-,y,H,J_0,\X_K)$ and $U^{\prime} \in \MM_{F}(y,x^+,H,J_0)$. Hence the ends of $\overline{\MM_{\hyb}(x^-,x^+,H,J_0,\X_K)}$ 
are half open intervals whose boundary points correspond to the broken hybrid type trajectories. To prove this result one uses the gluing methods established in 
\cite{Al6} and \cite{schwarz}, both of which localize either at the Morse part or the Floer part and can therefore be applied. 
\end{proof}

Now we construct the isomorphism as follows : \\

Let $H \in \Hr$ be given, then  by the results of section \ref{transF} we can assume that, up to a small perturbation, which leaves the periodic orbits fixed, $H$ is such that 
$(H,J_0)$ is a regular pair. We choose  $K\in \widehat \Kc_{\reg}$ and set for $\V = \M\times(\R^n\times\Hm^+)$ 
$$\upsilon(x,y) =    \# \MM_{\hyb}(x,y,H,J_0,\X_K) \mod 2 \,, \quad m(x,\V) = \mu(y)$$
to be the number of connected components of the corresponding moduli spaces. Due to Theorem \ref{compact2}, Theorem \ref{res} and Lemma \ref{zero} the above identity is well defined. 
Observe furthermore that since 
the solutions 
$$(v,u) \in \MM_{\hyb}(x,y,H,J_0,\X_K)$$
are half Morse trajectories and half pseudo-holomorphic curves due to the Carleman similarity 
principle, see \cite{hof}, each solution is isolated and is therefore a connected component.
Consider the abelian groups 
$$ C_k(H) = \bigoplus_{\genfrac{}{}{0pt}{}{x \in \Pc_0(H)\,,}{\mu(x)=k }}\Z_2 x \,. $$  
Then we define the following map $\Phi_k :\big(C_k(H),\p^M_k\big) \lo \big(C_k(H),\p^F_k\big)$ between the Morse- and Floer-complex on the generators as 
\begin{equation}\label{iso}
 \Phi_k(x) =\sum_{\genfrac{}{}{0pt}{}{y \in \Pc_0(H)\,,}{\mu(x)=\mu(y) }}\upsilon(x,y)y
\end{equation}
and state the following Theorem. 
\begin{thmmm}[Main Theorem] The Morse-complex of the Hamiltonian action $A_H$ is chain-isomorphic to the Floer-complex of $(H,J_0)$. 
\end{thmmm}
\begin{proof} First we show that $\Phi$ is an homomorphism. For this consider the following identity 
\begin{equation}\label{numb}
 \sum_{\genfrac{}{}{0pt}{}{z \in \Pc_0(H)\,,}{\mu(y)=\mu(z) }}\sum_{\genfrac{}{}{0pt}{}{y \in \Pc_0(H)\,,}{\mu(x)-\mu(y)=1 }}\upsilon(y,z)\rho(x,y)z 
 + 
 \sum_{\genfrac{}{}{0pt}{}{z \in \Pc_0(H)\,,}{\mu(y)-\mu(z)=1 }}\sum_{\genfrac{}{}{0pt}{}{y \in \Pc_0(H)\,,}{\mu(x)=\mu(y) }}\nu(y,z)\upsilon(x,y)z =0 
\end{equation}
modulo $2$, where $\rho(x,y)$ is the number of connected components of the intersection $\WW^u(x)\cap\WW^s(y)$ with respect to $\X_K$  
and $\nu(y,z)$  is the number of connected components of $\widehat{\MM_F}(x,y,H,J_0)$. Indeed the coefficient in \eqref{numb} is the number of broken 
hybrid type trajectories from $x$ to $z$, which by Proposition  \ref{propf} equals the number of boundary points of $\MM_{\hyb}(x,z,H,J_0,\X_K)$. 
Since the boundary of a compact one-dimensional manifold always consists of an even number of points,\eqref{numb} holds. In particular there holds 
$$\sum_{\genfrac{}{}{0pt}{}{z \in \Pc_0(H)\,,}{\mu(y)=\mu(z) }}\sum_{\genfrac{}{}{0pt}{}{y \in \Pc_0(H)\,,}{\mu(x)-\mu(y)=1 }}\upsilon(y,z)\rho(x,y)z 
\quad=
 \sum_{\genfrac{}{}{0pt}{}{z \in \Pc_0(H)\,,}{\mu(y)-\mu(z)=1 }}\sum_{\genfrac{}{}{0pt}{}{y \in \Pc_0(H)\,,}{\mu(x)=\mu(y) }}\nu(y,z)\upsilon(x,y)z   $$
modulo $2$. Hence $$ \Phi_{*-1}\circ \p^M_* = \p^F_*\circ \Phi_{*} $$ and therefore $\Phi$ is an homomorphism as claimed. Now we order 
the critical points $x_1,\dots,x_n\in \Pc_0(H)$, $n\in\N$ with $\mu(x_i) =k$, $i =1,\dots,n$ by increasing action, choosing any order for subsets of solutions 
with identical action, i.e., $ A_H(x_1) \leq A_H(x_2) \leq \dots \leq A_H(x_{n-1})\leq A_H(x_n)$. Then, by Lemma \ref{zero}, $\Phi_k$ is of the form 
\begin{equation*}
 \Phi_k = \li(\begin{array}{cccc}
               1 &     *   &  * \\
               0 & \ddots & *    \\
               0 & 0      & 1                
              \end{array}\re)
\end{equation*}
where $* = 1$ or $*=0$. Hence $\Phi$ is an isomorphism.
\end{proof}

{\bf Acknowledgement} I thank the unknown referee and Alberto Abbondandolo for providing me with the improvements. Furthermore I want to thank Matthias Schwarz and Felix Schm\"aschke 
for many stimulating discussions.
\bibliographystyle
%{alpha}
%{amsplain}
%{amsalpha}
{plain}

% ------------------------------------------------------------------------
\end{document}